\documentclass[11pt]{amsart}

\newtheorem{thm}{Theorem}[section]
\newtheorem*{thm*}{Theorem}

\newtheorem{claim}[thm]{Claim}

\newtheorem{cor}[thm]{Corollary}

\newtheorem{lem}[thm]{Lemma}
\newtheorem*{lem*}{Lemma}
\newtheorem{mainthm}{Theorem}
\newtheorem*{mainthm*}{Theorem}
\newtheorem{maincor}[mainthm]{Corollary}

\newtheorem{prop}[thm]{Proposition}

\theoremstyle{definition}

\newtheorem*{case*}{Case}

\newtheorem{data}[thm]{Data}
\newtheorem{defn}[thm]{Definition}
\newtheorem*{defn*}{Definition}
\newtheorem{exmp}[thm]{Example}
\newtheorem*{exmp*}{Example}

\newtheorem{step}{Step}\renewcommand{\thestep}{}

\theoremstyle{remark}

\renewcommand{\thecase}{}

\newtheorem{rmk}[thm]{Remark}
\newtheorem*{rmk*}{Remark}

\makeatletter
\def\alphenumi{
  \def\theenumi{\alph{enumi}}
  \def\p@enumi{\theenumi}
  \def\labelenumi{(\@alph\c@enumi)}}
\makeatother

\makeatletter
\def\thecase{\@arabic\c@case}
\makeatother

\makeatletter
\def\thestep{\@arabic\c@step}
\makeatother

\newcount\hh
\newcount\mm
\mm=\time
\hh=\time
\divide\hh by 60
\divide\mm by 60
\multiply\mm by 60
\mm=-\mm
\advance\mm by \time
\def\hhmm{\number\hh:\ifnum\mm<10{}0\fi\number\mm}

\let\oldmarginpar\marginpar
\renewcommand\marginpar[1]{\-\oldmarginpar[\raggedleft\footnotesize #1]%
{\raggedright\footnotesize #1}}

\renewcommand\emptyset{\varnothing}

\newcommand\AAA{\mathbb{A}}

\newcommand\CC{\mathbb{C}}

\newcommand\HH{\mathbb{H}}

\newcommand\NN{\mathbb{N}}

\newcommand\PP{\mathbb{P}}

\newcommand\RR{\mathbb{R}}

\newcommand\VV{\mathbb{V}}

\newcommand\cB{{\mathcal{B}}}

\newcommand\cF{{\mathcal{F}}}

\newcommand\cR{{\mathcal{R}}}
\newcommand\cS{{\mathcal{S}}}

\newcommand\cU{{\mathcal{U}}}

\newcommand\fg{{\mathfrak{g}}}

\newcommand\fS{{\mathfrak{S}}}

\newcommand\fu{{\mathfrak{u}}}
\newcommand\fU{{\mathfrak{U}}}

\newcommand\fV{{\mathfrak{V}}}

\newcommand\sA{{\mathscr{A}}}
\newcommand\sB{{\mathscr{B}}}

\newcommand\sG{{\mathscr{G}}}
\newcommand\sH{{\mathscr{H}}}
\newcommand\sI{{\mathscr{I}}}

\newcommand\sL{{\mathscr{L}}}

\newcommand\sN{{\mathscr{N}}}
\newcommand\sO{{\mathscr{O}}}

\newcommand\sS{{\mathscr{S}}}

\newcommand\sU{{\mathscr{U}}}

\newcommand\sX{{\mathscr{X}}}

\newcommand\bgamma{{\boldsymbol{\gamma}}}
\newcommand\bchi{{\boldsymbol{\chi}}}

\newcommand\bpsi{{\boldsymbol{\psi}}}

\newcommand\bB{{\mathbf{B}}}

\newcommand\bC{{\mathbf{C}}}

\newcommand\bH{{\mathbf{H}}}

\newcommand\bO{{\mathbf{O}}}

\newcommand\bs{{\mathbf{s}}}

\newcommand\bzero{{\mathbf{0}}}

\newcommand\eps{\varepsilon}

\newcommand\GL{\operatorname{GL}}
\newcommand\Or{\operatorname{O}}

\newcommand\SO{\operatorname{SO}}

\DeclareMathOperator{\Sp}{Sp}

\newcommand\SU{\operatorname{SU}}
\newcommand\U{\operatorname{U}}

\newcommand\less{\setminus}

\newcommand\ad{{\operatorname{ad}}}
\newcommand\Ad{{\operatorname{Ad}}}

\newcommand\Aut{\operatorname{Aut}}

\newcommand\Center{\operatorname{Center}}

\newcommand\Conf{\operatorname{Conf}}

\newcommand\codim{\operatorname{codim}}

\newcommand\Coker{\operatorname{Coker}}

\newcommand\dist{\operatorname{dist}}

\newcommand{\esssup}{\operatornamewithlimits{ess\ sup}}

\newcommand\Fr{\operatorname{Fr}}

\newcommand\Gl{\operatorname{Gl}}

\newcommand\Hol{\operatorname{Hol}}

\newcommand\ind{\operatorname{ind}}

\newcommand\Ind{\operatorname{Index}}
\DeclareMathOperator{\Inj}{Inj}

\DeclareMathOperator{\Int}{Int}
\newcommand\Imag{\operatorname{Im}}

\newcommand\Isom{\operatorname{Isom}}

\newcommand\Ker{\operatorname{Ker}}

\DeclareMathOperator{\pt}{pt}

\newcommand\Ran{\operatorname{Ran}}

\newcommand\Riem{\operatorname{Riem}}

\newcommand\Scale{\operatorname{Scale}}

\newcommand\Span{\operatorname{Span}}

\newcommand\Stab{\operatorname{Stab}}

\newcommand\supp{\operatorname{supp}}

\newcommand\Sym{\operatorname{Sym}}

\newcommand\tr{\operatorname{tr}}

\newcommand\vol{\operatorname{vol}}

\newcommand\euclid{{\mathrm{euclid}}}

\newcommand\id{{\mathrm{id}}}

\newcommand\loc{{\mathrm{loc}}}

\newcommand\mutatis{{\emph{mutatis mutandis }}}

\newcommand\round{{\mathrm{round}}}

\newcommand\sym{{\mathrm{sym}}}

\numberwithin{equation}{section}

\usepackage{amssymb}
\usepackage{amscd}
\usepackage{bbm}
\usepackage{chemarr}
\usepackage{chemarrow}
\usepackage{graphicx}
\usepackage{hyperref}
\usepackage{mathrsfs}
\usepackage{mathtools}
\usepackage[usenames]{color}
\usepackage{paralist}
\usepackage{slashed}
\usepackage{tikz-cd}
\usepackage{url}
\usepackage{verbatim}
\usepackage{xr}
\usepackage[all,cmtip]{xy}

\newcommand{\transv}{\mathrel{\text{\tpitchfork}}}
\makeatletter
\newcommand{\tpitchfork}{%
  \vbox{
    \baselineskip\z@skip
    \lineskip-.52ex
    \lineskiplimit\maxdimen
    \m@th
    \ialign{##\crcr\hidewidth\smash{$-$}\hidewidth\crcr$\pitchfork$\crcr}
  }%
}
\makeatother

\hypersetup{pdftitle={Gluing in geometric analysis via maps of Banach manifolds with corners and applications to gauge theory}}
\hypersetup{pdfauthor={Paul M. N. Feehan and Thomas G. Leness}}

\begin{document}

\title[Gluing via maps of Banach manifolds with corners]{Gluing in geometric analysis via maps of Banach manifolds with corners and applications to gauge theory}

\author[Paul M. N. Feehan]{Paul M. N. Feehan}
\address{Department of Mathematics, Rutgers, The State University of New Jersey, 110 Frelinghuysen Road, Piscataway, NJ 08854-8019, United States}
\email{feehan@math.rutgers.edu}
\urladdr{math.rutgers.edu/$\sim$feehan}
\author[Thomas G. Leness]{Thomas G. Leness}
\address{Department of Mathematics, Florida International University, Miami, FL 33199, United States}
\email{lenesst@fiu.edu}
\urladdr{fiu.edu/$\sim$lenesst}

\date{\today{ }\hhmm}

\begin{abstract}
We describe a new approach to the problem of constructing gluing parameterizations for open neighborhoods of boundary points of moduli spaces of anti-self-dual connections over closed four-dimensional manifolds. Our approach employs general results from differential topology for $C^1$ maps of smooth Banach manifolds with corners, providing a method that should apply to other problems in geometric analysis involving the gluing construction of solutions to nonlinear partial differential equations. 
\end{abstract}




\thanks{Paul Feehan was partially supported by National Science Foundation grant DMS-1510064 and Thomas Leness was partially supported by National Science Foundation grant DMS-1510063.}

\maketitle
\tableofcontents

\section{Introduction}
\label{sec:Introduction}
In this article, we develop a new approach to the problem of constructing gluing parameterizations for open neighborhoods of boundary points of moduli spaces of anti-self-dual connections over closed four-dimensional manifolds, building on prior approaches pioneered by Taubes \cite{JaffeTaubes, TauSelfDual, TauIndef, TauFrame}) and Donaldson \cite{DonConn, DK}. Our method employs general results from differential topology for $C^1$ maps of smooth Banach manifolds with corners, providing a package that should apply to other problems in geometric analysis involving the gluing construction of solutions to nonlinear partial differential equations. Gluing problems that may be amenable to this approach are listed in Section \ref{subsec:Other_gluing_scenarios}. The case of smooth Banach manifolds with boundary is sufficient for the application in this article to gluing families of anti-self-dual connections on smooth principal $G$-bundles $P_0$ and $P_1$ over four-dimensional, oriented, smooth, Riemannian manifolds $(X_0,g_0)$ and $(X_1,g_1)$, where $G$ is compact Lie group.

As in previous work of Taubes and Donaldson, we employ a \emph{splicing map for connections}. However, we do \emph{not} use splicing in the traditional way of first splicing families of anti-self-dual connections to construct approximate anti-self-dual connections and then solving an elliptic, quasilinear, second-order partial differential equation to obtain a family of exact solutions to the  anti-self-dual equation. Rather, we use splicing to construct a smooth \emph{surjective submersion} $\cS$ from products of Banach affine spaces of Sobolev connections on $P_0$ and $P_1$ and a finite-dimensional smooth manifold of auxiliary splicing data onto the Banach affine space of all Sobolev connections on the smooth, connected-sum principal $G$-bundle $P=P_0\#P_1$ over the smooth Riemannian connected sum $(X,g) = (X_0\# X_1, g_0\# g_1)$, where the width of the neck is controlled by a small positive scale parameter $\lambda$. The composition $F^+\circ\cS$ of the splicing map $\cS$ and the self-dual curvature map $F^+$ from the Banach affine space of all Sobolev connections on $P$ into the Banach space of $\ad P$-valued self-dual two-forms on $X$ is a smooth map. As in \cite{TauSelfDual, TauIndef}, the finite-dimensional manifold of auxiliary splicing data comprises open neighborhoods of points in $X_0$ and $X_1$ used to define the connected sum $(X,g)$, a space of principal $G$-bundle gluing parameters that is isomorphic to a copy of $G$, and an interval $(0,\lambda_0)$ of \emph{scale} (or \emph{neck}) parameters, where $\lambda_0\in(0,1]$ is a small constant. When $\lambda=0$, the neck is fully pinched, the four-manifolds $X_0$ and $X_1$ are joined at a single common point, and the splicing map $\cS$ restricts to the identity on the face of a smooth Banach manifold with boundary comprising products of Banach affine spaces of Sobolev connections on $P_0$ and $P_0$ and the finite-dimensional manifold of auxiliary splicing data containing the factor $[0,\lambda_0)$. The composition $F^+\circ\cS$ extends to a $C^1$ map of smooth Banach manifolds with boundary and vanishes transversely (in the sense of Definition \ref{defn:Transversality_maps_Banach_manifolds_boundary}) at boundary points (where $\lambda=0$) corresponding to pairs of regular anti-self-dual connections (where there are no cokernel obstructions) on $P_0$ and $P_1$, respectively. We then appeal to an abstract result (Theorem \ref{mainthm:Preimage_submanifold_under_transverse_map_and_implied_embedding}) from differential topology for Banach manifolds with boundary (and which extends to Banach manifolds with corners more generally) to give a \emph{gluing map} $\bgamma$ that provides a smooth coordinate chart (Theorem \ref{mainthm:Gluing}) for the moduli space $\widetilde{M}(P,g)$ of anti-self-dual connections on $P$ on an open neighborhood of a regular boundary point.

The preceding splicing and gluing paradigm extends to the case of \emph{non-regular} boundary points by analogy with the \emph{Kuranishi method} \cite{Kuranishi} for parameterizing open neighborhoods of interior points in the moduli space of anti-self-dual connections \cite{AHS,DK} (see Theorem \ref{maincor:Gluing_HA2_non-zero}). We also construct a splicing map $\fS$ for gauge transformations that is a surjective smooth submersion onto the Banach Lie group of all gauge transformations on $P$. Gauge equivariance then yields the corresponding gluing coordinate chart for the moduli space $M(P,g)$ of gauge-equivalence classes of anti-self-dual connections on $P$ on a neighborhood of a boundary point that may be non-regular or have non-trivial isotropy in the product of Banach Lie groups of gauge transformations on $P_0$ and $P_1$ (see Theorem \ref{maincor:Gluing_HA2_and_HA0_non-zero}).

While we restrict our attention in this article for the sake of simplicity to the case where $(X_1,g_1)$ is the four-dimensional sphere $S^4$ with its standard round metric of radius one, the paradigm just outlined should generalize to many other gluing scenarios, including those listed in Section \ref{subsec:Other_gluing_scenarios}.

\subsection{Main results}
\label{subsec:Main_results}
We begin with an abstract result that lays the foundation for our approach.

\subsubsection{Transversal maps of Banach manifolds with boundary}
If $X$ is a Banach manifold with boundary, we let $\partial X \subset X$ denote the subset of its boundary points and $\Int(X) = X \less \partial X$ denote the subset of its interior points. (See Section \ref{subsubsec:Margalef-Roig_1-2} for a formal definition of the boundary of a Banach manifold.)

\begin{defn}[Transversality for maps of Banach manifolds with boundary]
\label{defn:Transversality_maps_Banach_manifolds_boundary}	
Let $f:X\to X'$ be a $C^p$ map ($p\geq 1$) of $C^p$ Banach manifolds with boundary and $X'' \subset X'$ be a $C^p$ Banach submanifold with boundary. Then $f$ is \emph{transverse to $X''$ at $x\in X$}, denoted $f\transv_x X''$, if one of the following three conditions hold:
\begin{enumerate}
\item $x \notin f^{-1}(X'')$; or
\item $x \in f^{-1}(X'')\cap\Int(X)$ and
\begin{enumerate}	
	\item\label{item:Transversality_linear_span_interior_point}
	$T_{f(x)}X' = \Ran df(x) + T_{f(x)}X''$ and
	
	\item\label{item:Transversality_closed_complement_interior_point}
	$(df(x))^{-1}(T_{f(x)}X'')$ admits a closed complement in $T_xX$; or
\end{enumerate}
\item $x \in f^{-1}(X'')\cap\partial X$ and
\begin{enumerate}	
	\item\label{item:Transversality_linear_span_boundary_point}
	$T_{f(x)}X' = \Ran d(\partial f)(x) + T_{f(x)}X''$ and
	
	\item\label{item:Transversality_closed_complement_boundary_point}
	$(d(\partial f)(x))^{-1}(T_xX'')$ admits a closed complement in $T_xX'$,
\end{enumerate}
where $\partial f \equiv f \restriction \partial X : \partial X \to X'$.
\end{enumerate}
If $f\transv_x X''$ for all $x \in X$, then $f$ is \emph{transverse to $X''$}, denoted $f\transv X''$.
\end{defn}	

If $X$ is a $C^p$ Banach manifold with boundary, then a \emph{chart} for $X$ is a triple $(V,\psi,(E,\alpha))$ comprising an open subset $V\subset X$, a $C^p$ diffeomorphism $\psi:V\to E$ onto an open subset of a closed half space $E_\alpha^+ := \{x \in E: \alpha(x) \geq 0\}$, where $E$ is a real Banach space and $\alpha:E\to\RR$ is a continuous linear function. (See Section \ref{subsubsec:Margalef-Roig_1-2} for a formal definition of a $C^p$ Banach manifold with boundary.) Following Definition \ref{defn:Margalef-Roig_3-1-10}, one says that a $C^p$ Banach submanifold with boundary $Y \subset X$ is \emph{neat} if 
\[
  \partial Y = Y \cap \partial X.
\]
The forthcoming theorem is proved in Section \ref{subsec:Margalef-Roig_7-1}.

\begin{mainthm}[Preimage of a submanifold under a transverse map and the implied embedding map]
\label{mainthm:Preimage_submanifold_under_transverse_map_and_implied_embedding}
Let $X$ and $X'$ be $C^p$ Banach manifolds with boundary ($p\geq 1$), and $X'' \subset X'$ be a neat $C^p$ Banach submanifold, and $f:X\to X'$ be a $C^p$ map, and $x_0\in f^{-1}(X'')$ be a point. If $f \transv_{x_0} X''$, then there are a chart $(V,\psi,(E,\alpha))$ for $X$ with $\psi(x_0)=0$, a closed subspace
\begin{equation}
\label{eq:Maintheorem_definition_L}
L := d\psi(x_0)\left((df(x_0))^{-1}(T_{f(x_0)}X'')\right) \subset E
\end{equation}
with closed complement in $E$ and continuous inclusion operator $\iota_L:L\to E$, and a $C^p$ embedding
\begin{equation}
\label{eq:Maintheorem_preimage_submanifold_under_transverse_map_and_implied_embedding}
g \equiv \psi^{-1}\circ\iota_L  \restriction \psi(V)\cap L : \psi(V)\cap L \to \psi^{-1}(\psi(V)\cap L) \subset X
\end{equation}
from the relatively open subset $\psi(V)\cap L \subset E_\alpha^+$ onto a $C^p$ submanifold $\psi^{-1}(\psi(V)\cap L) \subset X$. Moreover, the following hold:
\begin{enumerate}
\item
\label{item:Maintheorem_Margalef-Roig_4-2-1_submanifold_set}
$\psi^{-1}(\psi(V)\cap L) = V\cap f^{-1}(X'')$, and
\item
\label{item:Maintheorem_Margalef-Roig_4-2-1_submanifold_tangent spaces}
$T_x(f^{-1}(X'')) = (df(x))^{-1}(T_{f(x)}X'') = (d\psi(x))^{-1}L$, for all $x \in V\cap f^{-1}(X'')$.  
\end{enumerate}
Finally,  $f^{-1}(X'')\cap V$ is a neat $C^p$ Banach submanifold of $V$.
\end{mainthm}

\begin{rmk}[Comment on extensions to Banach manifolds with corners]
Although Theorem \ref{mainthm:Preimage_submanifold_under_transverse_map_and_implied_embedding} is phrased in terms of maps of Banach manifolds with boundary, it can be easily extended to the setting of maps of Banach manifolds with corners by drawing on the technical generalizations described by Margalef Roig and Outerelo Dom{\'\i}nguez \cite{Margalef-Roig_Outerelo-Dominguez_differential_topology}. Such extensions are required in many applications, including to the development of gluing theory for anti-self-dual connections and $\SO(3)$ monopoles over four-dimensional manifolds in \cite{FL5,FL7,FL8}.
\end{rmk}

\begin{rmk}[Smooth maps of finite-dimensional manifolds that are transverse to a submanifold without boundary]
When $X$ and $X'$ are finite-dimensional manifolds, $X'$ and $X''$ are without boundary, and $f \transv X''$ in the sense of Definition \ref{defn:Transversality_maps_Banach_manifolds_boundary} (in other words, the smooth maps $f:\Int(X)\to X'$ and $\partial f:\partial X\to X'$ are transverse to $X''$), then Theorem \ref{mainthm:Preimage_submanifold_under_transverse_map_and_implied_embedding} implies that $f^{-1}(X'')$ is a smooth manifold with boundary
\[
  \partial(f^{-1}(X'')) = f^{-1}(X'')\cap\partial X
\]
and the codimension of $f^{-1}(X'')$ in $X$ is equal to the codimension of $X''$ in $X'$. See Guillemin and Pollack \cite[p. 60]{Guillemin_Pollack} for a statement and proof of this result when $X$ and $X'$ are embedded smooth submanifolds of Euclidean space; their statement is quoted here as Theorem \ref{thm:Preimage_theorem_domain_manifold_with_boundary}. 
\end{rmk}

Note that if $\partial f$ obeys the linear span condition in Item \eqref{item:Transversality_linear_span_boundary_point} of Definition \ref{defn:Transversality_maps_Banach_manifolds_boundary} at a boundary point $x\in\partial X$,
\begin{equation}
\label{eq:Linear_span_boundary_map_boundary_point}  
  T_{f(x)}X' = \Ran d(\partial f)(x) + T_{f(x)}X'',
\end{equation}
then we necessarily also have that $f$ obeys the linear span condition
\begin{equation}
\label{eq:Linear_span_map_boundary_point}  
T_{f(x)}X' = \Ran df(x) + T_{f(x)}X'',
\end{equation}
since the continuous linear operator $d(\partial f)(x): T_x(\partial X) \to T_{f(x)}X''$ can be expressed as the composition of the continuous linear operator $df(x): T_xX \to T_{f(x)}X''$ and the continuous embedding of Banach spaces $T_x(\partial X) \subset T_xX$. Example \ref{exmp:X_half_plane_X_prime_plane_f_transv_Y_but_partial_f_not_transv_Y} gives a simple illustration of how Theorem \ref{thm:Preimage_theorem_domain_manifold_with_boundary} can fail to hold when condition \eqref{eq:Linear_span_map_boundary_point} is satisfied at $x\in\partial X$ but not condition \eqref{eq:Linear_span_boundary_map_boundary_point}.

More generally, when $X'$ is a manifold with boundary and $X''\subset X'$ is submanifold with boundary but is not neat, Example \ref{exmp:X_and_X_prime_half_planes_f_transv_Y_but_partial_f_not_transv_Y_and_Y_not_neat} gives a simple illustration of how more general results, such as Theorem \ref{thm:Margalef-Roig_proposition_4-2-1} or Corollary \ref{cor:Margalef-Roig_7-1-20} (that allow $X'$ and $X''$ to have boundary) can fail to hold when $f\transv X''$ in the sense of Definition \ref{defn:Transversality_maps_Banach_manifolds_boundary} but $X''$ is not neat. 

\begin{rmk}[Interpretation of implied embeddings as gluing maps]
\label{rmk:theorem_preimage_submanifold_under_transverse_map_and_implied_embedding_gluing_map}  
The maps $g=\psi^{-1}\circ\iota_L \restriction \psi(V)\cap L$ in \eqref{eq:Maintheorem_preimage_submanifold_under_transverse_map_and_implied_embedding} or $g=\varphi^{-1}\circ\iota_K \restriction \varphi(U)\cap K$ in \eqref{eq:Maintheorem_preimage_point_under_submersion_and_implied_embedding} arise as the \emph{gluing map} when we apply Theorems \ref{mainthm:Preimage_submanifold_under_transverse_map_and_implied_embedding} or  \ref{mainthm:Preimage_point_under_submersion_and_implied_embedding} to prove Theorem \ref{mainthm:Gluing} and Corollaries \ref{maincor:Gluing_smooth}, \ref{maincor:Gluing_HA2_non-zero}, \ref{maincor:Gluing_HA2_and_HA0_non-zero}, and \ref{maincor:Gluing_HA2_and_HA0_non-zero_and_non-flat_metric}.
\end{rmk}

While Theorem \ref{mainthm:Preimage_submanifold_under_transverse_map_and_implied_embedding} allows considerable flexibility in application to the construction of gluing maps, the forthcoming simpler statement (Theorem \ref{mainthm:Preimage_point_under_submersion_and_implied_embedding}) should suffice for some applications.

\begin{defn}[Submersion of Banach manifolds with boundary]
\label{defn:Submersion_Banach_manifolds_boundary}	
Let $f:X\to X'$ be a $C^p$ map ($p\geq 1$) of $C^p$ Banach manifolds with boundary and $x\in X$ be a point. Then $f$ is a \emph{submersion at $x\in X$}, denoted $f\transv_x \pt$, if the following hold:
\begin{enumerate}
\item\label{item:Relatively_open_neighborhood_boundary_point_maps_into_boundary}
  There is an open neighborhood $V_x\subset X$ of $x$ such that $f (V_x\cap \partial X) \subset \partial X'$ and
  
\item\label{item:Surjectivity_linearization}
  $T_{f(x)}X' = \Ran df(x)$ and
	
\item\label{item:Kernel_closed_complement}
  $\Ker df(x)$ admits a closed complement in $T_xX$.
\end{enumerate}
If $f\transv_x \pt$ for all $x \in X$, then $f$ is a \emph{submersion}, denoted $f\transv \pt$.
\end{defn}	

\begin{mainthm}[Preimage of a point under a submersion and the implied embedding map]
\label{mainthm:Preimage_point_under_submersion_and_implied_embedding}
Let $X$ and $X'$ be $C^p$ Banach manifolds ($p\geq 1$) with boundary, $f:X\to X'$ be a $C^p$ map, $x_0'\in X'$ be a point, and $x_0\in f^{-1}(x_0')$. If $f$ is a submersion at $x_0$, then there are a chart $c = (U,\varphi,(E,\lambda))$ for $X$ with $\varphi(x_0)=0$, a closed subspace,
\begin{equation}
\label{eq:Maintheorem_definition_K}
  K := d\varphi(x_0)\left((df(x_0))^{-1}(0)\right) \subset E,
\end{equation}
with closed complement in $E$ and continuous inclusion operator $\iota_K:K\to E$, and a $C^p$ embedding
\begin{equation}
\label{eq:Maintheorem_preimage_point_under_submersion_and_implied_embedding}
  g \equiv \varphi^{-1}\circ\iota_K  \restriction \varphi(U)\cap K : \varphi(U)\cap K \to \varphi^{-1}(\varphi(U)\cap K) \subset X
\end{equation}
from the relatively open subset $\varphi(U)\cap K \subset E_\lambda^+$ onto the $C^p$ Banach submanifold $\varphi^{-1}(\varphi(U)\cap K) \subset X$. Moreover, the following hold:
\begin{enumerate}
\item
\label{item:Maintheorem_Margalef-Roig_4-1-13_submanifold_set}
$\varphi^{-1}(\varphi(U)\cap K) = U\cap f^{-1}(x_0')$.
\item
\label{item:Maintheorem_Margalef-Roig_4-1-13_submanifold_tangent spaces}
$T_x(f^{-1}(x_0')) = (df(x))^{-1}(0) = (d\varphi(x))^{-1}K$, for all $x \in U\cap f^{-1}(x_0')$.  
\end{enumerate}
Finally, if at every point $x \in f^{-1}(x_0')$ there is an open neighborhood $V_x$ of $x$ in $X$ such that $f(V_x \cap \partial X) \subset \partial X'$ and $f$ is a submersion at $x$, then $f^{-1}(x_0')$ is a $C^p$ Banach submanifold of $X$. 
\end{mainthm}

\subsubsection{Gluing maps for anti-self-dual connections}
Let $(X,g)$ denote a closed, connected, four-dimensional, oriented, smooth Riemannian manifold, $G$ denote a Lie group, and $P$ denote a smooth principal $G$-bundle over $X$. Let $p\in (2,\infty)$ and $\sA(P)$ denote the affine space of all $W^{1,p}$ connections on $P$, and $\Aut(P)$ denote the Banach Lie group of all $W^{2,p}$ automorphisms of $P$, and $\sB(P) = \sA(P)/\Aut(P)$ denote the Banach stratified quotient space. As customary, we write $\ad P := P\times_{\Ad}\fg$, where $\Ad:G\to\Aut(\fg)$ is the adjoint representation of $G$ on its Lie algebra $\fg$. If $A$ is a connection on $P$, we let $F_A \in \Omega^2(X;\ad P)$ denote its curvature and $F_A \in \Omega^+(X;\ad P)$ denote its self-dual component with respect to the splitting $\Omega^2(X) = \Omega^+(X)\oplus\Omega^-(X)$ of two-forms into their self-dual and anti-self-dual components with respect to the Riemannian metric, $g$. If $A$ is a smooth anti-self-dual connection on $P$, we let $H_A^\bullet(X;\ad P)$ denote the cohomology groups of the elliptic deformation complex
\begin{equation}
\label{eq:Elliptic_deformation_complex_ASD_equation}  
\Omega^0(X;\ad P) \xrightarrow{d_A} \Omega^1(X;\ad P) \xrightarrow{d_A^+} \Omega^{2,+}(X;\ad P)
\end{equation}
and $\bH_A^\bullet(X;\ad P)$ denote their harmonic representatives, so
\begin{subequations}
\label{eq:Harmonic_cohomology_elliptic_deformation_complex_ASD_equation}    
  \begin{align}
  \label{eq:H0}  
  \bH_A^0(X;\ad P) &:= \Ker d_A \cap \Omega^0(X;\ad P),
  \\
  \label{eq:H1}    
  \bH_A^1(X;\ad P) &:= \Ker(d_A^++d_A^*)\cap \Omega^1(X;\ad P),
  \\
  \label{eq:H2}  
  \bH_A^2(X;\ad P) &:= \Ker d_A^{+,*}\cap \Omega^+(X;\ad P).
\end{align}
\end{subequations}
Recall (see \cite[Equation (4.2.28)]{DK}) that the expected dimension $\dim M(P,g)$ of the \emph{moduli space of anti-self-dual connections} on $P$ (at a point $[A]$),
\begin{equation}
\label{eq:Moduli_space_asd_connections}
  M(P,g) := \{A \in \sA(P): F^{+,g}(A) = 0\}/\Aut(P).
\end{equation}
is given by minus the index of the elliptic deformation complex \eqref{eq:Elliptic_deformation_complex_ASD_equation}, namely
\begin{equation}
\label{eq:Index_elliptic_deformation_complex_ASD_equation}
  s_A(X;\ad P) := h_A^1(X;\ad P) - h_A^0(X;\ad P) - h_A^2(X;\ad P),
\end{equation}
where the quantities $h_A^\bullet(X;\ad P)$ denote the dimensions of the cohomology groups $H_A^\bullet(X;\ad P)$.

Our main result here, namely Theorem \ref{mainthm:Gluing}, complements earlier results due to Taubes (for example, \cite[Theorem 1.1]{TauSelfDual}, \cite[Theorems 1.1 and 1.2]{TauIndef}, and \cite[Proposition 8.2]{TauFrame}), Donaldson and Kronheimer (for example, \cite[Theorems 7.2.62, 8.2.3, and 8.2.4]{DK}), and Mrowka \cite[Main Theorem]{MrowkaThesis}. The novel aspect of our article is the fact that Theorem \ref{mainthm:Gluing} can be derived with relative ease from Theorem \ref{mainthm:Preimage_submanifold_under_transverse_map_and_implied_embedding} and the method extended to cover many other gluing scenarios, as described in Section \ref{subsec:Other_gluing_scenarios}. If $r\in(0,\infty)$ is a constant that is less than or equal to the injectivity radius of a Riemannian manifold $(X,g)$ and $x\in X$ is a point, then we let $B_r(x)$ denote the open geodesic ball in $X$ with center $x$ and radius $r$. We shall first prove a gluing result (see Theorem \ref{mainthm:Gluing}) under several simplifying hypotheses that we then successively relax (see Corollaries \ref{maincor:Gluing_HA2_non-zero},  \ref{maincor:Gluing_HA2_and_HA0_non-zero}, and  \ref{maincor:Gluing_HA2_and_HA0_non-zero_and_non-flat_metric}). Let $g_\round$ denote the standard, round metric of radius one on the sphere $S^4 = \{x\in\RR^5:|x|=1\}$.

\begin{mainthm}[Existence of local gluing chart near a boundary point of the moduli space of anti-self-dual connections]
\label{mainthm:Gluing}
Let $(X,g)$ denote a closed, connected, four-dimensional, oriented, smooth Riemannian manifold, $G$ denote a compact Lie group, $P_0$ denote a smooth principal $G$-bundle over $X$, and $P_1$ denote a smooth principal $G$-bundle over $S^4$. Let $A_{0\flat}$ be a smooth anti-self-dual connection over $(X,g)$ and $A_{1\flat}$ be a smooth centered anti-self-dual connection on $P_1$ over $(S^4,g_\round)$. Assume further that
\begin{gather}
  \label{eq:HA2_zero}
  \bH_{A_{0\flat}}^2(X;\ad P_0) = 0,
  \\
  \label{eq:HA0_zero}
  \bH_{A_{0\flat}}^0(X;\ad P_0) = 0,
  \\
  \label{eq:Riemg_zero}
  g \text{ is conformally flat on } B_{\varrho_0}(x_{0\flat}),
\end{gather}
for some point $x_{0\flat}\in X$ and constant $\varrho_0\in(0,1]$, where $\Riem(g)$ denotes the Riemann curvature tensor of the metric, $g$. Then there is a constant $\delta \in (0,\varrho_0]$ with the following significance. Let $p\in(2,\infty)$ be a constant and
\begin{align*}
  \bC_\delta(A_{0\flat}) &:= \left\{A_0 \in A_{0\flat} + \Ker d_{A_{0\flat}}^*\cap W^{1,p}(T^*X\otimes \ad P_0): F^{+,g}(A_0) = 0 \right.
  \\
  &\qquad \left.\text{and } \|A_0-A_{0\flat}\|_{W_{A_{0\flat}}^{1,2}(X)} < \delta \right\},
  \\
  \bC_\delta^\diamond(A_{1\flat}) &:= \left\{A_1 \in A_{1\flat} + \Ker d_{A_{1\flat}}^*\cap W^{1,p}(S^4;T^*S^4\otimes \ad P_1): F^{+,g_\round}(A_1) = 0 \right.
  \\
                         &\qquad\text{ and }\ (\Center[A_1],\Scale[A_1]) = (\bzero,1) \in \RR^4\times\RR^+
  \\
                         &\qquad\text{ and } \left.\|A_1-A_{1\flat}\|_{W_{A_{1\flat}}^{1,2}(X)} < \delta \right\}.
\end{align*}
Let $P\cong P_0\#_{(x_0,\lambda)} P_1$ denote the smooth principal $G$-bundle over the connected sum $X\#_{(x_0,\lambda)} S^4 \cong X$ defined by the fixed parameters in Data \ref{data:Fixed_parameters_for_definition_splicing_and_unsplicing_maps} and the triples $(\rho,x_0,\lambda) \in \Gl_{x_{0\flat}}\times B_\delta(x_{0\flat})\times (0,\lambda_0)$. Then there is a \emph{gluing map},
\begin{equation}
\label{eq:Gluing_map}
\bgamma: \bC_\delta(A_{0\flat})\times \bC_\delta^\diamond(A_{1\flat}) \times \Gl_{x_{0\flat}} \times B_\delta(x_{0\flat}) \times (0,\lambda_0) \to \sA(P),
\end{equation}
where
\begin{equation}
\label{eq:Bundle_gluing_parameters}
  \Gl_{x_{0\flat}} := \Isom_G(P_0|_{x_{0\flat}}, P_1|_s) \cong G,
\end{equation}
with the following properties:
\begin{enumerate}
\item\label{item:Gluing_map_is_C1_embedding} The map $\bgamma$ is a $C^1$ embedding.
  
\item\label{item:Gluing_map_image_is_open_subset_moduli_space_ASD_connections} The image of $\bgamma$ is an open subset of the moduli space of anti-self-dual connections on $P$:
\[
  \Imag\bgamma \subset M(P,g).
\]

\item\label{item:Gluing_map_extension_homeomorphism_manifolds_with_boundary} The map $\bgamma$ extends to a continuous embedding of manifolds with boundary,
\begin{multline}
\label{eq:Gluing_map_extended}
\bgamma: \bC_\delta(A_{0\flat})\times \bC_\delta^\diamond(A_{1\flat}) \times \Gl_{x_{0\flat}} \times B_\delta(x_{0\flat}) \times [0,\lambda_0)
\\
\to \sA(P)\sqcup \left(\sA(P_0)\times \sA(P_1) \times \Gl_{x_{0\flat}} \times B_\delta(x_{0\flat}) \right),
\end{multline}
when the codomain has the Uhlenbeck topology \cite[Section 4.4.1]{DK}.

\item\label{item:Gluing_map_extension_image_open_subset_bubbletree_compactification} The image of $\bgamma$ in \eqref{eq:Gluing_map_extended} is an open neighborhood of the boundary portion,
\[
  \bC_\delta(A_{0\flat})\times \bC_\delta^\diamond(A_{1\flat}) \times \Gl_{x_{0\flat}}\times B_\delta(x_{0\flat}) \times \{0\},
\]
in the bubble tree compactification $\widehat M(P,g)$ of $M(P,g)$.
\end{enumerate}
\end{mainthm}

\begin{maincor}[Smoothness of local gluing chart for the moduli space of anti-self-dual connections as a map of smooth manifolds with boundary]
\label{maincor:Gluing_smooth}
Continue the hypotheses of Theorem \ref{mainthm:Gluing}. Then the gluing map \eqref{eq:Gluing_map_extended} is the restriction of a $C^1$ map of smooth Banach manifolds with boundary in the sense of Theorem \ref{thm:Smoothness_splicing_map_manifold_with_boundary} and restricts to the identity map on the boundary portion
\[
  \bC_\delta(A_{0\flat})\times \bC_\delta^\diamond(A_{1\flat}) \times \Gl_{x_{0\flat}}\times B_\delta(x_{0\flat}) \times \{0\}.
\]
\end{maincor}

\begin{maincor}[Existence of local gluing chart near a boundary point of the moduli space of anti-self-dual connections that may be non-regular]
\label{maincor:Gluing_HA2_non-zero}
Continue the hypotheses of Theorem \ref{mainthm:Gluing} but omit the assumption \eqref{eq:HA2_zero} that $\bH_{A_{0\flat}}^2(X;\ad P_0)=0$. The gluing map $\bgamma$ in \eqref{eq:Gluing_map} still has the properties listed in Items \eqref{item:Gluing_map_is_C1_embedding} and \eqref{item:Gluing_map_extension_homeomorphism_manifolds_with_boundary}. Moreover, there is an \emph{obstruction section},
\begin{equation}
\label{eq:Obstruction_map}
\bchi: \bC_\delta(A_{0\flat})\times \bC_\delta^\diamond(A_{1\flat}) \times \Gl_{x_{0\flat}} \times B_\delta(x_{0\flat}) \times (0,\lambda_0) \to \bH_{A_{0\flat}}^2(X;\ad P_0),
\end{equation}
with the following properties:
\begin{enumerate}
\item\label{item:Obstruction_map_is_C1_section} The map $\bchi$ is a $C^1$ section of the product bundle.

\item\label{item:Gluing_map_image_zero_set_is_open_subset_moduli_space_ASD_connections} The image of the zero set of $\bchi$ under the gluing map $\bgamma$ is an open subset of the moduli space of anti-self-dual connections on $P$:
\[
  \bgamma(\bchi^{-1}(0)) \subset M(P,g).
\]
(This extends Item \eqref{item:Gluing_map_image_is_open_subset_moduli_space_ASD_connections} in Theorem \ref{mainthm:Gluing}.)

\item\label{item:Obstruction_map_extension_is_C1_section_vector_bundle_over_manifold_with_boundary} The map \eqref{eq:Obstruction_map} extends to a $C^1$ section of a vector bundle over a manifold with boundary.

\item\label{item:Gluing_map_extension_image_zero_set_is_open_subset_bubbletree_compactification} The map $\bgamma$ gives a homeomorphism from
\[
  \bchi^{-1}(0)\cap\left( \bC_\delta(A_{0\flat})\times \bC_\delta^\diamond(A_{1\flat}) \times \Gl_{x_{0\flat}} \times B_\delta(x_{0\flat}) \times [0,\lambda_0) \right)
\]
onto an open neighborhood of the boundary portion,
\[
  \bchi_0^{-1}(0)\cap\bC_\delta(A_{0\flat})\times \bC_\delta^\diamond(A_{1\flat}) \times \Gl_{x_{0\flat}}\times B_\delta(x_{0\flat}) \times \{0\},
\]
in the bubble tree compactification $\widehat M(P,g)$ of $M(P,g)$, where
\[
  \bchi_0:\bC_\delta(A_{0\flat}) \to \bH_{A_{0\flat}}^2(X;\ad P_0)
\]
is the \emph{Kuranishi obstruction section}
\footnote{See Section \ref{subsec:Local_Kuranishi_parameterization_neighborhood_interior_point_Donaldson_approach}}. (This extends Item \eqref{item:Gluing_map_extension_image_open_subset_bubbletree_compactification} in Theorem \ref{mainthm:Gluing}.)
\end{enumerate}
\end{maincor}

Recall that if $A$ is a connection on a principal $G$-bundle $P$, then $\Stab(A)$ denotes the stabilizer (or isotropy) group of $A$ in $\Aut(P)$.

\begin{maincor}[Existence of local gluing chart near a boundary point of the moduli space of anti-self-dual connections that may be non-regular or have non-trivial stabilizer]
\label{maincor:Gluing_HA2_and_HA0_non-zero}
Continue the hypotheses of Theorem \ref{mainthm:Gluing} but omit the assumptions \eqref{eq:HA2_zero} that $\bH_{A_{0\flat}}^2(X;\ad P_0)=0$ and \eqref{eq:HA0_zero} that $\bH_{A_{0\flat}}^0(X;\ad P_0)=0$. Then the gluing map $\bgamma$ and obstruction section $\bchi$ and their $C^1$ extensions are $\Stab(A_{0\flat})$-equivariant and Item \eqref{item:Gluing_map_extension_image_zero_set_is_open_subset_bubbletree_compactification} extends to the following: The map $\overline{\bgamma}$ gives a homeomorphism from
\[
  \bchi^{-1}(0)\cap\left( \bC_\delta(A_{0\flat})/\Stab(A_{0\flat}) \times \bC_\delta^\diamond(A_{1\flat}) \times \Gl_{x_{0\flat}} \times B_\delta(x_{0\flat}) \times (0,\lambda_0) \right)
\]
onto an open neighborhood of the boundary portion,
\[
  \bchi_0^{-1}(0)\cap\bC_\delta(A_{0\flat})/\Stab(A_{0\flat}) \times \bC_\delta^\diamond(A_{1\flat}) \times \Gl_{x_{0\flat}}\times B_\delta(x_{0\flat}) \times \{0\},
\]
in the bubble tree compactification $\widehat M(P,g)$ of $M(P,g)$.    
\end{maincor}


\begin{maincor}[Existence of local gluing chart near a boundary point of the moduli space of anti-self-dual connections that may be non-regular or have non-trivial stabilizer or when the Riemannian metric need not be locally flat]
\label{maincor:Gluing_HA2_and_HA0_non-zero_and_non-flat_metric}
Continue the hypotheses of Theorem \ref{mainthm:Gluing} but omit the assumptions \eqref{eq:HA2_zero} that $\bH_{A_{0\flat}}^2(X;\ad P_0)=0$ and \eqref{eq:HA0_zero} that $\bH_{A_{0\flat}}^0(X;\ad P_0)=0$ and \eqref{eq:Riemg_zero} that $g$ is conformally flat near $x_{0\flat}$. Then the conclusions of Corollary \ref{maincor:Gluing_HA2_and_HA0_non-zero} continue to hold.
\end{maincor}

\subsection{Application to other gluing problems in geometric analysis}
\label{subsec:Other_gluing_scenarios}
The framework that we describe in this article should apply to more challenging gluing problems for anti-self-dual connections or $\SO(3)$ monopoles over four-dimensional manifolds, and also extend to other applications in geometric analysis.

We do not consider applications to Gromov--Witten invariants or symplectic field theory since the analytical difficulties involved in gluing pseudoholomorphic curves in symplectic manifolds are distinctly challenging. A general theory of \emph{polyfolds} have been developed for this purpose by Hofer, Wysocki, and Zehnder in \cite{Hofer_2006, Hofer_Wysocki_Zehnder_1996_ppcs1, Hofer_Wysocki_Zehnder_1998_ppcs1_correction, Hofer_Wysocki_Zehnder_1995_ppcs2, Hofer_Wysocki_Zehnder_1999_ppcs3, Hofer_Wysocki_Zehnder_1996_ppcs4, Hofer_Wysocki_Zehnder_2007, Hofer_Wysocki_Zehnder_2009gafa, Hofer_Wysocki_Zehnder_2009gt, Hofer_Wysocki_Zehnder_2010dcds, Hofer_Wysocki_Zehnder_2010ma, Hofer_Wysocki_Zehnder_2017}. See the survey article by Fabert, Fish, Golovko, Wehrheim \cite{Fabert_Fish_Golovko_Wehrheim_2016} for an introduction to these concepts. Related ideas have been and remain under development by Joyce \cite{Joyce_2012, Joyce_2016, Joyce_2016arxiv}, with similar goals. Our approach to gluing is distinct from the theory of polyfolds; it is also quite different from the methods of other researchers in this field, including Fukaya, Joyce, and McDuff and their collaborators \cite{McDuff_Tehrani_Fukaya_Joyce_2019}.

Our approach also differs from the \emph{Cauchy data matching} gluing constructions for constant mean curvature Riemannian metrics by Mazzeo, Pacard, and Pollack \cite{Mazzeo_Pacard_2001, Mazzeo_Pacard_Pollack_2001}, constant scalar curvature K\"ahler metrics by Arezzo, Pacard, and Singer \cite{Arezzo_Pacard_Singer_2011}, Yamabe metrics by Mazzeo and Pacard \cite{Mazzeo_Pacard_1999}, and Seiberg--Witten monopoles by Kronheimer and Mrowka \cite[Sections 18 and 19]{KMBook}.

\subsection{Outline of the article}
\label{subsec:Outline}
Section \ref{sec:Preliminaries} provides an overview of concepts and notation in gauge theory that we shall need in this article for our proofs of Theorem \ref{mainthm:Gluing} and Corollaries \ref{maincor:Gluing_smooth}, \ref{maincor:Gluing_HA2_non-zero}, \ref{maincor:Gluing_HA2_and_HA0_non-zero}, and \ref{maincor:Gluing_HA2_and_HA0_non-zero_and_non-flat_metric}. In Section \ref{sec:Local_Kuranishi_parameterization_moduli_space_anti-self-dual_connections} we give two expositions of the Kuranishi model for an open neighborhood of an interior point in the moduli space $M(P,g)$ of anti-self-dual connections on a principal $G$-bundle over a closed, four-dimensional Riemannian manifold $(X,g)$ and introduce our approach to constructing a Kuranishi model for an open neighborhood of a boundary point of $M(P,g)$ using techniques drawn from differential topology for Banach manifolds with corners \cite{Margalef-Roig_Outerelo-Dominguez_differential_topology}, specializing to the case of Banach manifolds with boundary in this article.

In Section \ref{sec:Banach_manifolds_boundary}, we summarize the concepts that we shall need for the development and application of the techniques of differential topology in the setting of Banach manifolds with boundary, culminating in the proofs of our Theorems \ref{mainthm:Preimage_submanifold_under_transverse_map_and_implied_embedding} and \ref{mainthm:Preimage_point_under_submersion_and_implied_embedding}. Our review closely follows the work of Margalef Roig and Outerelo Dom{\'\i}nguez \cite{Margalef-Roig_Outerelo-Dominguez_differential_topology}, though we simplify their definitions and results for Banach manifolds with corners to the case of Banach manifolds with boundary. However, as we note in Section \ref{subsec:Other_gluing_scenarios}, a generalization of this article to allow for more than one bubble would require us to avail of the methods and results of \cite{Margalef-Roig_Outerelo-Dominguez_differential_topology} in the case of Banach manifolds with corners.

In Section \ref{sec:Splicing_map_connections}, we describe the lengthy construction of the \emph{splicing map for connections}. We first consider the case of a pair of principal $G$-bundles $P_0$ and $P_1$ over a pair $(X_0,g_0)$ and $(X_1,g_1)$ of closed, connected, four-dimensional, oriented, smooth Riemannian manifolds and form a connected-sum principal $G$-bundle $P$ over a connected-sum four-manifold $X = X_0\# X_1$. We broadly follow the method described by Donaldson and Kronheimer in \cite[Section 7.2.1]{DK}, but we add detail that will become important in our later calculations. This construction involves (in part --- see Data \ref{data:Splicing_parameters_connected_sum_G-bundle_connected_sum_Riemannian_4-manifold} for the complete list of parameter choices) choosing basepoints $x_{0\flat} \in X_0$ and $x_{1\flat} \in X_1$, a scale parameter $\lambda\in(0,\lambda_0)$ (for a small constant $\lambda_0\in(0,1]$), open balls $B_{2\sqrt{\lambda}}(x_{i\flat}) \subset X_i$  for $i=0,1$. We initially assume that the metrics $g_0$ and $g_1$ are flat near the basepoints and identify the small annuli $\Omega(x_{i\flat};\frac{1}{2}\sqrt{\lambda},2\sqrt{\lambda}) \subset X_i$ for $i=0,1$ via the conformal, orientation-reversing diffeomorphism $f_\lambda$ in \cite[Equation (7.2.2)]{DK} (we suppress notation indicating dependence on other choices is suppressed for simplicity) to give
\[
  X = X_0\#_\lambda X_1 := X_0'\cup_{f_\lambda} X_1'
\]
where $X_i' := X_i \less B_{\frac{1}{2}\sqrt{\lambda}}(x_{i\flat})$. Given connections $A_i$ on $P_i$, we cut them off over the annuli $\Omega(x_{i\flat};\frac{1}{2}\sqrt{\lambda},2\sqrt{\lambda})$ with the aid of a partition of unity $\{\chi_0,\chi_1\}$ for $X$ to form connections $\chi_iA_i$ that coincide with $A_i$ over $X_i \less B_{2\sqrt{\lambda}}(x_{i\flat})$ (where $\chi_i=1$) and the product connection $\Theta$ over $B_{\frac{1}{2}\sqrt{\lambda}}(x_{i\flat})\times G$ (where $\chi_i=0$). We obtain a connection $A=\chi_0A_0+\chi_1A_1$ on $P$ and thus define the splicing map for connections,
\begin{equation}
\label{eq:Splicing_map_connections_X0_and_X1}
\cS: \sA(P_0)\times \sA(P_1) \times \Gl_{x_{0\flat},x_{1\flat}} \to \sA(P),
\end{equation}
again broadly following the method described by Donaldson and Kronheimer in \cite[Sections 4.4.2, 4.4.3, and 7.2.1]{DK}. Here, $\Gl_{x_{0\flat},x_{1\flat}} = \Isom_G(P_0|_{x_0\flat}, P_1|_{x_{1\flat}})$, the set of $G$-equivariant isomorphisms of $P_0|_{x_0\flat}$ with $P_1|_{x_{1\flat}}$, or \emph{bundle gluing parameters}. It is straightforward to prove (as we do in Section \ref{sec:Splicing_map_connections}) that $\cS$ in \eqref{eq:Splicing_map_connections_X0_and_X1} is a smooth surjective submersion of smooth Banach manifolds.

In order to illustrate the application of our abstract results (Theorems \ref{mainthm:Preimage_submanifold_under_transverse_map_and_implied_embedding} and \eqref{eq:Maintheorem_preimage_point_under_submersion_and_implied_embedding}) from differential topology for Banach manifolds with boundary in the simplest possible setting, we specialize to the case $(X_1,g_1) = (S^4,g_\round)$ and $(X_0,g_0) = (X,g)$, choose $x_{1\flat}$ to be the south pole in $S^4$, allow the center point $x_0 \in X$ defining the connected sum to vary in a small open ball $B_\delta(x_{0\flat}) \subset X$ (for a constant $\delta\in(0,1]$ that is less than half the injectivity radius of $(X,g)$), and allow the scale $\lambda\in(0,\lambda_0)$ to vary. We thus obtain the splicing map for connections \eqref{eq:Splicing_map_connections} that we primarily study in this article: 
\[
\cS: \sA(P_0)\times \sA(P_1) \times \Gl_{x_{0\flat}}\times B_\delta(x_{0\flat}) \times (0,\lambda_0) \to \sA(P),
\]
where $\Gl_{x_{0\flat}} := \Isom_G(P_0|_{x_0\flat}, P_1|_s)$. In the definition \eqref{eq:Splicing_map_connections}, the map $\cS$ is easily shown to be a surjective submersion, even for fixed parameters $(x_0,\lambda) \in B_\delta(x_{0\flat}) \times (0,\lambda_0)$. On the other hand, the domain of $\cS$ in \eqref{eq:Splicing_map_connections} clearly contains some redundancy since
\begin{itemize}
\item Varying the point $x_0 \in B_\delta(x_{0\flat})$ is equivalent to varying the \emph{center of mass} of the energy density $|F_{A_1}|^2$ over $\RR^4 \cong S^4\less\{s\}$, and
\item Varying the parameter $\lambda \in (0,\lambda_0)$ is equivalent to varying the \emph{scale} or \emph{standard deviation} of the energy density $|F_{A_1}|^2$ over $\RR^4$.
\end{itemize}
We then prove that by restricting $\cS$ to the codimension-five submanifold of the domain in \eqref{eq:Splicing_map_connections} obtained by replacing the Banach affine space $\sA(P_1)$ with the codimension-five submanifold $\sA^\diamond(P_1)$ of \emph{centered connections} (those connections with center of mass at the origin of $\RR^4$ or north pole of $S^4$ and scale one), we still obtain a splicing map \eqref{eq:Splicing_map_connections_centered_connections_4-sphere} that is a smooth surjective submersion:
\[
\cS: \sA(P_0)\times \sA^\diamond(P_1) \times \Gl_{x_{0\flat}}\times B_\delta(x_{0\flat}) \times (0,\lambda_0) \to \sA(P).
\]
The domain of $\cS$ in \eqref{eq:Splicing_map_connections_centered_connections_4-sphere} has a boundary at $\lambda=0$. 

We recall in Section \ref{sec:Splicing_gauge_transformations} that the quotient space $(\sA(P_1) \times P_1|_s)/\Aut(P_1)$ is naturally identified with $\sB_0(P_1) = \sA(P_1)/\Aut_0(P_1)$, where $\Aut_0(P_1) \subset \Aut(P_1)$ is the normal Banach Lie subgroup of automorphisms that restrict to the identity map on the fiber $P_1|_s$. If we fix a fiber point once and for all,
\begin{equation}
\label{eq:Fixed_fiber_point_x0} 
  p_0 \in P_0|_{x_{0\flat}},
\end{equation}
then we obtain a $G$-equivariant isomorphism
\begin{equation}
\label{eq:Isomorphism_bundle_gluing_parameters_fiber_points_south_pole}  
  \Gl_{x_{0\flat}} \ni \rho \mapsto p_1 = \rho(p_0) \in P_1|_s.
\end{equation}
With this in mind, it is convenient to modify the domain of $\cS$ in \eqref{eq:Splicing_map_connections} to give a map
\begin{multline*}
\cS: \sA(P_0)\times \sA^\diamond(P_1) \times P_1|_s \times B_\delta(x_{0\flat}) \times (0,\lambda_0) \ni (A_0,A_1,p_1,x_0,\lambda)
\\
\mapsto (A,p_1,x_0,\lambda) \in \sA(P)\times P_1|_s \times B_\delta(x_{0\flat}) \times (0,\lambda_0).
\end{multline*}
In Section \ref{sec:Smooth_extension_splicing_map_connections}, we prove that $\cS$ extends (in the sense described in that section) to a $C^1$ map of Banach manifolds with boundary that restricts to the identity map on the boundary face $\{\lambda=0\}$.

In order to prove Theorem \ref{mainthm:Gluing} and its corollaries, we must construct local parameterizations of the finite-dimensional moduli subspace of anti-self-dual connections,
\[
  M(P,g) = \{A \in \sA(P):F_A^+=0\}/\Aut(P) \subset \sB(P).
\]
Therefore, we are led in Section \ref{sec:Composition_self-dual_curvature_splicing_maps} to consider the composed map
\[
F^+\circ\cS: \sA(P_0)\times \sA^\diamond(P_1) \times \Gl_{x_{0\flat}}\times B_\delta(x_{0\flat}) \times (0,\lambda_0) \to L^p(\wedge^+(T^*X)\otimes\ad P).
\]
In Section \ref{sec:Smooth_extension_composition_self-dual_curvature_and_splicing_maps}, we prove that $F^+\circ\cS$ extends  (in the sense described in that section) to a $C^1$ map of Banach manifolds with boundary that restricts to the map
\[
  (A_0,A_1,p_1,x_0,0) \mapsto (F_{A_0}^+,F_{A_1}^+)
\]  
on the boundary face $\{\lambda=0\}$. The points $[A_1,p_1] \in M_0(P_1,g_\round)$ are \emph{always} (gauge-equivalence classes of) regular points of the smooth map
\[
  \sA(P_1) \times P_1|_s \ni (A_1,p_1) \mapsto F_{A_1}^+ \in L^p(\wedge^+(T^*S^4)\otimes\ad P_1).
\]
In the hypotheses of Theorem \ref{mainthm:Gluing}, we also assume that $A_{0\flat}$ is a regular point of the smooth map
\[
  \sA(P_0) \ni A_0 \mapsto F_{A_0}^+ \in L^p(\wedge^+(T^*X)\otimes\ad P_0).
\]
In Section \ref{sec:Completion_proof_main_gluing_theorem}, we apply our results from differential topology for abstract Banach manifolds with boundary, Theorem \ref{mainthm:Preimage_submanifold_under_transverse_map_and_implied_embedding}, to conclude the existence of a gluing map with factored codomain,
\begin{multline}
\label{eq:Gluing_map_factored_codomain}  
\widehat{\bgamma}: \bC_\delta(A_{0\flat}) \times \bC_\delta^\diamond(A_{1\flat}) \times P_1|_s \times B_\delta(x_{0\flat}) \times (0,\lambda_0)
\\
\to \sA(P_0)\times \sA^\diamond(P_1) \times P_1|_s \times B_\delta(x_{0\flat}) \times (0,\lambda_0),
\end{multline}
that extends to a $C^1$ map of smooth Banach manifolds with boundary. We obtain the desired gluing map \eqref{eq:Gluing_map} with codomain $\sA(P)$ as the composition
\[
\bgamma = \cS\circ\widehat{\bgamma}: \bC_\delta(A_{0\flat}) \times \bC_\delta^\diamond(A_{1\flat}) \times P_1|_s \times B_\delta(x_{0\flat}) \times (0,\lambda_0) \to \sA(P).
\]
The preceding observations yield Theorem \ref{mainthm:Gluing} and Corollary \ref{maincor:Gluing_smooth}.


By adapting methods due to Donaldson and Kronheimer \cite{DonConn, DK} and Taubes \cite{TauIndef, TauFrame}, that are in turn inspired by ideas of Kuranishi \cite{Kuranishi} from the context of deformation of complex structure, it is relatively straightforward to extend Theorem  \ref{mainthm:Gluing} to the case where $A_{0\flat}$ is not a regular point and prove Corollary \ref{maincor:Gluing_HA2_non-zero} in Section \ref{sec:Non-regular_boundary_points}. In Section \ref{sec:Splicing_gauge_transformations}, we construct the splicing map for based gauge transformations,
\[
  \fS:\Aut_0(P_0) \times \Aut_0(P_1) \times P_1|_s \times B_\delta(x_{0\flat}) \times (0,\lambda_0) \to \Aut_0(P),
\]
observe that $\fS$ is a smooth surjective submersion that extends in a natural way to a $C^1$ map of Banach manifolds with boundary that restricts to an identity map on the boundary face $\{\lambda=0\}$. We then consider the gauge equivariance of the maps $\cS$ and $F^+$, consider the case where $A_{0\flat}$ may have a non-trivial isotropy group in $\Aut(P_0)$, and prove Corollary \ref{maincor:Gluing_HA2_and_HA0_non-zero_and_non-flat_metric}. Finally, in Section \ref{sec:Riemannian_metric_not_locally_flat} we remove the hypothesis in Theorem \ref{mainthm:Gluing} that $g$ is flat near $x_{0\flat}$ and prove Corollary \ref{maincor:Gluing_HA2_and_HA0_non-zero_and_non-flat_metric}.

\subsection{Acknowledgments}
\label{subsec:Acknowledgments}
Paul Feehan is grateful to Helmut Hofer, Rafe Mazzeo, Tom Parker, Cliff Taubes, and Zhengyi Zhou for helpful communications and conversations. This article was completed while he was a visiting professor in the Department of Mathematics at Princeton University. He would like to thank David Gabai and Peter Ozsv{\'a}th for arranging his visit during his sabbatical year and thank Alice Chang for her invitation to speak about this work in the Differential Geometry and Geometric Analysis seminar. Both authors thank the National Science Foundation for their support.

\section{Preliminaries}
\label{sec:Preliminaries}
Throughout our article, $G$ denotes a compact Lie group and $P$ a smooth principal $G$-bundle over a closed, smooth manifold, $X$, of dimension $d \geq 2$ and endowed with Riemannian metric, $g$. We let\footnote{We follow the notational conventions of Friedman and Morgan \cite[p. 230]{FrM}, where they define $\ad P$ as we do here and define $\Ad P$ to be the group of automorphisms of the principal $G$-bundle, $P$.}
$\ad P := P\times_{\ad}\fg$ denote the real vector bundle associated to $P$ by the adjoint representation of $G$ on its Lie algebra,
$\Ad:G \ni u \to \Ad_u \in \Aut\fg$. We fix an inner product on the Lie algebra $\fg$ that is invariant under the adjoint action of $G$ and thus define a fiber metric on $\ad P$. When $\fg$ is semisimple, one may use a negative multiple of the Cartan--Killing form $\kappa:\fg\times\fg\to\RR$ to define such an inner product on $\fg$ --- for example, see Hilgert and Neeb \cite[Definition 5.5.3 and Theorem 5.5.9]{Hilgert_Neeb_structure_geometry_lie_groups}. More generally, because $G$ is compact it has a faithful representation $\rho:G\to \Aut(V)$ for some complex vector space $V$ as a consequence of the Peter--Weyl Theorem and so $G$ is isomorphic to a closed subgroup of $\U(n)$ or $\Or(n)$ for some integer $n$ (see Br\"ocker and tom Dieck \cite[Theorem III.4.1 and Exercise III.4.7.1]{BrockertomDieck} or Knapp \cite[Corollary 4.22]{Knapp_1986}). We can then obtain the desired inner product on $\fg$ by applying \cite[Proposition 4.24]{Knapp_1986} or by restricting the inner product $\langle\xi,\eta\rangle := \tr(\xi^*\eta)$, for all $\xi,\eta\in\fu(n)$.

Because choices of conventions in Yang--Mills gauge theory vary among authors and as such choices will matter here, we shall summarize our choices. We follow the mathematical conventions of Kobayashi and Nomizu \cite[Chapters II and III]{Kobayashi_Nomizu_v1}, with amplifications described by Bleecker \cite[Chapters 1--3]{Bleecker_1981} that are useful in gauge theory, though we adopt the notation employed by Donaldson and Kronheimer \cite[Chapters 2--4]{DK} and Uhlenbeck \cite{UhlLp}. Bourguignon and Lawson \cite[Section 2]{Bourguignon_Lawson_1981} provide a useful summary of Yang--Mills gauge theory that overlaps with our development here.

We assume that $G$ acts on $P$ on the right \cite[Definition 1.1.1]{Bleecker_1981}, \cite[Section I.1.5]{Kobayashi_Nomizu_v1}. We let $A$ denote a smooth connection on $P$ through any one of its three standard equivalent definitions, namely \cite[Definitions 1.2.1, 1.2.2, and 1.2.3 and Theorems 1.2.4 and 1.2.5]{Bleecker_1981}, \cite[Section II.1]{Kobayashi_Nomizu_v1}:
\begin{inparaenum}[(\itshape i\upshape)]
\item a connection one-form $A \in \Omega^1(P;\fg)$,
\item a family of horizontal subspaces $H_p\subset T_pP$ smoothly varying with $p\in P$, or
\item a set of smooth local connection one-forms $A_\alpha \in \Omega^1(U_\alpha;\fg)$ with respect to an open cover $\{U_\alpha\}_{\alpha\in\sI}$ of $X$ and smooth local sections $\sigma_\alpha:U_\alpha \to P$.
\end{inparaenum}
In particular, if $g_{\alpha\beta}:U_\alpha\cap U_\beta\to G$ is a smooth transition function \cite[Definition 1.1.3]{Bleecker_1981}, \cite[Section I.1.5]{Kobayashi_Nomizu_v1} defined by $\sigma_\beta = \sigma_\alpha g_{\alpha\beta}$, then \cite[Definition 1.2.3]{Bleecker_1981}, \cite[Proposition II.1.4]{Kobayashi_Nomizu_v1}
\begin{equation}
\label{eq:Kobayashi_Nomizu_proposition_II-1-4}
A_\beta = \Ad(g_{\alpha\beta}^{-1})A_\alpha + g_{\alpha\beta}^*\theta \quad\text{on } U_\alpha\cap U_\beta,  
\end{equation}
where $\theta \in \Omega^1(G;\fg)$ is the \emph{Maurer--Cartan form} (or \emph{canonical one-form}); when $G\subset\GL(n,\CC)$, then \eqref{eq:Kobayashi_Nomizu_proposition_II-1-4} simplifies to give
\[
  A_\beta = g_{\alpha\beta}^{-1}A_\alpha g_{\alpha\beta} + g_{\alpha\beta}^{-1}dg_{\alpha\beta} \quad\text{on } U_\alpha\cap U_\beta.
\]
In particular, if $B$ is any other smooth connection on $P$, then $A-B \in \Omega^1(X;\ad P)$ \cite[Theorem 3.2.8]{Bleecker_1981}, where we let
\[
  \Omega^l(X; \ad P) :=  C^\infty(X;\wedge^l(T^*X)\otimes\ad P)
\]
the Fr{\'e}chet space of $C^\infty$ sections of $\wedge^l(T^*X)\otimes\ad P$, for an integer $l\geq 0$.

Given a connection $A$ on $P$, one obtains the \emph{exterior covariant derivative} \cite[Definitions 2.2.2 and 3.1.3]{Bleecker_1981}, \cite[Proposition II.5.1]{Kobayashi_Nomizu_v1}
\[
  d_A: \bar\Omega^l(P;\fg) \to \bar\Omega^{l+1}(P;\fg),
\]
where $\l\geq 0$ is an integer and $\bar\Omega^l(P;\fg) \subset \Omega^l(P;\fg)$ is the subspace of \emph{tensorial $l$-forms $\varphi$ of type $\ad\, G$} such that \cite[Definition 3.1.2]{Bleecker_1981}, \cite[p. 75]{Kobayashi_Nomizu_v1}
\begin{inparaenum}[(\itshape i\upshape)]
\item
\label{item:tensorial_invariance}
$R_g^*\varphi = \Ad(g^{-1})\varphi$ for all $g\in G$, where $R_g:P\to P$ denotes right multiplication by $g$, and
\item
\label{item:tensorial_vanish_vertical_fibers}
$\varphi_p(\xi_1,\ldots,\xi_l)=0$ if any one of $\xi_i\in T_pP$ is vertical, for $p\in P$.
\end{inparaenum}
If $\varphi \in \Omega^l(P;\fg)$ obeys condition \eqref{item:tensorial_invariance} but not \eqref{item:tensorial_vanish_vertical_fibers}, then $\varphi$ is a \emph{pseudotensorial $l$-form of type $\ad\,G$}. In particular, $A \in \Omega^1(P;\fg)$ is a pseudotensorial $1$-form of type $\ad\, G$ by \cite[Proposition II.1.1]{Kobayashi_Nomizu_v1}. As customary \cite[Equation (2.1.12)]{DK}, we also let
\begin{equation}
\label{eq:Exterior_covariant_derivative}
  d_A: \Omega^l(X;\ad P) \to \Omega^{l+1}(X;\ad P),
\end{equation}
denote the equivalent expression for exterior covariant derivative and let
\begin{equation}
\label{eq:Exterior_covariant_derivative_L2_adjoint}
d_A^*:\Omega^{l+1}(X; \ad P) \to \Omega^l(X; \ad P),
\end{equation}
and denote its $L^2$-adjoint with respect to the Riemannian metric \cite[Equation (2.1.24)]{DK}.

If $\varphi \in \bar\Omega^l(P;\fg)$, then \cite[Corollary 3.1.6]{Bleecker_1981}
\begin{equation}
\label{eq:Bleecker_corollary_3-1-6}
  d_A\varphi = d\varphi + [A,\varphi] \in \bar\Omega^{l+1}(P;\fg).
\end{equation}
If $\varphi \in \Omega^l(X;\ad P)$, then we have the corresponding local expressions,
\begin{equation}
\label{eq:Bleecker_corollary_3-1-6_local}
  d_A\varphi\restriction_{U_\alpha} = d\varphi + [A_\alpha,\varphi] \in \Omega^{l+1}(U_\alpha;\fg),
\end{equation}
or in the case of $G\subset\GL(n,\CC)$ \cite[Theorem 2.2.12]{Bleecker_1981},
\[
  d_A\varphi\restriction_{U_\alpha} = d\varphi + A_\alpha\wedge\varphi - (-1)^l\varphi\wedge A_\alpha  \in \Omega^{l+1}(U_\alpha;\fg).
\]
The \emph{curvature} of $A\in \Omega^1(P;\fg)$ is defined by \cite[Definition 2.2.3]{Bleecker_1981}, \cite[p. 77]{Kobayashi_Nomizu_v1}
\begin{equation}
\label{eq:Kobayashi_Nomizu_page_77}
  F_A = d_A A \in \bar\Omega^2(P;\fg),
\end{equation}
and by virtue of the \emph{structure equation} \cite[Theorem 2.2.4]{Bleecker_1981}, \cite[Theorem II.5.2]{Kobayashi_Nomizu_v1}, one has
\begin{equation}
\label{eq:Kobayashi_Nomizu_theorem_2-5-2}
F_A = dA + \frac{1}{2}[A,A] \in \bar\Omega^2(P;\fg).
\end{equation}
(Note that $d_A\varphi \in \bar\Omega^{l+1}(P;\fg)$ even if $\varphi \in \Omega^l(P;\fg)$ is only pseudotensorial by \cite[Proposition  II.5.1 (c)]{Kobayashi_Nomizu_v1}.) We also write $F_A \in \Omega^2(X;\ad P)$ for the curvature equivalently defined by the corresponding set of local expressions \cite[Theorem 2.2.11]{Bleecker_1981}
\begin{equation}
\label{eq:Bleecker_theorem_2-2-11}
  F_A\restriction_{U_\alpha} = dA_\alpha + \frac{1}{2}[A_\alpha,A_\alpha] \in \Omega^2(U_\alpha;\fg),
\end{equation}
or in the case of $G\subset\GL(n,\CC)$ \cite[Corollary 2.2.13]{Bleecker_1981},
\[
  F_A\restriction_{U_\alpha} = dA_\alpha + A_\alpha\wedge A_\alpha \in \Omega^2(U_\alpha;\fg).
\]
If $a \in \bar\Omega^1(P;\fg)$, then \eqref{eq:Kobayashi_Nomizu_theorem_2-5-2} yields
\[
  F_{A+a} = dA + \frac{1}{2}[A,A] + da + \frac{1}{2}[A,a] + \frac{1}{2}[a,A] + \frac{1}{2}[a,a],
\]
that is, using \eqref{eq:Bleecker_corollary_3-1-6_local} and $[a,A]=[A,a]$ by the forthcoming \eqref{eq:Bleecker_definition_2-1-1_one_forms},
\begin{equation}
\label{eq:Donaldson_Kronheimer_2-1-14}
F_{A+a} = F_A + d_Aa + \frac{1}{2}[a,a].
\end{equation}
or in the case of $G\subset\GL(n,\CC)$ \cite[Equation 2.1.14]{DK},
\[
  F_{A+a} = F_A + d_Aa + a\wedge a.
\]
We note that if $a, b \in \Omega^1(X;\ad P)$ and $\xi,\eta\in C^\infty(TX)$, then \cite[Definition 2.1.1]{Bleecker_1981}
\begin{equation}
  \label{eq:Bleecker_definition_2-1-1_one_forms}
  [a,b](\xi,\eta) = [a(\xi),b(\eta)] - [a(\eta),b(\xi)]
\end{equation}
or in the case of $G\subset\GL(n,\CC)$ \cite[Theorem 2.2.12]{Bleecker_1981},
\[
  [a,b] = a\wedge b + b\wedge a.
\]
We let $\Aut(P)$ denote the Fr{\'e}chet space of all smooth automorphisms of $P$ \cite[Definition 3.2.1]{Bleecker_1981}, or \emph{gauge transformations}. We recall that $\Aut(P) \cong \Omega^0(X;\Ad P)$ by \cite[Theorem 3.2.2]{Bleecker_1981}, where $\Ad P := P\times_G G$ and $g\in G$ acts on $G$ on the left by conjugation via $h \mapsto ghg^{-1}$ for all $h\in G$ \cite[Definition 3.1.1]{Bleecker_1981}. If $\sA(P)$ denotes the Fr{\'e}chet space of all connections on $P$, then one obtains a right action \cite[Theorem 3.2.5]{Bleecker_1981}, \cite[Theorem II.6.1]{Kobayashi_Nomizu_v1},
\begin{equation}
\label{eq:Right_action_gauge_transformations_on_connections}
  \sA(P) \times \Aut(P) \ni (A,u) \mapsto u(A) = u^*A \in \sA(P).
\end{equation}
If $u \in \Aut(P)$ is represented locally by $u(\sigma_\alpha) = \sigma_\alpha s_\alpha$ on $U_\alpha \subset X$, where $\sigma_\alpha:U_\alpha\to P$ is a local section and $s_\alpha:U_\alpha\to G$ is a smooth map, then \cite[Theorem 3.2.14]{Bleecker_1981}
\begin{equation}
\label{eq:Bleecker_theorem_3-2-14}
u(A)\restriction_{U_\alpha} = \Ad(s_\alpha^{-1})A_\alpha + s_\alpha^*\theta \in \Omega^1(U_\alpha;\fg),  
\end{equation}  
or in the case of $G\subset\GL(n,\CC)$,
\[
  u(A)\restriction_{U_\alpha} = s_\alpha^{-1}A_\alpha s_\alpha + s_\alpha^{-1}ds_\alpha \in \Omega^1(U_\alpha;\fg).
\]
If $B$ is any other smooth connection on $P$ and $G\subset\GL(n,\CC)$, then
\begin{align*}
  (u(A) - B)_\alpha
  &=
    s_\alpha^{-1}A_\alpha s_\alpha + s_\alpha^{-1}ds_\alpha - B_\alpha
  \\
  &= s_\alpha^{-1}(A_\alpha-B_\alpha) s_\alpha + s_\alpha^{-1}(ds_\alpha + [B_\alpha,s_\alpha])
  \\
  &= s_\alpha^{-1}(A-B)_\alpha s_\alpha + s_\alpha^{-1}d_Bs_\alpha \quad\text{on } U_\alpha.
\end{align*}
If $s \in \Omega^0(X;\Ad P)$ is represented locally by the collection $\{s_\alpha\}_{\alpha\in\sI}$, then (as in \cite[p. 32]{UhlLp}) the corresponding global expression for the action of $u\in\Aut(P)$ is given by 
\begin{equation}
\label{eq:Uhlenbeck_1982_Lp_gauge_action_page_32}
u(A)-B = s^{-1}(A-B)s + s^{-1}d_Bs.
\end{equation}
In order to construct Sobolev spaces of connections and gauge transformations, extending the usual definitions of Sobolev spaces of functions on open subsets of Euclidean space in Adams and Fournier \cite[Chapter 3]{AdamsFournier}, we shall need suitable covariant derivatives. If $E$ is a smooth vector bundle over $X$ with covariant derivative \cite[Section III.1]{Kobayashi_Nomizu_v1}
\[
  \nabla:C^\infty(X;E) \to C^\infty(X;T^*X\otimes E),
\]
and $A$ is smooth connection on $P$ with induced covariant derivative (see \cite[Equation (2.1.12) (ii)]{DK} or Kobayashi \cite[Equation (1.1.1)]{Kobayashi})
\begin{equation}
\label{eq:Covariant_is_exterior_covariant_derivative_on_sections}
  \nabla_A = d_A:C^\infty(X;\ad P) \to C^\infty(X;T^*X\otimes \ad P),
\end{equation}
we let $\nabla_A$ denote the induced covariant derivative on the tensor product bundle $E\otimes\ad P$,
\[
  \nabla_A:C^\infty(X;E\otimes\ad P) \to C^\infty(X;T^*X\otimes E\otimes \ad P).
\]
The covariant derivative on $E=\wedge^l(T^*X)$ is induced by the Levi--Civita connection on $T^*X$.

We denote the Banach space of sections of $\wedge^l(T^*X)\otimes\ad P$ of Sobolev class $W^{k,p}$, for any $k\in \NN$ and $p \in [1,\infty]$, by $W_A^{k,p}(X; \wedge^l(T^*X)\otimes\ad P)$, with norm,
\[
\|\phi\|_{W_A^{k,p}(X)} := \left(\sum_{j=0}^k \int_X |\nabla_A^j\phi|^p\,d\vol_g \right)^{1/p},
\]
when $1\leq p<\infty$ and
\[
\|\phi\|_{W_A^{k,\infty}(X)} := \sum_{j=0}^k \esssup_X |\nabla_A^j\phi|,
\]
when $p=\infty$, where $\phi \in W_A^{k,p}(X; \wedge^l(T^*X)\otimes\ad P)$.

For $p \geq 1$ and a fixed $C^\infty$ connection on $P$, we let
\begin{equation}
\label{eq:Affine_space_W1p_connections}  
  \sA^{1,p}(P) := A_1 + W_{A_1}^{1,p}(X;T^*X\otimes\ad P)
\end{equation}
denote the affine space of Sobolev $W^{1,p}$ connections on $P$. For $p \in (d/2,\infty)$, we let $\Aut^{2,p}(P)$ denote the Banach Lie group of Sobolev $W^{2,p}$ automorphisms of $P$ \cite[Section 2.3.1]{DK}, \cite[Appendix A and p. 32 and pp. 45--51]{FU}, \cite[Section 3.1.2]{FrM}, let
\begin{equation}
\label{eq:Quotient_space_W1p_connections}  
 \sB^{1,p}(P) := \sA^{1,p}(P)/\Aut^{2,p}(P)
\end{equation}
denote the quotient space of gauge-equivalence classes of $W^{1,p}$ connections on $P$, and let
\begin{equation}
\label{eq:Quotient_map_W1p_connections}  
 \pi: \sA^{1,p}(P) \ni A \mapsto [A] \in \sB^{1,p}(P)
\end{equation}
denote the quotient map.

\subsection{Notation and conventions}
\label{subsec:Notation}
Throughout this article, constants are generally denoted by $C$ (or $C(*)$ to indicate explicit dependencies) and may increase from one line to the next in a series of inequalities. We write $\eps \in (0,1]$ to emphasize a positive constant that is understood to be small or $K \in [1,\infty)$ to emphasize a constant that is understood to be positive but finite. We let $\Inj(X,g)$ denote the injectivity radius of a smooth Riemannian manifold $(X,g)$. Following Adams and Fournier \cite[Sections 1.26 and 1.28]{AdamsFournier}, for an open subset $U\subset\RR^n$ and integer $m\geq 0$, we let $C^m(U)$ (respectively, $C^m(\bar U)$) denote the vector space of (real or complex-valued) functions on $U$ which, together with their derivatives up to order $m$, are continuous (respectively, bounded and uniformly continuous) on $U$. The H\"older spaces $C^{m,\lambda}(\bar U)$ for $\lambda\in(0,1]$ are defined as in \cite[Section 1.29]{AdamsFournier}. We write $C^{m,\lambda}(U)$ (or equivalently, $C_{\loc}^{m,\lambda}(U)$) for the vector space of functions $f$ such that $f \in C^{m,\lambda}(\bar V)$ for all $V\Subset U$.

Unless we need to indicate a different regularity for connections and gauge transformations, we shall always assume that $p\in(d/2,\infty)$ and abbreviate $\sA^{1,p}(P)$, $\Aut^{2,p}(P)$, $\sB^{1,p}(P)$, and so on, by $\sA(P)$, $\Aut(P)$, $\sB(P)$, respectively. We call $X$ an \emph{admissible four-manifold} if it is a closed\footnote{By which we mean, as usual, compact and without boundary.}, connected, four-dimensional, orientable, smooth manifold and call $(X,g)$ an \emph{admissible Riemannian four-manifold} if $X$ is an admissible four-manifold that is equipped with a smooth Riemannian metric $g$. As usual, we let $\NN=\{1,2,3,\ldots\}$ denote the set of positive integers.

\section{Kuranishi charts for moduli spaces of anti-self-dual connections}
\label{sec:Local_Kuranishi_parameterization_moduli_space_anti-self-dual_connections}

\subsection{Kuranishi chart around an interior point via orthogonal projection}
\label{subsec:Local_Kuranishi_parameterization_neighborhood_interior_point_Taubes_approach}
The construction of the \emph{Kuranishi chart} for an open neighborhood of an \emph{interior} point $[A]$ of the moduli space $M(P,g)$ in \eqref{eq:Moduli_space_asd_connections} is well known and goes back to Atiyah, Hitchin, and Singer \cite[Theorem 6.1]{AHS}, based on an idea of Kuranishi in the deformation of complex structures \cite{Kuranishi}, and described by Donaldson and Kronheimer \cite{DK}, Freed and Uhlenbeck \cite{FU}, and Friedman and Morgan \cite{FrM} in terms of the elliptic deformation complex \eqref{eq:Elliptic_deformation_complex_ASD_equation} for a smooth anti-self-dual connection $A$. The approach we describe here is modeled on that of Taubes \cite{TauIndef, TauFrame}. Because the map
\[
  \sA(P) \ni A \mapsto F_A^+ \in W^{1,p}(\wedge^+(T^*X)\otimes\ad P)
\]
defines a $C^1$ Fredholm section\footnote{The $C^1$ maps, sections, and manifolds discussed here are actually real analytic.} of the vector bundle
\[
 \fV(P) := \sA(P)\times_{\Aut(P)} W^{1,p}(\wedge^+(T^*X)\otimes\ad P),
\]
one can apply the Implicit Mapping Theorem for $C^1$ maps of $C^1$ Banach manifolds to provide
\begin{itemize}
\item a $\Stab_A$-equivariant $C^1$ embedding
\[
  \bgamma: \bH_A^1(X;\ad P) \supset \bO_A \ni \tau \mapsto A + \bgamma(\tau) \in A+\Ker d_A^*\cap W^{1,p}(T^*X\otimes\ad P)
\]
from a $\Stab_A$-invariant open neighborhood $\bO_A$ of the origin.

\item a $\Stab_A$-equivariant $C^1$ map
  \[
    \bpsi: \bH_A^1(X;\ad P) \supset \bO_A \ni \tau \to \bpsi(\tau) \in \bH_A^2(X;\ad P)
  \]
  such that
\[
  \bgamma: \bO_A\cap\bpsi^{-1}(0)/\Stab_A \to M(P,g)
\]
is a continuous embedding onto an open neighborhood of $[A]$ in $M(P,g)$, where $\Stab_A = \{u\in\Aut(P): u(A)=A\}$ is the stabilizer or isotropy subgroup\footnote{This subgroup of $\Aut(P)$ is isomorphic to the centralizer of the holonomy subgroup for $A$ in $G$.} for $A$ in $\Aut(P)$.
\end{itemize}  

This chart can be obtained by splitting the local defining equation, $F_{A+a}^+=0$ for $a \in \Ker d_A^*\cap W^{1,p}(T^*X\otimes\ad P)$, into two parts: For each $\tau\in \bO_A$, first solve for the unique $v=\wp(\tau) \in \Pi_A^\perp W^{2,p}(\wedge^+(T^*X)\otimes\ad P)$ such that for $a=d_A^{+,*}v$ one has
\[
  \Pi_A^\perp F_{A+\tau+a}^+ = 0.
\]
Indeed, we can solve for $v$ by using the map
\[
  \bar\bB_\delta^+(A) \ni v \mapsto \Pi_A^\perp F_{A+\tau+d_A^{+,*}v}^+ \in \Pi_A^\perp L^p(\wedge^+(T^*X)\otimes\ad P)
\]
to define a self map of $\bar\bB_\delta^+(A)$,
\[
\bar\bB_\delta^+(A) \ni v \mapsto G_AF_{A+\tau+d_A^{+,*}v}^+ \in \bar\bB_\delta^+(A),
\]
and hence solve a fixed point equation for $v$, where
\[
  \bB_\delta^+(A) := \left\{v \in \Pi_A^\perp W^{2,p}(\wedge^+(T^*X)\otimes\ad P): \|v\|_{W^{2,p}(X)} < \delta\right\}
\]
and 
\[
G_A: L^p(\wedge^+(T^*X)\otimes\ad P) \to W^{2,p}(\wedge^+(T^*X)\otimes\ad P)
\]
is the Green's operator for $d_A^+d_A^{+,*}$, so that
\begin{align*}
  d_A^+d_A^{+,*}G_A &= \Pi_A^\perp \quad\text{on } L^p(\wedge^+(T^*X)\otimes\ad P),
  \\
  G_Ad_A^+d_A^{+,*} &= \Pi_A^\perp \quad\text{on } W^{2,p}(\wedge^+(T^*X)\otimes\ad P).
\end{align*}
We then consider the subset of all $\tau \in \bO_A$ such that
\[
  \bpsi(\tau) := \Pi_A F_{A+\bgamma(\tau)}^+ = 0,
\]
where $\Pi_A$ is $L^2$-orthogonal projection from $L^2(\wedge^+(T^*X)\otimes\ad P)$ onto the finite-dimensional subspace $\Ker d_A^+d_A^{+,*}$, noting that $d_A^+d_A^{+,*}$ is an $L^2$-self-adjoint elliptic operator with discrete spectrum in $[0,\infty)$, and $\bgamma(\tau):=\tau+d_A^{+,*}\wp(\tau)$ for all $\tau\in\bO_A$.

In the preceding discussion, while one can say little about the zero set $\bO_A\cap\bpsi^{-1}(0)$, it is at least straightforward to apply the Implicit Mapping Theorem to reduce the problem of describing a local neighborhood of a point in $M(P,g)$ to one of describing the finite-dimensional local model $\bO_A\cap\bpsi^{-1}(0)/\Stab_A$. In this article, we shall therefore focus the majority of our attention on the case where the origin is a regular point of the map
\[
  W^{1,p}(T^*X\otimes\ad P) \ni a \mapsto F_{A+a}^+ \in L^p(\wedge^+(T^*X)\otimes\ad P),
\]
that is, when $H_A^2(X;\ad P)=0$, since the extension to the case $H_A^2(X;\ad P)\neq 0$ requires us only to replace the role of the preceding map by
\[
  W^{1,p}(T^*X\otimes\ad P) \ni a \mapsto \Pi_A^\perp F_{A+a}^+ \in \Pi_A^\perp L^p(\wedge^+(T^*X)\otimes\ad P).
\]
When $H_A^2(X;\ad P)=0$, one can obtain the \emph{gluing map} $\bgamma$ from an immediate application of Theorems \ref{mainthm:Preimage_point_under_submersion_and_implied_embedding} or \ref{mainthm:Preimage_submanifold_under_transverse_map_and_implied_embedding} in the special case of an abstract $C^1$ map of smooth Banach manifolds \emph{without} boundary, namely the Banach spaces
\[
  \Ker d_A^*\cap W^{1,p}(T^*X\otimes\ad P) \quad\text{and}\quad L^p(\wedge^+(T^*X)\otimes\ad P),
\]
and one follows this paradigm \mutatis to obtain the general case $H_A^2(X;\ad P)\neq 0$.

Our main observation in this article is that by replacing the preceding pair of Banach spaces with a pair of Banach manifolds with boundary, or Banach manifolds with corners more generally, one can again directly deduce the existence of the desired gluing map --- parameterizing a neighborhood of a point in the bubble tree compactification $\widehat M(P,g)$ of $M(P,g)$ --- from an application of the Inverse Mapping Theorem for $C^1$ maps of $C^1$ Banach manifolds with corners.

In this article, we focus on the problem of describing an open neighborhood of a point in $\widehat M(P,g)$ corresponding to formation of a single bubble point of curvature concentration and, for this purpose, it is enough to apply the Inverse Mapping Theorem for $C^1$ maps of smooth Banach manifolds with boundary in the shape of Theorems \ref{mainthm:Preimage_point_under_submersion_and_implied_embedding} or \ref{mainthm:Preimage_submanifold_under_transverse_map_and_implied_embedding}. In the general case of many bubble points of curvature concentration, one would have to apply versions of Theorems \ref{mainthm:Preimage_point_under_submersion_and_implied_embedding} or \ref{mainthm:Preimage_submanifold_under_transverse_map_and_implied_embedding} for maps of Banach manifolds with corners. However, such a generalization is purely technical and follows in a straightforward manner from methods described in this article and by Margalef Roig and Outerelo Dom{\'\i}nguez in \cite{Margalef-Roig_Outerelo-Dominguez_differential_topology} in their treatment of differential topology for Banach manifolds with corners. We refer to Ulyanov \cite{Ulyanov} for a manifolds with corners resolution of singularities for the symmetric products $\Sym^l(X)$ appearing in the Uhlenbeck compactification $\bar M(P,g)$ of $M(P,g)$.

\subsection{Kuranishi chart around an interior point via stabilization}
\label{subsec:Local_Kuranishi_parameterization_neighborhood_interior_point_Donaldson_approach}
There is a convenient alternative approach to constructing the Kuranishi model for a neighborhood of an anti-self-dual connection due to Donaldson (see Donaldson \cite{DonConn}, Donaldson and Kronheimer \cite[pp. 290--291]{DK}, or Donaldson and Sullivan \cite{Donaldson_Sullivan_1990}). While the Taubes--Kuranishi approach in Section \ref{subsec:Local_Kuranishi_parameterization_neighborhood_interior_point_Taubes_approach} seeks to replace the codomain $L^p(\wedge^+(T^*X)\otimes\ad P)$ by the smaller subspace
\[
  \Pi_A^\perp L^p(\wedge^+(T^*X)\otimes\ad P)
\]
such that the smooth composition $a\mapsto \Pi_A^\perp F_{A+a}^+$ is a submersion on an open neighborhood of the origin, where $\Ran \Pi_A = H_A^2(X;\ad P)$ and $\Pi_A^\perp = \id - \Pi_A$, in the Donaldson approach one instead replaces the domain $W^{1,p}(T^*X\otimes\ad P)$ by the larger space
\[
  W^{1,p}(T^*X\otimes\ad P)\oplus H_A^2(X;\ad P).
\]
When $\bH_A^2(X;\ad P) \neq 0$, then the smooth map
\[
  W^{1,p}(T^*X\otimes\ad P) \ni a \mapsto F_{A+a}^+ \in L^p(\wedge^+(T^*X)\otimes\ad P)
\]
has derivative $DF_{A+\cdot}^+(0) = d_A^+$ at the origin that is not surjective. We let
\[
  L_A:\bH_A^2(X;\ad P) \to L^p(\wedge^+(T^*X)\otimes\ad P)
\]
denote the natural inclusion map. Hence, the smooth map
\[
  W^{1,p}(T^*X\otimes\ad P)\oplus \bH_A^2(X;\ad P) \ni (a,v) \mapsto F_{A+a}^+ + L_Av \in L^p(\wedge^+(T^*X)\otimes\ad P)
\]
has derivative $d_A^+\oplus L_A$ at the origin that is, by construction, surjective. Therefore, an open neighborhood of the origin $(0,0)$ in the set
\[
  \left\{(a,v) \in \Ker d_A^*\cap W^{1,p}(T^*X\otimes\ad P)\oplus \bH_A^2(X;\ad P): F_{A+a}^+ + L_Av = 0 \right\}
\]
is an open smooth manifold of dimension equal to that of
\begin{multline*}
  \Ker(d_A^+\oplus L_A)\cap \left(\Ker d_A^*\cap W^{1,p}(T^*X\otimes\ad P)\oplus \bH_A^2(X;\ad P)\right)
  \\
  = \Ker d_A^+ \cap \Ker d_A^*\cap W^{1,p}(T^*X\otimes\ad P),
\end{multline*}
namely, $h_A^1(X;\ad P) = \dim\bH_A^1(X;\ad P)$. We now obtain a model for an open neighborhood in $M(P,g)$ of the point $[A]$ by cutting down the preceding set of pairs $(a,v)$ via the equation $L_Av=0$, that is, $v=0$.

\subsection{Excision principle for the index of an elliptic operator and gluing}
\label{subsec:Excision_principle_index_elliptic_operator_gluing}
The elliptic complex \eqref{eq:Elliptic_deformation_complex_ASD_equation} for an anti-self-dual connection $A$ on a principal $G$-bundle $P$ over $X$ may be rolled up in the standard way \cite[Section 1.5]{Gilkey2} to define a first-order elliptic operator
\begin{equation}
\label{eq:Elliptic_deformation_complex_ASD_equation_rolled_up}  
d_A^+ + d_A^*:\Omega^1(X;\ad P) \to \Omega^{2,+}(X;\ad P)\oplus \Omega^0(X;\ad P)
\end{equation}
with
\[
  \Ker(d_A^+ + d_A^*) = \bH_A^1(X;\ad P)
  \quad\text{and}\quad
  \Coker(d_A^+ + d_A^*) = \bH_A^2(X;\ad P)\oplus \bH_A^0(X;\ad P).
\]
The expected dimension of $M(P,g)$ (at a point $[A]$) is given by
\begin{align*}
  s_A(X;\ad P) &= \Ind(d_A^+ + d_A^*)
  \\
  &= \Ker(d_A^+ + d_A^*) - \Coker(d_A^+ + d_A^*)
  \\
  &= h_A^1(X;\ad P) - h_A^0(X;\ad P) - h_A^2(X;\ad P),
\end{align*}
just as in \eqref{eq:Index_elliptic_deformation_complex_ASD_equation}.

Suppose that we are given the following data required to construct a smooth principal bundle over a closed, connected, four-dimensional, oriented, smooth Riemannian manifold:

\begin{data}[Splicing data for a connected sum principal bundle over a connected sum Riemannian four-manifold]
\label{data:Splicing_parameters_connected_sum_G-bundle_connected_sum_Riemannian_4-manifold}  
Let $(X_0,g_0)$ and $(X_1,g_1)$ be admissible Riemannian four-manifolds; $G$ be a compact Lie group; $P_0$ and $P_1$ be smooth principal $G$-bundles over $X_0$ and $X_1$, respectively; $x_0 \in X$ and $x_1 \in X_1$ be points; $\lambda_0\in(0,1]$ be a scale parameter whose square root that is less than one quarter of the injectivity radii of $(X_i,g_i)$ for $i=0,1$; $v_0$ and $v_1$ be oriented, orthonormal frames for $TX_0|_{x_0}$ and $TX_1|_{x_1}$, respectively; $\rho \in \Isom_G(P_0|_{x_0}, P_1|_{x_1})$ be a bundle gluing parameter; and $A_{0\flat}$ on $P_0$ and $A_{1\flat}$ on $P_1$ be smooth connections.
\end{data}  

Varying the choice of bundle gluing parameter $\rho \in \Isom_G(P_0|_{x_0}, P_1|_{x_1})$ is equivalent to a choice of fiber points $p_0 \in P_0|_{x_0}$ and $p_1 \in P_1|_{x_1}$ and varying one of those points. As explained in \cite[Section 7.2.1]{DK}, the Data \ref{data:Splicing_parameters_connected_sum_G-bundle_connected_sum_Riemannian_4-manifold} can be used to define a closed, connected, oriented, smooth connected sum manifold $X=X_0\# X_1$, where $X_0$ and $X_1$ are joined by a small cylinder with cross section $S^3$, and a smooth connected sum principal $G$-bundle $P=P\#P_1$ over $X$. On the complement of the small cylinder, the Riemannian metric $g$ can be defined to agree with $g_0$ on $X_0$ and $g_1$ on $X_1$. 

Given anti-self-dual connections $A_0$ on $P_0$ and $A_1$ on $P_1$, one can use the Data \ref{data:Splicing_parameters_connected_sum_G-bundle_connected_sum_Riemannian_4-manifold} to form an approximately anti-self-dual connection $A$ on $P$ using the splicing method described in \cite[Section 7.2.1]{DK}. The excision principle for elliptic operators \cite{AS4}, \cite[Proposition 7.1.2]{DK} yields the following formula \cite[Equation (7.2.47)]{DK} for the expected dimension of the moduli space $M(P,g)$ (at a point $[A]$) in terms of the expected dimensions of the moduli spaces $M(P_0,g_0)$ and $M(P_1,g_1)$ (at points $[A_0]$ and $[A_1]$, respectively):
\begin{equation}
\label{eq:Excision_dimension_formula_ASD_elliptic_operator}
s_A(X;\ad P) = s_{A_0}(X_0;\ad P_0) + s_{A_1}(X_1;\ad P_1) + \dim G.  
\end{equation}
In the simplest case where $\bH_{A_0}^2(X_0;\ad P_0) = 0$ and $\bH_{A_1}^2(X_1;\ad P_1) = 0$ (no cokernel obstructions to deformation) and $\bH_{A_0}^0(X_0;\ad P_0) = 0$ and $\bH_{A_1}^0(X_1;\ad P_1) = 0$ (trivial isotropy groups), the dimension of $M(P,g)$ (at a point $[A]$) is given by
\begin{equation}
\label{eq:Excision_dimension_formula_ASD_elliptic_operator_regular_points_trivial_isotropy}
  h_A^1(X;\ad P) = h_{A_0}^1(X_0;\ad P_0) + h_{A_1}^1(X_1;\ad P_1) + \dim G.
\end{equation}
The preceding dimension formula can be used to help identify local coordinates for $M(P,g)$ near the boundary point defined by $\lambda=0$, as in \cite[Section 7.2.5]{DK}.

\subsection{Kuranishi chart around a boundary point}
\label{subsec:Local_Kuranishi_parameterization_neighborhood_boundary_point}
Let $P_1$ be a smooth principal $G$-bundle over the four-dimensional sphere, $S^4 = \{y\in\RR^5: |y|=1\}$, with its standard round metric $g_\round$ of radius one. In the simplest example of gluing a family of anti-self-dual connections on a principal $G$-bundle $P_1$ onto a family of anti-self-dual connections on a principal $G$-bundle $P_0$ over $(X,g)$, the essential idea is to use a splicing map $\cS$ to define a $C^1$ Banach manifold with boundary structure on an open subset of the infinite-dimensional quotient space $\sB(P)$ of $W^{1.p}$ connections ($p>2$) on the principal $G$-bundle $P$ obtained by splicing $P_0$ and $P_1$ over a small annulus $\Omega(x_0;\sqrt{\lambda}/4, 4\sqrt{\lambda})$ in $X$ defined by a small scale parameter $\lambda\in(0,1]$ and a point $x_0\in X$. The self-dual components of the curvatures of connections on $P$ define a section $F^+$ of a $C^1$ Banach vector bundle over $\sB(P)$ with fiber $L^p(\wedge^+(T^*X)\otimes\ad P)$ that extends to the $C^1$ Banach manifold with boundary, $\bar\sB(P)$. (For the purpose of this Introduction, we ignore the minor additional complication posed by the presence of points $[A]$ in $\sB(P)$ represented by connections $A$ with nontrivial isotropy in the Banach Lie group $\Aut(P)$ of $W^{2,p}$ automorphisms of $P$.) For compact Lie groups $G$, the moduli spaces $M(P_1)$ of anti-self-dual connections are nonempty by virtue of the construction due to Atiyah, Hitchin, Drinfel$'$d, and Manin \cite{AtiyahGeomYM, ADHM, Atiyah_Hitchin_Singer_1977, AHS}. We also assume that the moduli spaces $M(P_0)$ of anti-self-dual connections on $P_0$ are nonempty. A typical point in the boundary $\bar\sB(P)$, where $\lambda=0$, is represented by $(A_0,A_1,x_0,0)$, where $A_0$ is a connection on $P_0$ and $A_1$ is a connection on $P_1$ whose curvature density $|F_{A_1}|^2$ has center-of-mass at the north pole of $S^4$ and standard deviation one. If both $A_0$ and $A_1$ are anti-self-dual, then $[A_0,A_1,x_0,0]$ lies in the zero-locus $(F^+)^{-1}(0)\cap \bar\sB(P)$. The existence of an embedding (the `gluing map') from a finite-dimensional manifold of gluing data defined by open neighborhoods of $[A_0]$ in $M(P_0)$ and $[A_1]$ in $M(P_1)$ onto an open neighborhood of $[A_0,A_1,x_0,0]$ in the bubble-tree compactification $\bar M(P)$ of $M(P)$ now follows from a version of the Inverse Mapping Theorem for $C^1$ maps of smooth Banach manifolds with boundary, namely, Theorem \ref{mainthm:Preimage_point_under_submersion_and_implied_embedding} or \ref{mainthm:Preimage_submanifold_under_transverse_map_and_implied_embedding}. As noted earlier, when there are many bubble points, one would have to apply versions of Theorems \ref{mainthm:Preimage_point_under_submersion_and_implied_embedding} or \ref{mainthm:Preimage_submanifold_under_transverse_map_and_implied_embedding} for maps of Banach manifolds with corners. 

When we consider the proof of existence of a bubble tree compactification for $M(P,g)$, we shall need to restrict our attention to \emph{compact} Lie groups in order to apply Uhlenbeck's Weak Compactness Theorem \cite{UhlLp} or in order to take advantage of existence of anti-self-dual connections over $S^4$ \cite{AtiyahGeomYM, ADHM, Atiyah_Hitchin_Singer_1977, AHS}; until that stage, however, we may allow $G$ to be any Lie group.

\subsubsection{Lessons from the analysis of Donaldson's Collar Theorem}
One of the key calculations is to show that all partial derivatives of $\cF=F^+\circ\,\cS$ extend continuously from $\lambda\in (0,\lambda_0]$ up to $\lambda=0$, for example, as the limits as $\lambda \downarrow 0$ of partial derivatives that are defined when $\lambda>0$. This shows that $\cF$ is $C^1$ up to $\lambda=0$, as would be required by a boundary version of the Implicit Mapping Theorem. Recall that even greater boundary regularity known for Donaldson's Collar Map due to prior results of Groisser and Parker \cite{Groisser_1993, Groisser_1998, GroisserParkerSphere, GroisserParkerGeometryDefinite}.
See Donaldson \cite[Theorem 11]{DonApplic}, Freed and Uhlenbeck \cite[Chapter 9]{FU}, Groisser \cite{Groisser_1993, Groisser_1998} and Groisser and Parker \cite[Figure 1 and Theorems II, III, and IV]{GroisserParkerGeometryDefinite}.

\section{Differential topology for Banach manifolds with boundary}
\label{sec:Banach_manifolds_boundary}
In this section, we review essential concepts from differential topology for Banach manifolds with boundary, drawing heavily on the monograph by Margalef Roig and Outerelo Dom{\'\i}nguez \cite{Margalef-Roig_Outerelo-Dominguez_differential_topology}, and conclude with proofs of Theorems \ref{mainthm:Preimage_submanifold_under_transverse_map_and_implied_embedding} and \ref{mainthm:Preimage_point_under_submersion_and_implied_embedding}. Nice developments of some of these concepts for Banach manifolds without boundary are provided by Abraham, Marsden, and Ratiu \cite{AMR} and by Klingenberg \cite{Klingenberg_riemannian_geometry}. (While often cited as a reference differential topology for Banach manifolds without boundary, Lang \cite{Lang_introduction_differential_topology} is inaccurate in some respects, as we note below.)

\subsection{Preimage of a submanifold without boundary under a smooth map of manifolds without boundary}
We recall that if $f: X \to Y$ is a smooth map of \emph{finite}-dimensional smooth manifolds without boundary and $Z \subset Y$ is a smooth submanifold without boundary, then $f$ is \emph{transverse} to $Z$, denoted $f\transv Z$, if either $f^{-1}(Z) = \emptyset$ or
\[
  \Ran df(x) + T_{f(x)}Z = T_{f(x)}Y, \quad\forall x \in f^{-1}(Z).
\]
In particular, if $f\transv Z$, then $f^{-1}(Z)\subset X$ is a smooth manifold without boundary and
\[
  \codim f^{-1}(Z) = \codim Z.
\]
See Guillemin and Pollack \cite[Theorem, p. 28]{Guillemin_Pollack} or Hirsch \cite[p. 22 and Theorem 1.3.3]{Hirsch} for this statement of the \emph{Preimage Theorem}. When we pass to the setting of \emph{infinite}-dimensional Banach manifolds, however, the preceding definition of transversality requires refinement in order to yield the analogous statement of the Preimage Theorem.

\begin{defn}[Transversality for maps of Banach manifolds without boundary]
\label{defn:Transversality_all_Banach_manifolds_without_boundary_AMR}    
(See Abraham, Marsden, and Ratiu \cite[Definition 3.5.10]{AMR}.)
Let $f: X \to Y$ be a $C^p$ map ($p\geq 1$) of $C^p$ Banach manifolds without boundary and let $Z \subset Y$ be a $C^p$ Banach submanifold without boundary. Then $f$ is \emph{transverse to $Z$ at $x\in X$}, denoted $f\transv_x Z$, if either $f(x) \notin Z$ or if $f(x) \in Z$, then
\begin{enumerate}
\item
\label{item:Definition_transversality_all_Banach_manifolds_without_boundary_span}   
$\Ran df(x) + T_{f(x)}Z = T_{f(x)}Y$, and
\item
\label{item:Definition_transversality_all_Banach_manifolds_without_boundary_split}   
The subspace $(df(x))^{-1}(T_{f(x)}Z)$ has a closed complement\footnote{Lang \cite[p. 27 and Proposition 2.2.4]{Lang_introduction_differential_topology} is inaccurate here since he omits the condition on existence of a closed complement.} in $T_xX$. 
\end{enumerate}
If $f\transv_x Z$ for all $x \in X$, then $f$ is \emph{transverse to $Z$}, denoted $f\transv Z$.
\end{defn}

Condition \eqref{item:Definition_transversality_all_Banach_manifolds_without_boundary_span} in Definition \ref{defn:Transversality_all_Banach_manifolds_without_boundary_AMR} is purely algebraic; there is no assumption that $\Ran df(x)$ has a closed complement in $T_{f(x)}Y$. Condition \eqref{item:Definition_transversality_all_Banach_manifolds_without_boundary_split} in Definition \ref{defn:Transversality_all_Banach_manifolds_without_boundary_AMR} is automatic when $X$ is a Hilbert manifold or finite-dimensional.

\begin{thm}[Preimage of a Banach manifold without boundary under a map whose domain and codomain are Banach manifolds without boundary]
\label{thm:Preimage_theorem_all_manifolds_without_boundary_AMR}
(See Abraham, Marsden, and Ratiu \cite[Theorem 3.5.12]{AMR} or Margalef Roig and Outerelo Dom{\'\i}nguez \cite[Proposition 7.1.14]{Margalef-Roig_Outerelo-Dominguez_differential_topology}, which includes the case of Banach manifolds without boundary.)
Let $f: X \to Y$ be a $C^p$ map ($p\geq 1$) of $C^p$ Banach manifolds without boundary and let $Z \subset Y$ be a Banach submanifold without boundary. If $f \transv Z$, then $f^{-1}(Z)$ is a $C^p$ Banach submanifold without boundary, $T_x(f^{-1}(Z)) = (df(x))^{-1}(T_{f(x)}Z)$, and if $Z$ has finite codimension, then $\codim f^{-1}(Z) = \codim Z$.
\end{thm}

Margalef Roig and Outerelo Dom{\'\i}nguez further assert in \cite[Proposition 7.1.14]{Margalef-Roig_Outerelo-Dominguez_differential_topology} that the induced map
\[
df(x): T_xX/T_x(f^{-1}(Z)) \to T_{f(x)}Y/T_{f(x)}Z
\]
is an isomorphism of Banach spaces and use this isomorphism to conclude that codimension of submanifolds is preserved under pullback by maps that have the transversality property described in Definition \ref{defn:Transversality_all_Banach_manifolds_without_boundary_AMR}.

\subsection{Preimage of a submanifold without boundary under a map of a manifold with boundary into  a manifold without boundary}
When $X$ has boundary but $Y$ is without boundary, the following version of Theorem \ref{thm:Preimage_theorem_all_manifolds_without_boundary_AMR} is well-known in the case of finite-dimensional manifolds.

\begin{thm}[Preimage of a manifold without boundary by a map whose domain is a manifold with boundary and codomain is a manifold without boundary]
\label{thm:Preimage_theorem_domain_manifold_with_boundary}  
(See Guillemin and Pollack \cite[Theorem, p. 60]{Guillemin_Pollack} or Hirsch \cite[Theorem 1.4.2]{Hirsch}.)  
Let $X$ be a finite-dimensional smooth manifold with boundary, $Y$ be a finite-dimensional manifold without boundary, and $Z \subset Y$ be a submanifold without boundary. If $f: X \to Y$ is a smooth map such that $\mathring{f} \transv Z$ and $\partial f \transv Z$, where $\mathring{f} := f \restriction \Int(X)$ and $\partial f := f \restriction \partial X$, then the preimage $f^{-1}(Z)$ is a smooth manifold with boundary
\[
  \partial(f^{-1}(Z)) = f^{-1}(Z)\cap\partial X
\]
and $\codim f^{-1}(Z) = \codim Z$.
\end{thm}

Following Hirsch \cite[p. 30]{Hirsch}, one calls $W \subset X$ a \emph{neat} submanifold if $\partial W = W \cap \partial X$ and $W$ is covered by coordinate charts $(\varphi, U)$ for $X$ such that
\[
W \cap U = \varphi^{-1}(\RR^m),
\]
where $m=\dim W$. A \emph{neat} embedding is one whose image is a neat submanifold. See \cite[Figure 1.6]{Hirsch} for one illustration of a submanifold that is neat and two that are not. In general, $W$ is neat if and only if $\partial W = W\cap\partial X$ and $W$ is not tangent to $\partial X$ at any point of $x\in\partial W$, that is, $T_xW \not\subset T_x(\partial X)$ \cite[p. 31]{Hirsch}.

\subsection{Elementary examples}
When the domain and codomain of a smooth map are finite-dimensional manifolds, the presence of non-empty boundaries in the domain or codomain requires modifications in the Preimage Theorem \ref{thm:Preimage_theorem_all_manifolds_without_boundary_AMR} for manifolds \emph{without} boundary in order to give the nicest possible analogue for manifolds \emph{with} boundary. The elementary examples in this section illustrate some of the key considerations.

Let $f:X\to X'$ be a smooth map of finite-dimensional smooth manifolds with boundary and $X'' \subset X'$ be a smooth submanifold with boundary. We abbreviate $\partial f = f \restriction \partial X$ and $\mathring{f} = f \restriction\Int(X)$ as in Theorem \ref{thm:Preimage_theorem_domain_manifold_with_boundary}. If
\begin{align*}
  T_{f(x)}X' &= \Ran df(x) + T_{f(x)}X'', \quad\forall\, x \in f^{-1}(Y)\cap \Int(X),
  \\
  T_{f(x)}X' &= \Ran d(\partial f)(x) + T_{f(x)}X'', \quad\forall\, x \in f^{-1}(Y)\cap \partial X,
\end{align*}
then $f\transv X''$ by Definition ~\ref{defn:Transversality_maps_Banach_manifolds_boundary}. We abbreviate writing that $f$ is a submersion (so $X''$ is any point in $X'$) by $f \transv \pt$.

\begin{exmp}[Domain of $f_1$ is a half plane, codomain of $f_1$ is a plane, $\mathring{f_1}\transv \pt$, $\partial f_1\transv \pt$, and $Y$ is a submanifold of the codomain]
\label{exmp:X_half_plane_X_prime_plane_f_and_partial_f_transv_Y}  
Consider $f_1:\HH^3 \ni (x,y,z) \mapsto (x,y) \in \RR^2$, where $\HH^3 = \{(x,y,z)\in\RR^3: z \geq 0\}$ and $\partial\HH^3 =  \{(x,y,0):(x,y)\in\RR^2\}$, so that $\partial f_1: \partial\HH^3 \ni (x,y,0) \mapsto (x,y) \in \RR^2$. The maps $f_1$ and $\partial f_1$ preserve strata. 

Let $Y=\{(0,y):y\in\RR\} \subset \RR^2$, so $\partial Y=\emptyset$. We have $f_1^{-1}(Y) = \{(0,y,z): y \in \RR, z\geq 0\}$ and $(\partial f_1)^{-1}(Y) = \{(0,y,0):y\in\RR\}$. Observe that $\mathring{f_1}\transv \pt$ and $\partial f_1 \transv \pt$. Note that
\begin{multline*}
  \partial\{f_1^{-1}(Y)\} = \{(0,y,0): y\in\RR\} \quad\text{and}
  \\
  f_1^{-1}(Y) \cap \partial\HH^3 = \{(0,y,0): y\in\RR\} \cap \partial\HH^3 = \{(0,y,0): y\in\RR\},
\end{multline*}
giving
\[
    \partial\{f_1^{-1}(Y)\} = f_1^{-1}(Y) \cap \partial\HH^3.
\]
Moreover, $\codim (f_1^{-1}(Y);\HH^3) = 1 = \codim (Y;\HH^2)$. The conclusions agree with our expectation from Theorem \ref{thm:Preimage_theorem_domain_manifold_with_boundary}. \qed
\end{exmp}

\begin{exmp}[Domain of $f_2$ is a half plane, codomain of $f_2$ is a plane, $\mathring{f_2}\transv \pt$, and $\partial f_2\not\transv Y$]
\label{exmp:X_half_plane_X_prime_plane_f_transv_Y_but_partial_f_not_transv_Y} 
Consider $f_2:\HH^3 \ni (x,y,z) \mapsto (y,z) \in \RR^2$, where $\HH^3 = \{(x,y,z)\in\RR^3: z \geq 0\}$ and $\partial\HH^3 =  \{(x,y,0):(x,y)\in\RR^2\}$, so that $\partial f_2: \partial\HH^3 \ni (x,y,0) \mapsto (y,0) \in Y \subset \RR^2$. The maps $f_2$ and $\partial f_2$ preserve strata. 

Let $Y = \{(y,0):y\in\RR\} \subset \RR^2$, so $\partial Y=\emptyset$. We have $f_2^{-1}(Y) = \{(x,y,0):x,y\in\RR\}$ and $(\partial f_2)^{-1}(Y) = \{(x,y,0):x,y\in\RR\}$. Observe that $\mathring{f_2}\transv \pt$ but $\partial f_2 \not\transv Y$. Note that
\[
    \partial\{f_2^{-1}(Y)\} = \emptyset \quad\text{and}\quad f_2^{-1}(Y) \cap \partial\HH^3 = \{(x,y,0):x,y\in\RR\} \cap \partial\HH^3 = \partial\HH^3,
\]
giving
\[
    \partial\{f_2^{-1}(Y)\} \neq f_2^{-1}(Y) \cap \partial\HH^3.
\]
However, $\codim (f_2^{-1}(Y);\HH^3) = 1 = \codim (Y;\RR^2)$. The conclusions agree with our expectation from Theorem \ref{thm:Preimage_theorem_domain_manifold_with_boundary}. \qed
\end{exmp}

\begin{exmp}[Domain and codomain of $f_2$ are half planes, $\mathring{f_2} \transv \pt$, $\partial f_2\not\transv Y$, and $Y$ is not a neat submanifold of the codomain]
\label{exmp:X_and_X_prime_half_planes_f_transv_Y_but_partial_f_not_transv_Y_and_Y_not_neat}   
Consider $f_2:\HH^3 \ni (x,y,z) \mapsto (y,z) \in \HH^2$, where $\HH^3 = \{(x,y,z)\in\RR^3: z \geq 0\}$ and $\HH^2 = \{(y,z)\in\RR^2: z \geq 0\}$, so $\partial\HH^3 =  \{(x,y,0):(x,y)\in\RR^2\}$ and $\partial\HH^2 =  \{(y,0):y\in\RR\}$, and $\partial f_2: \partial\HH^3 \ni (x,y,0) \mapsto (y,0) \in \HH^2$. The maps $f_2$ and $\partial f_2$ preserve strata. 

Let $Y = \{(y,0):y\in\RR\} \subset \HH^2$, so $\partial Y = \emptyset$. We have $f_2^{-1}(Y) = \{(x,y,0):x,y\in\RR\}$ and $(\partial f_2)^{-1}(Y) = \{(x,y,0):x,y\in\RR\}$. Observe that $\mathring{f_2} \transv \pt$ but $\partial f_2 \not\transv Y$. Note that
\[
    \partial\{f_2^{-1}(Y)\} = \emptyset \quad\text{and}\quad f_2^{-1}(Y) \cap \partial\HH^3 = \{(x,y,0):x,y\in\RR\} \cap \partial\HH^3 = \partial\HH^3,
\]
giving
\[
    \partial\{f_2^{-1}(Y)\} \neq f_2^{-1}(Y) \cap \partial\HH^3.
\]
However, $\codim (f_2^{-1}(Y);\HH^3) = 1 = \codim (Y;\HH^2)$. The conclusions agree with our expectation from Theorem \ref{thm:Margalef-Roig_proposition_4-2-1}. \qed
\end{exmp}

\begin{exmp}[Domain and codomain of $f_2$ are half planes, $\mathring{f_2}\transv \pt$, $\partial f_2\transv \partial Z$, and $Z$ is a neat submanifold of the codomain]
Consider $f_2:\HH^3 \ni (x,y,z) \mapsto (y,z) \in \HH^2$, where $\HH^3 = \{(x,y,z)\in\RR^3: z \geq 0\}$ and $\HH^2 = \{(y,z)\in\RR^2: z \geq 0\}$, so $\partial\HH^3 =  \{(x,y,0):(x,y)\in\RR^2\}$ and $\partial\HH^2 =  \{(y,0):y\in\RR\}$, and $\partial f_2: \partial\HH^3 \ni (x,y,0) \mapsto (y,0) \in \partial\HH^2 \subset \HH^2$. The maps $f_2$ and $\partial f_2$ preserve strata. 

Let $Z = \{(0,z):z\geq 0\} \subset \HH^2$, so $\partial Z = \{(0,0)\} \subset \partial \HH^2$. We have $f_2^{-1}(Z) = \{(x,0,z):x\in\RR, z\geq 0\}$ and $(\partial f_2)^{-1}(Z) = \{(0,0,0)\}$. Observe that $\mathring{f_2}\transv \pt$ and $\partial f_2 \transv Z$ (since $T_{(0,0)} Z = z$-axis and $\Ran d(\partial f_2) = y$-axis). Note that
\begin{multline*}
    \partial\{f_2^{-1}(Z)\} = \{(x,0,0):x\in\RR\} \quad\text{and}
    \\
    f_2^{-1}(Z) \cap \partial\HH^3 = \{(x,0,z):x\in\RR, z\geq 0\} \cap \partial\HH^3 = \{(x,0,0):x\in\RR\},
\end{multline*}
giving
\[
    \partial\{f_2^{-1}(Z)\} = f_2^{-1}(Z) \cap \partial\HH^3.
\]
Moreover, $\codim (f_2^{-1}(Z);\HH^3) = 1 = \codim (Z;\HH^2)$. The conclusions agree with our expectation from Theorem \ref{thm:Margalef-Roig_proposition_4-2-1}. \qed
\end{exmp}

\begin{exmp}[Domain and codomain of $f_3$ are half planes, $\mathring{f_3}\transv Z$ and $\partial f_3 \transv Z$]
Consider $f_3:\HH^3 \ni (x,y,z) \mapsto (y,0) \in \HH^2$, where $\HH^3 = \{(x,y,z)\in\RR^3: z \geq 0\}$ and $\HH^2 = \{(y,z)\in\RR^2: z \geq 0\}$, so $\partial\HH^3 =  \{(x,y,0):(x,y)\in\RR^2\}$ and $\partial\HH^2 =  \{(y,0):y\in\RR\}$, and $\partial f_3: \partial\HH^3 \ni (x,y,0) \mapsto (y,0) \in \HH^2$. The maps $f_3$ and $\partial f_3$ preserve strata. 

Let $Z = \{(0,z):z\geq 0\} \subset \HH^2$, so $\partial Z = \{(0,0)\} \subset \partial \HH^2$. We have $f_3^{-1}(Z) = \{(x,0,0):x\in\RR\}$ and $(\partial f_3)^{-1}(Z) = \{(x,0,z):x\in\RR, z\geq 0\}$. Observe that $\mathring{f_3} \transv Z$ (since $T_{(0,z)} Z = z$-axis and $\Ran d\mathring{f_3} = y$-axis) and $\partial f_3 \transv \partial Z$ (since $T_{(0,0)} Z = z$-axis and $\Ran d(\partial f_3) = y$-axis). Note that
\begin{multline*}
  \partial\{f_3^{-1}(Z)\} = \{(x,0,0):x\in\RR\} \quad\text{and}
  \\
    f_3^{-1}(Z) \cap \partial\HH^3 = \{(x,0,0):x\in\RR\} \cap \partial\HH^3 = \{(x,0,0):x\in\RR\},
  \end{multline*}
giving
\[
    \partial\{f_3^{-1}(Z)\} = f_3^{-1}(Z) \cap \partial\HH^3.
\]
Moreover, $\codim (f_3^{-1}(Z);\HH^3) = 1 = \codim (Z;\HH^2)$. The conclusions agree with our expectation from Theorem \ref{thm:Margalef-Roig_proposition_4-2-1}. \qed
\end{exmp}

\subsection{Banach manifolds with boundary}
Margalef Roig and Outerelo Dom{\'\i}nguez define Banach manifolds with corners in \cite[Section 1.2]{Margalef-Roig_Outerelo-Dominguez_differential_topology}, based on their analysis of differentials of maps over open subsets of quadrants of
Banach spaces in \cite[Section 1.1]{Margalef-Roig_Outerelo-Dominguez_differential_topology}. We shall only need special cases of their results applying to Banach manifolds with boundary, based on the half planes in Banach spaces instead of the more general quadrants in Banach spaces employed in \cite{Margalef-Roig_Outerelo-Dominguez_differential_topology}. Joyce \cite{Joyce_2012} and Melrose \cite{Melrose_differential_analysis_manifolds_corners} provide complementary treatments of finite-dimensional manifolds with corners, but their emphasis are somewhat different to that of \cite{Margalef-Roig_Outerelo-Dominguez_differential_topology}, whose treatment directly addresses our need.

\subsubsection{Differential of maps over open sets of half planes of Banach spaces}
\label{subsubsec:Margalef-Roig_1-1}
If $E, F$ are real Banach spaces, we let $\sL(E,F)$ denote the Banach space of bounded, linear operators $u:E\to F$ with the operator norm
\[
  \|u\|_{\sL(E,F)} := \sup_{x\in E\less\{0\}} \frac{\|ux\|_F}{\|x\|_E}.
\]
If $F=\RR$, we let $E^*=\sL(E,\RR)$ denote the continuous dual space\footnote{When necessary to make a distinction, we let $L(E,F)$ denote the vector space of linear operators $u:E\to F$ and $E^\vee := L(E,\RR)$ denote the algebraic dual space of $E$.} of $E$.

\begin{defn}[Hyperplanes and half planes]
\label{defn:Margalef-Roig_1-1-1_half_plane}
(See Margalef Roig and Outerelo Dom{\'\i}nguez \cite[Definition 1.1.1]{Margalef-Roig_Outerelo-Dominguez_differential_topology}.)  
Let $E$ be a real Banach space and $\lambda \in E^* = \sL(E,\RR)$ be non-constant. We call $E_\lambda^0 := \{x\in E: \lambda(x) = 0\}$ a \emph{hyperplane} and $E_\lambda^+ := \{x\in E: \lambda(x) \geq 0\}$ a \emph{half plane}.
\end{defn}

We shall also find it convenient to denote the boundary and interior of $E_\lambda^+$ by
\begin{equation}
\label{eq:Boundary_and_interior_half_plane}
\partial E_\lambda^+ := E_\lambda^0 = \Ker\lambda \quad\text{and}\quad \Int(E_\lambda^+) = E_\lambda^+\less E_\lambda^0 = \{x\in E: \lambda(x) > 0\}.
\end{equation}
If $\mu:E\to\RR$ is another non-constant linear map and $E_\mu^+ = E_\lambda^+$, then there exists a number $c > 0$ such that $\lambda = c\mu$ \cite[Section 2.4]{Lang_introduction_differential_topology}. To see this, observe that
\[
  \Ker\lambda = \Ker\mu, 
\]
while $\dim E/\Ker\lambda = 1$ and so $E = \Ker\lambda\oplus F$, for some closed complement $F\subset E$ of real dimension one by \cite[Lemma 4.21 (b)]{Rudin}. Choose $x_0 \in F$ such that $\lambda(x_0)=1$, so $x_0 \in \Int(E_\lambda^+)$ and because $\Int(E_\lambda^+) = \Int(E_\mu^+)$, then $\mu(x_0)>0$ too. Define $\alpha \in E^*$ by
\[
  \alpha := \lambda - \frac{\lambda(x_0)}{\mu(x_0)}\mu
\]
and observe that $\Ker\alpha = \Ker\lambda$ while $\alpha(x_0)=0$ and so $\alpha\equiv 0$ on $F$ and hence $\alpha\equiv 0$ on $E$. Thus, we can take $c = \lambda(x_0)/\mu(x_0)$.

Note that if we allowed $\lambda=1$ in Definition \ref{defn:Margalef-Roig_1-1-1_half_plane}, then we would have $E_\lambda^0=\emptyset$ and $E_\lambda^+=E$; this choice will allow us to consider manifolds without boundary as special cases of manifolds with boundary. Conversely, if we had allowed $\lambda=0$ in Definition \ref{defn:Margalef-Roig_1-1-1_half_plane}, then we would have $E_\lambda^0=E$ and $E_\lambda^+=E$ and if $E_\mu^+=E$ for some $\mu\in E^*$, then we would necessarily also have $\mu=0$.

\begin{defn}[Derivative]
\label{defn:Margalef-Roig_1-1-6}  
(See Margalef Roig and Outerelo Dom{\'\i}nguez \cite[Definition 1.1.6]{Margalef-Roig_Outerelo-Dominguez_differential_topology}.)  
Let $E, F$ be real Banach spaces, $\lambda\in E^*$, and $U \subset E_\lambda^+$ be an open subset, and $x\in U$, and $f:U\to F$ be a map. If there exists $u \in \sL(E,F)$ such that
\[
\lim_{y\to x} \frac{\|(f(y)-f(x)-u(y-x)\|_F}{\|y-x\|_E} = 0,
\]
then $f$ is \emph{differentiable} at $x$, and $u$ is the \emph{derivative} of (or \emph{tangent} to) $f$ at $x$ and denoted by $Df(x) = f'(x)$. If $f$ is differentiable at every point $x\in U$, then $f$ is \emph{differentiable on $U$}.
\end{defn}

According to \cite[Proposition 1.1.5]{Margalef-Roig_Outerelo-Dominguez_differential_topology}, the derivative $u$ is unique. Definition \ref{defn:Margalef-Roig_1-1-6}, unlike that of Lang \cite[Section 2.4]{Lang_introduction_differential_topology}, does \emph{not} require any choice of extension of $f$ to some open neighborhood of $x$ in $E$. According to \cite[Proposition 1.1.13]{Margalef-Roig_Outerelo-Dominguez_differential_topology}, if $f$ is $p-1$ times differentiable on $U$ and $p$ times differentiable at $x$, where $p\geq 2$, then $D^pf(x) \in \sL^p(E,F) = \sL(\otimes^p E,F)$ is a $p$-linear, continuous, and symmetric operator. We let $\sL_\sym^p(E,F) \subset \sL^p(E,F)$ denote the closed subspace of $p$-linear, continuous, and symmetric operators. One says that $f$ is \emph{map of class $p$} (or a \emph{$C^p$ map}) if $f$ is $p$ times differentiable on $U$ and the map $D^pf: U \to \sL_\sym^p(E,F)$ is continuous \cite[Definition 1.1.14]{Margalef-Roig_Outerelo-Dominguez_differential_topology}.

\subsubsection{Differentiable manifolds with boundary}
\label{subsubsec:Margalef-Roig_1-2}
Let $X$ be a set and\footnote{We write $\NN=\{0,1,2,3,\ldots\}$ for the set of natural numbers including zero.} $p\in\NN\cup\{\infty\}$. Following \cite[Section 1.1.2]{Margalef-Roig_Outerelo-Dominguez_differential_topology}, one says that $(U,\varphi,(E,\lambda))$ is a \emph{chart} for $X$ if the following hold: $U$ is a subset of $X$, and $E$ is a real Banach space, $\lambda\in E^*$, and $\varphi:U\to E_\lambda^+$ is an injective map, and $\varphi(U)$ is an open subset of $E_\lambda^+$. One calls two charts $(U,\varphi,(E,\lambda))$ and $(U',\varphi',(E',\lambda'))$ \emph{compatible of class $p$} (or \emph{$C^p$ compatible}) if $\varphi(U\cap U')$ and $\varphi'(U\cap U')$ are open subsets of $E_\lambda^+$ and ${E'}_{\lambda'}^{+}$, respectively, and the maps
\[
  \varphi'\circ\varphi^{-1}:\varphi(U\cap U') \to \varphi'(U\cap U')
  \quad\text{and}\quad
  \varphi\circ{\varphi'}^{-1}:\varphi'(U\cap U') \to \varphi(U\cap U')
\]
are $C^p$ (and hence homeomorphisms). A set $\sA$ of charts for $X$ is called an \emph{atlas of class $p$ on $X$} if the domains of the charts cover $X$ and any two of them are $C^p$ compatible. According to \cite[Definition 1.2.2]{Margalef-Roig_Outerelo-Dominguez_differential_topology}, the equivalence class $[\sA]$ defined by an atlas $\sA$ is called a \emph{differentiable structure of class $p$ on $X$} and the pair $(X,[\sA])$ is called a \emph{differentiable manifold of class $p$} (or a \emph{$C^p$ Banach manifold}), usually denoted simply by $X$. By \cite[Proposition 1.2.3]{Margalef-Roig_Outerelo-Dominguez_differential_topology}, the set $\cB=\{U: U \text{ is a domain of a chart for } X\}$ is a basis for a topology on $X$.

\begin{rmk}[Topological properties of Banach manifolds with boundary]
\label{rmk:Corollary_Margalef-Roig_1-4-12}  
The topologies of the Banach manifolds that we encounter in our applications to gauge theory will generally have additional features. For example, they are typically regular Hausdorff spaces, paracompact, and modelled on separable real Banach spaces. Such manifolds are metrizable according to \cite[Corollary 1.4.12]{Margalef-Roig_Outerelo-Dominguez_differential_topology}. 
\end{rmk}  

Let $E$ be a real Banach space, $\lambda\in E^*$, and $U$ be an open subset of $E_\lambda^+$. Following \cite[Definition 1.2.6]{Margalef-Roig_Outerelo-Dominguez_differential_topology}, we call $\partial U = \partial_\lambda U = \{x\in U:\lambda(x)=0\}$ the \emph{$\lambda$-boundary} of $U$ and call $\Int(U) = \Int_\lambda(U) = \{x\in U:\lambda(x)>0\}$ the \emph{$\lambda$-interior} of $U$.

One needs to prove that the boundary is preserved by diffeomorphisms. If $E,F$ are real Banach spaces, $\lambda\in E^*$ and $\mu\in F^*$ are non-constant, and $U\subset E_\lambda^+$ and $V \subset F_\mu^+$ are open subsets, we recall \cite[Definition 1.2.9]{Margalef-Roig_Outerelo-Dominguez_differential_topology} that a map $f:U\to V$ is a \emph{diffeomorphism of class $p$} (or a $C^p$ \emph{diffeomorphism}) if it is bijective and both $f$ and $f^{-1}:V\to U$ are of class $p$.

\begin{thm}[Boundary invariance]
\label{thm:Margalef-Roig_1-2-12}
(See Margalef Roig and Outerelo Dom{\'\i}nguez \cite[Theorem 1.2.12]{Margalef-Roig_Outerelo-Dominguez_differential_topology}.)  
Let $E,F$ be real Banach spaces, $\lambda\in E^*$ and $\mu\in F^*$ be non-constant, $U$ be an open subset of $E_\lambda^+$, and $V$ be an open subset of $F_\mu^+$, and $f:U \to V$ be a $C^p$ diffeomorphism ($p\geq 1$). Then
\[
  f(\Int_\lambda(U)) = \Int_\mu(V) \quad\text{and}\quad f(\partial_\lambda U) = \partial_\mu V
\]
while
\[
  f \restriction \Int_\lambda(U) :\Int_\lambda(U) \to \Int_\mu(V)
\]
is a $C^p$ diffeomorphism.
\end{thm}

Theorem \ref{thm:Margalef-Roig_1-2-12} yields the

\begin{prop}[Boundary and interior of a manifold]
\label{prop:Margalef-Roig_1-2-13}
(See Margalef Roig and Outerelo Dom{\'\i}nguez \cite[Proposition 1.2.13]{Margalef-Roig_Outerelo-Dominguez_differential_topology}.)  
If $X$ is a differentiable manifold of class $p\geq 1$ and $x\in X$ is a point and $(U,\varphi,(E,\lambda))$ and $(V,\psi,(F,\mu))$ are charts for $X$ such that $x\in U\cap V$, then $\varphi(x) \in \partial_\lambda\varphi(U) \iff \psi(x) \in \partial_\mu\psi(V)$ and $\varphi(x) \in \Int_\lambda(\varphi(U)) \iff \psi(x) \in \Int_\mu(\psi(V))$.
\end{prop}

Hence, by virtue of Proposition \ref{prop:Margalef-Roig_1-2-13} one can make the

\begin{defn}[Boundary and interior of a $C^p$ Banach manifold]
\label{defn:Margalef-Roig_1-2-14_and_16}
(See Margalef Roig and Outerelo Dom{\'\i}nguez \cite[Definitions 1.2.14 and 1.2.16]{Margalef-Roig_Outerelo-Dominguez_differential_topology}.)  
Let $X$ be a differentiable manifold of class $p\geq 1$. Then $\partial X := \{x\in X: \varphi(x) \in \partial_\lambda\varphi(U) \text{ for some chart } (U,\varphi,(E,\lambda))\}$ is called the \emph{boundary of $X$} while $\Int(X) := \{x\in X: \varphi(x) \in \Int_\lambda(\varphi(U)) \text{ for some chart } (U,\varphi,(E,\lambda))\}$ is called the \emph{interior of $X$}.
\end{defn}

\begin{rmk}[Manifolds without boundary as a special case of manifolds with boundary]
\label{rmk:Manifolds_with_and_without_boundary}  
As usual, the definition of manifold with boundary subsumes that of a manifold without boundary by taking $\lambda=1$, as noted following Definition \ref{defn:Margalef-Roig_1-1-1_half_plane}.
\end{rmk}  

\begin{prop}[Differentiable structure of class $p$ on the boundary and interior of a $C^p$ Banach manifold]
\label{prop:Margalef-Roig_1-2-18_and_19}
(See Margalef Roig and Outerelo Dom{\'\i}nguez \cite[Proposition 1.2.18 and Corollary 1.2.19]{Margalef-Roig_Outerelo-Dominguez_differential_topology}.)  
Let $X$ be a differentiable manifold of class $p\geq 1$. Then the following hold:
\begin{enumerate}
\item There is a unique differentiable structure on $\Int(X)$ such that for all $x \in \Int(X)$ and all charts $(U,\varphi,(E,\lambda))$ for $X$
  with $x\in U$ and $\varphi(x)=0$, the triplet $(U,\varphi,E)$ is a chart for $\Int(X)$. Also, $\Int(X)$ has no boundary and its topology is the topology induced by $X$.
\item There is a unique differentiable structure on $\partial X$ such that for all $x \in \partial X$ and all charts $(U,\varphi,(E,\lambda))$ for $X$ with $x\in U$ and $\varphi(x)=0$, the triplet $(U\cap\partial X,\varphi \restriction U\cap\partial X, E_\lambda^0)$ is a chart for $\partial X$. Also, $\partial X$ has no boundary and its topology is the topology induced by $X$.  
\end{enumerate}
\end{prop}

\subsubsection{Differentiable maps}
\label{subsubsec:Margalef-Roig_1-3}
We begin with the

\begin{defn}[Maps of class $p$]
\label{defn:Margalef-Roig_1-3-2}
(See Margalef Roig and Outerelo Dom{\'\i}nguez \cite[Definition 1.3.2]{Margalef-Roig_Outerelo-Dominguez_differential_topology}.)  
Let $X$ and $X'$ be differentiable manifolds of class $p\geq 1$. We say that $f:X\to X'$ is a \emph{map of class $p$} or a \emph{$C^p$ map} if for every $x\in X$ there are a chart $(U,\varphi,(E,\lambda))$ for $X$ at $x$ and a chart $(V,\psi,(F,\mu))$ of $X'$ at $f(x)$ such that $f(U) \subset V$ and the map
\[
  \psi\circ f\circ\varphi^{-1}:\varphi(U) \to \psi(V)
\]
is a map of class $p$.
\end{defn}

One can show \cite[p. 36]{Margalef-Roig_Outerelo-Dominguez_differential_topology} that every $C^p$ map ($p\geq 1$) is necessarily a continuous map. The map $f$ in Definition \ref{defn:Margalef-Roig_1-3-2} is a \emph{diffeomorphism of class $p$} (or $C^p$ \emph{diffeomorphism}) if $f$ is bijective and $f^{-1}$ is also a map of class $p$ \cite[Definition 1.3.4]{Margalef-Roig_Outerelo-Dominguez_differential_topology}.

\begin{thm}[Boundary invariance]
\label{thm:Margalef-Roig_1-3-6}
(See Margalef Roig and Outerelo Dom{\'\i}nguez \cite[Theorem 1.3.6]{Margalef-Roig_Outerelo-Dominguez_differential_topology}.)
Let $X$ and $X'$ be differentiable manifolds of class $p\geq 1$ and $f:X\to X'$ be a diffeomorphism of class $p$. Then $f(\partial X) = \partial X'$ and $f(\Int(X)) = \Int(X')$.
\end{thm}

Moreover, in the setting of Theorem \ref{thm:Margalef-Roig_1-3-6} we recall by \cite[Proposition 1.3.7]{Margalef-Roig_Outerelo-Dominguez_differential_topology} that $f\restriction \partial X: \partial X \to \partial X'$ and $f\restriction \Int(X): \Int(X) \to \Int(X')$ are diffeomorphisms of class $p$.

\subsubsection{Tangent spaces and tangent bundles}
Let $X$ be a $C^p$ Banach manifold ($p\geq 1$) and $x\in X$. Let $C_xX := \{(c,v):c=(U,\varphi,(E,\lambda))$ is a chart for $X$ at $x$ and $v\in E\}$ and consider the binary relation $\sim$ on $C_xX$ defined by
\[
  (c,v) \sim (c',v') \iff D(\varphi'\circ\varphi^{-1})(\varphi(x))v = v'.
\]
According to \cite[Proposition 1.6.1]{Margalef-Roig_Outerelo-Dominguez_differential_topology}, this binary relation is an equivalence relation on $C_xX$ and one denotes
\[
  T_xX := C_xX/\sim
\]
and denotes the equivalence class of $(c,v)$ by $[c,v]$.

\begin{prop}[Tangent space of a differentiable manifold at a point]
\label{prop:Margalef-Roig_1-6-2}
(See Margalef Roig and Outerelo Dom{\'\i}nguez \cite[Proposition 1.6.2]{Margalef-Roig_Outerelo-Dominguez_differential_topology}.) 
Let $X$ be a $C^p$ Banach manifold ($p\geq 1$) and $x\in X$. Then the following hold:
\begin{enumerate}
\item For every chart $c=(U,\varphi,(E,\lambda))$ for $X$ at $x$, the map $\sO_c^x:E \to T_xX$
defined by $O_c^x(v) = [c,v]$ is bijective.
\item There is a unique structure of a real topological vector space on $T_xX$ such that for every chart $c=(U,\varphi,(E,\lambda))$ for $X$ with $x\in U$, the map $\sO_c^x:E\to T_xX$ is a linear homeomorphism and this structure is Banachable.
\item If $c=(U,\varphi,(E,\lambda))$ and $c'=(U',\varphi',(E',\lambda'))$ are charts for $X$ at the point $x$, then
\[
(\sO_{c'}^x)^{-1}\circ\sO_c^x = D(\varphi'\circ\varphi)^{-1}(\varphi(x)).
\]
\end{enumerate}
\end{prop}

The real Banachable space $T_xX$ is the \emph{tangent space} of $X$ at $x$ and the elements of $T_xX$ are \emph{tangent vectors} of $X$ at $x$. It is important to remember that $T_xX \cong E$ (as Banach spaces) irrespective of whether $x\in\Int(X)$ or $x\in\partial X$.

\begin{prop}[Tangent space of a differentiable manifold at a point]
\label{prop:Margalef-Roig_1-6-3}
(See Margalef Roig and Outerelo Dom{\'\i}nguez \cite[Proposition 1.6.3]{Margalef-Roig_Outerelo-Dominguez_differential_topology}.) 
If $X$ and $X'$ are $C^p$ Banach manifolds ($p\geq 1$) and $f:X\to X'$ is a $C^p$ map and $x\in X$, then there is a unique continuous linear map,
\[
  df(x) = T_xf:T_xX \to T_{f(x)}X',
\]
such that for every chart $c = (U,\varphi,(E,\lambda))$ for $X$ at $x$ and every chart $c' = (U',\varphi',(E',\lambda'))$ for $X'$ at $f(x)$ one has
\[
  df(x) = \sO_{c'}^{f(x)} \circ D(\varphi'\circ\varphi)^{-1}(\varphi(x)) \circ (\sO_c^x)^{-1}.
\]
\end{prop}

The map $ df(x)$ is called the \emph{derivative} or \emph{tangent map} for $f$ at the point $x$ \cite[p. 63]{Margalef-Roig_Outerelo-Dominguez_differential_topology}. The set $TX = \cup_{x\in X} T_xX$ denotes the \emph{tangent bundle} for $X$, with projection map $\tau_X: TX \to X$ defined by $(x,v) \mapsto x$ for all $(x,v) \in TX$ \cite[p. 66]{Margalef-Roig_Outerelo-Dominguez_differential_topology}. 

\begin{prop}[Properties of the tangent bundle]
\label{prop:Margalef-Roig_1-6-9}
(See Margalef Roig and Outerelo Dom{\'\i}nguez \cite[Proposition 1.6.9]{Margalef-Roig_Outerelo-Dominguez_differential_topology}.) 
If $X$ is a $C^p$ Banach manifold ($p\geq 1$), then $TX$ has a unique structure as a $C^{p-1}$ manifold and the following hold:
\begin{enumerate}
\item The projection $\tau_X:TX\to X$ is a $C^{p-1}$ map.
\item If $p\geq 2$, then for all $(x,v) \in TX$,
  \[
    (x,v) \in \Int(TX) \iff x\in\Int(X) \quad\text{and}\quad (x,v) \in \partial(TX) \iff x\in\partial X. 
  \]
\end{enumerate}
\end{prop}

If $f:X\to X'$ is a $C^p$ map, then the map $df:TX \to TX'$ given by $(x,v) \mapsto (f(x),df(x)v)$ is of class $C^{p-1}$ \cite[Proposition 1.6.10]{Margalef-Roig_Outerelo-Dominguez_differential_topology}.

Tangent vectors to $X$ at a point $x$ may be equivalently described in terms of curves passing through $x$ \cite[pp. 67--68]{Margalef-Roig_Outerelo-Dominguez_differential_topology}. If $\alpha: [0,a)\to X$ is a $C^1$ map such that $\alpha(0)=x$, then the element of $T_xX$ defined by
\[
  d\alpha(0)\circ \sO_{c_0}^0(1)
\]
is called the \emph{tangent vector} to $\alpha$ at the point $0$ and is denoted by $\dot\alpha(0)$, where $c_0 = ([0,a), \iota, (\RR, 1_\RR))$ and $\iota:[0,a)\to\RR$ is the inclusion map. If $\beta: (b,0]\to X$ is a $C^1$ map such that $\beta(0)=x$, then the element of $T_xX$ defined by
\[
  d\beta(0)\circ \sO_{c_0}^0(1)
\]
is called the \emph{tangent vector} to $\beta$ at the point $0$ and is denoted by $\dot\beta(0)$, where $c_0 = ((b,0], \iota, (\RR, -1_\RR))$ and $\iota:(b,0]\to\RR$ is the inclusion map. If $\gamma: (c,d)\to X$ is a $C^1$ map such that $\gamma(0)=x$, then $\dot\gamma(0)$ is similarly defined.

If $\alpha: [0,a)\to X$ (respectively, $\alpha: (b, 0]\to X$) is a $C^1$ map such that $\alpha(0)=x$ and $\dot\alpha(0)=v$, then $v$ is called an \emph{inner} (respectively, \emph{outer}) tangent vector at $x$. The set of the inner tangent vectors at $x$ is denoted by $(T_xX)^i$ and the set of the outer tangent vectors at $x$ is denoted by $(T_xX)^o$.

\begin{prop}[Inner and outer tangent vectors]
\label{prop:Margalef-Roig_1-6-12_and_13}
(See Margalef Roig and Outerelo Dom{\'\i}nguez \cite[Propositions 1.6.12 and 1.6.13]{Margalef-Roig_Outerelo-Dominguez_differential_topology}.) 
If $X$ is a $C^p$ Banach manifold ($p\geq 1$) and $x\in X$, then the following hold:
\begin{enumerate}
\item $(T_xX)^i = -(T_xX)^o$.
\item $T_xX = \Span((T_xX)^i) = \Span((T_xX)^o)$.
\item $\sO_c^x(E_\lambda^+) = (T_xX)^i$ if $c = (U,\varphi,(E,\lambda))$ is a chart for $X$ at $x$.  
\end{enumerate}  
\end{prop}

\begin{prop}[Strictly inner and outer tangent vectors]
\label{prop:Margalef-Roig_1-6-15}
(See Margalef Roig and Outerelo Dom{\'\i}nguez \cite[Proposition 1.6.15]{Margalef-Roig_Outerelo-Dominguez_differential_topology}.) 
If $X$ is a $C^p$ Banach manifold ($p\geq 1$) and $x\in X$ and $c = (U,\varphi,(E,\lambda))$ and $c' = (U',\varphi',(E',\lambda'))$ are charts for $X$ such that $x\in U\cap U'$ and $\varphi(x)=0=\varphi'(x)$, then
\[
  \sO_c^x(\Int(E_\lambda^+)) = \sO_{c'}^x(\Int(E_{\lambda'}'^+)) \subset (T_xX)^i.
\]
\end{prop}

The elements of $\sO_c^x(\Int(E_\lambda^+))$ are called \emph{strictly inner tangent vectors} at $x$ and the elements of $-\sO_c^x(\Int(E_\lambda^+))$ are called \emph{strictly outer tangent vectors} at $x$. (If $x\in\Int(X)$, then all tangent vectors in $T_xX$ are both strictly inner and strictly outer.)

\begin{prop}[Characterization if inner and strictly inner tangent vectors]
\label{prop:Margalef-Roig_1-6-16}
(See Margalef Roig and Outerelo Dom{\'\i}nguez \cite[Proposition 1.6.16]{Margalef-Roig_Outerelo-Dominguez_differential_topology}.) 
If $X$ is a $C^p$ Banach manifold ($p\geq 1$) and $x\in X$ and $c = (U,\varphi,(E,\lambda))$ is a chart for $X$ at $x$ and $v\in T_xX$, then:
\begin{enumerate}
\item $v \text{ is inner } \iff \lambda(\sO_c^x)^{-1}(v)) \geq 0$; 
\item $v \text{ is strictly inner } \iff \lambda(\sO_c^x)^{-1}(v)) > 0$.
\end{enumerate}
\end{prop}

\begin{prop}[Smooth maps and inner tangent vectors]
\label{prop:Margalef-Roig_1-6-17}
(See Margalef Roig and Outerelo Dom{\'\i}nguez \cite[Proposition 1.6.17]{Margalef-Roig_Outerelo-Dominguez_differential_topology}.) 
If $f:X\to X'$ is a $C^p$ map ($p\geq 1$) of $C^p$ Banach manifolds and $x\in X$, then $df(x)((T_xX)^i) \subset (T_{f(x)}X')^i$.
\end{prop}

\subsection{Inverse mapping theorem for manifolds with boundary}
The essential ingredient that is required to extend transversality results for maps of manifolds without boundary to maps of manifolds with boundary is the inverse mapping theorem for manifolds with boundary.

\begin{thm}[Inverse mapping theorem for maps of open subsets of half planes]
\label{thm:Margalef-Roig_2-2-4}
(See Margalef Roig and Outerelo Dom{\'\i}nguez \cite[Theorem 2.2.4]{Margalef-Roig_Outerelo-Dominguez_differential_topology}.)  
Let $E$ and $F$ be real Banach spaces, $\lambda \in E^*$ and $\mu \in F^*$, and $U \subset E_\lambda^+$ be an open subset, $f:U\to F_\mu^+$ be a $C^p$ map ($p\geq 1$) such that $f(\partial_\lambda U) \subset \partial F_\mu^+$, and $x \in U$ be a point. Then the following are equivalent:
\begin{enumerate}
\item $Df(x):E\to F$ is a linear homeomorphism,
\item There exist an open neighborhood $U_1 \subset U$ of $x$ and an open neighbourhood $V\subset F_\mu^+$ of $f(x)$ such that $f$ is a $C^p$ diffeomorphism from $U_1$ onto $V$.
\end{enumerate}
\end{thm}

\begin{rmk}[Extension of smooth functions on arbitrary subsets]
\label{rmk:Extension_smooth_functions_arbitrary_subsets}  
If $M$ and $N$ are smooth manifolds with or without boundary, and $E \subset M$ is an arbitrary subset, we recall (see Lee \cite[p. 45]{Lee_john_smooth_manifolds}) that a map $f:E\to N$ is said to be \emph{smooth on $E$ } if it has a smooth extension in a neighborhood of each point: that is, if for every $x \in E$ there is an open subset $W \subset M$ containing $x$ and a smooth map $\tilde f:W\to N$ such that $\tilde{f} \restriction W\cap E = f\restriction W\cap E$. One can prove that that if the set $E$ is contained in the closure of its interior, then the derivatives $df(x)$, for each $x\in E$, are uniquely defined by the extensions.
Deep results due to Whitney \cite{Whitney_1934tams_b, Whitney_1936tams} provide sufficient conditions on $f$ for the existence of smooth extensions; see Fefferman \cite{Fefferman_2005} for more recent and sharper forms of Whitney's extension theorems. 
\end{rmk}

\begin{defn}[Local diffeomorphisms of manifolds with boundary]
\label{defn:Margalef-Roig_2-2-5}
(See Margalef Roig and Outerelo Dom{\'\i}nguez \cite[Definition 2.2.5]{Margalef-Roig_Outerelo-Dominguez_differential_topology}.)  
Let $X$ and $X'$ be differentiable manifolds of class $p$ and $f:X\to X'$ be a map.
\begin{enumerate}
\item $f$ is a \emph{local diffeomorphism of class $p$ at $x_0\in X$} if there are open neighbourhoods $V_{x_0}\subset X$ of $x_0$ and $V_{f(x_0)}\subset X'$ of $f(x_0)$ such that $f$ is a diffeomorphism of class $p$ from $V_{x_0}$ onto $V_{f(x_0)}$.
  
\item $f$ is a local diffeomorphism of class $p$ from $X$ to $X'$ if it is a local diffeomorphism of class $p$ at every $x\in X$.
\end{enumerate}
\end{defn}

Any local diffeomorphism of class $p$ is necessarily a map of class $p$.

\begin{thm}[Inverse mapping theorem for maps of manifolds with boundary]
\label{thm:Margalef-Roig_2-2-6}
(See Margalef Roig and Outerelo Dom{\'\i}nguez \cite[Theorem 2.2.6]{Margalef-Roig_Outerelo-Dominguez_differential_topology}.)
Let $X$ and $X'$ be differentiable manifolds of class $p$ and $f:X\to X'$ be a $C^p$ map ($p\geq 1$) and $x_0\in X$ be a point. Then the following are equivalent:
\begin{enumerate}
\item $df(x_0)$ is a linear homeomorphism and there is an open neighborhood $V_{x_0}\subset X$ of $x_0$ such that $f(V_{x_0}\cap \partial X) \subset \partial X'$.
\item $f$ is a local diffeomorphism of class $p$ at $x_0$.
\end{enumerate}
\end{thm}

\subsection{Submanifolds, immersions, and embeddings of manifolds with boundary}

\subsubsection{Submanifolds of manifolds with boundary}

\begin{defn}[Submanifold of a manifold with boundary and adapted chart]
\label{defn:Margalef-Roig_3-1-1_and_3-1-2}
(See Margalef Roig and Outerelo Dom{\'\i}nguez \cite[Definitions 3.1.1 and 3.1.2]{Margalef-Roig_Outerelo-Dominguez_differential_topology}.)  
Let $X$ be a $C^p$ Banach manifold with boundary ($p\geq 1$) and $X'\subset X$ be a subset. Then $X'$ is a $C^p$ \emph{submanifold} of $X$ if for every $x’ \in X'$ there are a chart $c=(U,\phi,E,\lambda)$ for $X$ with $x’ \in U$ and $\phi(x’)=0$, a closed linear subspace $F\subset E$ that admits a closed complement in $E$, and $\mu \in F^*$ such that $\phi(U\cap X') = \phi(U)\cap F_\mu^+$ and is an open subset of $F_\mu^+$. Moreover, $c=(U,\phi,E,\lambda)$ is called a \emph{chart adapted to $X'$ at $x'$ through $(F,\mu)$}.
\end{defn}

\begin{prop}[Adapted charts]
\label{prop:Margalef-Roig_proposition_3-1-3}
(See Margalef Roig and Outerelo Dom{\'\i}nguez \cite[Proposition 3.1.3]{Margalef-Roig_Outerelo-Dominguez_differential_topology}.)  
Let $X$ be a $C^p$ Banach manifold with boundary ($p\geq 1$), and $X'\subset X$ be a subset and $x'\in X'$ be a point, $c=(U,\phi,E,\lambda)$ be a chart for $X$ with $\phi(x’)=0$, and $F\subset E$ be a closed linear subspace that admits a closed complement in $E$, and $\mu \in F^*$. Then the following are equivalent:
\begin{enumerate}
\item\label{item:Margalef-Roig_proposition_3-1-3_a}
 $c=(U,\phi,E,\lambda)$ is a chart adapted to $X'$ at $x'$ through $(F,\mu)$.  
\item\label{item:Margalef-Roig_proposition_3-1-3_b}
$\phi(U\cap X') = \phi(U)\cap F_\mu^+$ and $F_\mu^+ \subset E_\lambda^+$. 
\end{enumerate}  
\end{prop}

We remark that the assertion $F_\mu^+ \subset E_\lambda^+$ in Item \eqref{item:Margalef-Roig_proposition_3-1-3_b} of Proposition \ref{prop:Margalef-Roig_proposition_3-1-3} is not part of the Definition \ref{defn:Margalef-Roig_3-1-1_and_3-1-2} of an adapted chart.

\begin{defn}[Neat submanifold of a manifold with boundary]
\label{defn:Margalef-Roig_3-1-10}
(See Margalef Roig and Outerelo Dom{\'\i}nguez \cite[Definition 3.1.10]{Margalef-Roig_Outerelo-Dominguez_differential_topology}.)  
If $X'$ is a $C^p$ Banach submanifold with boundary ($p\geq 1$) of a $C^p$ Banach manifold with boundary $X$, then $X'$ is a \emph{neat} submanifold of $X$ if
\begin{equation}
\label{eq:Margalef-Roig_3-1-10a}
  \partial X' = (\partial X)\cap X'.
\end{equation}
\end{defn}

It will be convenient to interpret Definition \ref{defn:Margalef-Roig_3-1-10} in the model case of half planes. 

\begin{lem}[Half plane as a neat submanifold of another half plane]
\label{lem:Half_plane_neat_submanifold_half_plane}
Let $E$ be a real Banach space, $F \subset E$ be a closed linear subspace, $\lambda \in E^*$, and $\mu \in F^*$. If $F_\mu^+ \subset E_\lambda^+$ is a neat submanifold, then there is a positive constant $c$ such that
\[
  \mu = c\lambda\circ\iota_F,
\]
where $\iota_F:F\to E$ denotes the continuous inclusion operator.
\end{lem}

\begin{proof}
By Definition \ref{defn:Margalef-Roig_3-1-10}, we have
\begin{equation}
\label{eq:Neat_sub_hyperplane}
  \partial F_\mu^+ = (\partial E_\lambda^+)\cap F_\mu^+.
\end{equation}
First, suppose that $\mu$ is a positive constant. Then $\partial F_\mu^+ = \emptyset$ and $F_\mu^+ = F$ and the identity \eqref{eq:Neat_sub_hyperplane} yields $E_\lambda^0 \cap F = \emptyset$, which can only occur if $\partial E_\lambda^+ = \emptyset$, and hence $\lambda$ is also a positive constant. In this case, the conclusion holds with $c=\mu/\lambda$.

Second, suppose that $\mu$ is identically zero. Then $\partial F_\mu^+ = F = F_\mu^+$ and the identity \eqref{eq:Neat_sub_hyperplane} yields $F = E_\lambda^0\cap F$, so that $F \subset \Ker\lambda$ and $\lambda\circ\iota_F$ is identically zero and the conclusion holds for any positive constant $c$.

Finally, suppose that $\mu \in F^*$ is non-constant. The identity \eqref{eq:Neat_sub_hyperplane} is equivalent to $F_\mu^0 = E_\lambda^0 \cap F_\mu^+$ and thus
\begin{equation}
\label{eq:Kernel_mu_subspace_kernel_lambda}
  \Ker\mu \subset \Ker\lambda.
\end{equation}
Let $G \subset F$ be a one-dimensional closed complement of $\Ker\mu$, so $F = \Ker\mu\oplus G$, and choose $x_0\in G$ such that $\mu(x_0)>0$ and hence $x_0 \in \Int(F_\mu^+)$. The identity \eqref{eq:Neat_sub_hyperplane} implies that $\Int(F_\mu^+) \subset \Int(E_\lambda^+)$ for, otherwise, if $x_0 \in \partial E_\lambda^+$ then \eqref{eq:Neat_sub_hyperplane} and $x_0 \in F_\mu^+$ would yield $x_0 \in \partial F_\mu^+$, a contradiction. Therefore, $\lambda(x_0)>0$ and we may define $\alpha\in F^*$ by
\[
  \alpha := \mu - \frac{\mu(x_0)}{\lambda(x_0)}\lambda\circ\iota_F.
\]
But $\alpha(x_0) = 0$, so $\alpha \equiv 0$ on $G$, while $\alpha \equiv 0$ on $\Ker\mu$ by \eqref{eq:Kernel_mu_subspace_kernel_lambda}, and consequently $\alpha \equiv 0$ on $F$. The conclusion now holds with $c = \mu(x_0)/\lambda(x_0)$.
\end{proof}

\begin{rmk}[Interpretation of the definition of neat submanifold in coordinate charts]
\label{rmk:Interpetation_definition_neat_submanifold_coordinate_charts}
Suppose that $X'$ is a neat $C^p$ Banach submanifold ($p\geq 1$) of a $C^p$ Banach manifold $X$ and $x' \in X'$. Let $c=(U,\phi,E,\lambda)$ be a chart adapted to $X'$ at $x'$ through $(F,\mu)$, as provided by Definition \ref{defn:Margalef-Roig_3-1-1_and_3-1-2}. Note that $F_\mu^+ \subset E_\lambda^+$ by Item \eqref{item:Margalef-Roig_proposition_3-1-3_b} of Proposition \ref{prop:Margalef-Roig_proposition_3-1-3}. From the identity \eqref{eq:Margalef-Roig_3-1-10a}, we have
\[
  U\cap\partial X' = U\cap(\partial X)\cap X'.
\]
But $\phi(U\cap\partial X) = \phi(U)\cap\partial E_\lambda^+$ by Definition \ref{defn:Margalef-Roig_1-2-14_and_16} (which relies on Theorem \ref{thm:Margalef-Roig_1-2-12}) and $\phi(U\cap\partial X') = \phi(U)\cap\partial F_\mu^+$ by Definitions \ref{defn:Margalef-Roig_3-1-1_and_3-1-2} and \ref{defn:Margalef-Roig_1-2-14_and_16}, so applying the map $\phi$ to the preceding identity gives
\[
  \phi(U)\cap\partial F_\mu^+ = \phi(U) \cap (\partial E_\lambda^+) \cap  F_\mu^+.
\]
But this implies that
\[
  \partial F_\mu^+ = (\partial E_\lambda^+) \cap  F_\mu^+
\]
and so $F_\mu^+$ is a neat submanifold of $E_\lambda^+$ by Definition \ref{defn:Margalef-Roig_3-1-10}. Consequently, $\mu = c\lambda\circ\iota_F$ for some positive constant $c$ by Lemma \ref{lem:Half_plane_neat_submanifold_half_plane}.
\end{rmk}

\begin{rmk}[Neat and totally neat submanifolds of manifolds with corners]
\label{rmk:Neat_and_totally_submanifolds_manifolds_with_corners} 
We refer to Margalef Roig and Outerelo Dom{\'\i}nguez \cite[Definition 3.1.10]{Margalef-Roig_Outerelo-Dominguez_differential_topology} for the more general version of Definition \ref{defn:Margalef-Roig_3-1-10} for manifolds with corners, where there is a distinction between the concepts of neat and totally submanifolds of manifolds with corners. In our statement of Definition \ref{defn:Margalef-Roig_3-1-10}, we have relied on \cite[Definitions 1.2.6 and 1.2.14]{Margalef-Roig_Outerelo-Dominguez_differential_topology} to interpret the concept of the \emph{index} of a point $x$ in a manifold with corners in our specialization of \cite[Definition 3.1.10]{Margalef-Roig_Outerelo-Dominguez_differential_topology}. In particular, $\ind_{X'}(x') = 0 \iff x' \in \Int(X')$ and $\ind_{X'}(x') = 1 \iff x' \in \partial X'$ and similarly for points in $X$. If $x' \in \Int(X')$, then the equality \eqref{eq:Margalef-Roig_3-1-10a} implies that $x' \in \Int(X)$, while if $x' \in \partial X'$, then the equality \eqref{eq:Margalef-Roig_3-1-10a} implies that $x' \in \partial X$. Consequently, $\ind_{X'}(x') = \ind_X(x')$ for all $x' \in X'$ and Condition (b) (I) in \cite[Definition 3.1.10]{Margalef-Roig_Outerelo-Dominguez_differential_topology}, which defines the concept of a \emph{totally neat submanifold}, is equivalent to the condition \eqref{eq:Margalef-Roig_3-1-10a}.
\end{rmk}

\subsubsection{Immersions of manifolds with boundary}
We recall the

\begin{defn}[Immersion] 
\label{defn:Margalef-Roig_3-2-1}
(See  Margalef Roig and Outerelo Dom{\'\i}nguez \cite[Definition 3.2.1]{Margalef-Roig_Outerelo-Dominguez_differential_topology}.)
Let $X$ and $X'$ be $C^p$ Banach manifolds ($p\geq 1$), $f:X\to X'$ be a $C^p$ map, and $x\in X$ be a point. One says that $f$ is an \emph{immersion at $x$} if there are a chart $c=(U,\varphi,(E,\lambda))$ for $X$ with $\varphi(x)=0$ and a chart $c'=(U',\varphi',(E',\lambda'))$ for $X'$ with $\varphi'(f(x))=0$ such
that $f(U)\subset U'$, and $E\subset E'$ is a closed linear subspace that admits a closed complement in $E'$, and $\varphi(U) \subset \varphi'(U')$, and
\[
  \varphi' \circ f \circ \varphi^{-1}: \varphi(U) \to \varphi'(U')
\]
is the inclusion map (and thus $E_\lambda^+ \subset E_{\lambda'}'^+$ and $E_\lambda^0 \subset E_{\lambda'}'^0$). If $f$ is a $C^p$ immersion at every point $x\in X$, one says that $f$ is a $C^p$ \emph{immersion on $X$}.
\end{defn}

\begin{prop}[Openness of the immersion property] 
\label{prop:Margalef-Roig_3-2-2}
(See Margalef Roig and Outerelo Dom{\'\i}nguez \cite[Proposition 3.2.2]{Margalef-Roig_Outerelo-Dominguez_differential_topology}.)
Let $X$ and $X'$ be $C^p$ Banach manifolds ($p\geq 1$) and $f:X\to X'$ be a $C^p$ map. Then the subset $\{x\in X: f \text{ is an immersion at } x\}$ is open in $X$.
\end{prop}

\begin{thm}[Infinitesimal characterizations of immersions at points whose images are interior] 
\label{thm:Margalef-Roig_3-2-3}
(See  Margalef Roig and Outerelo Dom{\'\i}nguez \cite[Theorem 3.2.6]{Margalef-Roig_Outerelo-Dominguez_differential_topology}.)
Let $X$ and $X'$ be $C^p$ Banach manifolds ($p\geq 1$) and $f:X\to X'$ be a $C^p$ map and $x\in X$ be a point such that $f(x)\in\Int(X')$. Require that $p<\infty$ if $x\in\partial X$ and $X$ is infinite-dimensional. Then the following are equivalent:
\begin{enumerate}
\item $f$ is an immersion at $x \in X$.
\item $df(x):T_xX\to T_{f(x)}X'$ is an injective operator and $\Ran df(x)$ admits a closed complement in $T_{f(x)}X'$.
\end{enumerate}
\end{thm}


We now recall a characterization of immersions in which $f(x)$ could belong to $\partial X'$. We first have the

\begin{defn}[Index of a tangent vector] 
\label{defn:Margalef-Roig_3-2-11}
(See  Margalef Roig and Outerelo Dom{\'\i}nguez \cite[Definition 3.2.11]{Margalef-Roig_Outerelo-Dominguez_differential_topology}.)
Let $X$ be a differentiable manifold of class $p$, and $x\in X$, and $v\in (T_xX)^i$. We define the \emph{index of $v$ in $(T_xX)^i$} to be
\[
  \ind(v) = \ind\left((\sO_c^x)^{-1}(v)\right),
\]
that is, the index of the vector $(\sO_c^x)^{-1}(v)$ in $E_\lambda^+$, where $c=(U,\varphi, (E,\lambda))$ is a chart for $X$ with $x\in U$ and $\varphi(x)=0$.
\end{defn}

Recall from \cite[Definition 1.2.6]{Margalef-Roig_Outerelo-Dominguez_differential_topology} that for $w \in E_\lambda^+$, one defines $\ind(w) = 0$ if $w \in \Int(E_\lambda^+)$ and $\ind(w) = 1$ if $w \in \partial E_\lambda^+$.

\begin{thm}[Infinitesimal characterizations of immersions] 
\label{thm:Margalef-Roig_3-2-12}
(See  Margalef Roig and Outerelo Dom{\'\i}nguez \cite[Theorem 3.2.12]{Margalef-Roig_Outerelo-Dominguez_differential_topology}.)
Let $X$ and $X'$ be $C^p$ Banach manifolds ($p\geq 1$) and $f:X\to X'$ be a $C^p$ map and $x\in X$ be a point such that
\begin{itemize}
\item There is an open neighborhood $V_x$ of $x$ in $X$ with $f(V_x\cap \partial X) \subset \partial X'$,
\item $\ind(v) = \ind(df(x)v)$ for all $v\in (T_xX)^i$.  
\end{itemize}
Then the following hold:
\begin{enumerate}
\item If $df(x)$ is an injective operator and $\Ran df(x)$ is a closed subspace, then $df(x)((T_xX)^i) = (T_{f(x)}X')^i\cap df(x)(T_xX)$.

\item $f$ is an immersion at $x$ if and only if $df(x)$ is an injective operator and $\Ran df(x)$ admits a closed complement in $T_{f(x)}X'$.
\end{enumerate}
\end{thm}

\begin{prop}[Characterizations of immersions] 
\label{prop:Margalef-Roig_3-2-13}
(See  Margalef Roig and Outerelo Dom{\'\i}nguez \cite[Proposition 3.2.13]{Margalef-Roig_Outerelo-Dominguez_differential_topology}.)
Let $X$ and $X'$ be $C^p$ Banach manifolds ($p\geq 1$) and $f:X\to X'$ be a $C^p$ map and $x\in X$ be a point. Then the following are equivalent:
\begin{enumerate}
\item $f$ is an immersion at $x \in X$.
\item There is an open neighborhood $V_x$ of $x$ in $X$ such that $f(V_x)$ is a $C^p$ Banach submanifold of $X'$ and $f:V_x \to f(V_x)$ is a $C^p$ diffeomorphism.
\end{enumerate}
\end{prop}

\subsubsection{Embeddings of manifolds with boundary}

\begin{defn}[Embedding] 
\label{defn:Margalef-Roig_3-3-1}
(See  Margalef Roig and Outerelo Dom{\'\i}nguez \cite[Definition 3.3.1]{Margalef-Roig_Outerelo-Dominguez_differential_topology}.)
Let $X$ and $X'$ be $C^p$ Banach manifolds ($p\geq 1$) and $f:X\to X'$ be a $C^p$ map. One says that $f$ is a $C^p$ \emph{embedding} if $f$ is an immersion and $f:X\to f(X)$ is a homeomorphism.
\end{defn}

\begin{prop}[Characterizations of immersions] 
\label{prop:Margalef-Roig_3-3-2}
(See  Margalef Roig and Outerelo Dom{\'\i}nguez \cite[Proposition 3.3.2]{Margalef-Roig_Outerelo-Dominguez_differential_topology}.)
Let $X$ and $X'$ be $C^p$ Banach manifolds ($p\geq 1$) and $f:X\to X'$ be a $C^p$ map. Then the following are equivalent:
\begin{enumerate}
\item $f$ is an embedding.
\item $f(X)$ is a $C^p$ Banach submanifold of $X'$ and $f:X\to f(X)$ is a $C^p$ diffeomorphism.
\end{enumerate}  
\end{prop}

\begin{cor}[Characterizations of submanifolds as images of embeddings] 
\label{cor:Margalef-Roig_3-3-3}
(See  Margalef Roig and Outerelo Dom{\'\i}nguez \cite[Corollary 3.3.3]{Margalef-Roig_Outerelo-Dominguez_differential_topology}.)
Let $X$ be a $C^p$ Banach manifold ($p\geq 1$) and $X'$ be a subset of $X$. Then the following are equivalent:
\begin{enumerate}
\item $X'$ is a $C^p$ Banach submanifold of $X$.
\item $X'$ is the image of a $C^p$ embedding.
\end{enumerate}  
\end{cor}

\subsection{Submersions and preimage of a submanifold with boundary under a submersion}

\subsubsection{Submersions}

\begin{defn}[Submersion as a map with a smooth right inverse]
\label{defn:Margalef-Roig_4-1-1}
(See  Margalef Roig and Outerelo Dom{\'\i}nguez \cite[Definition 4.1.1]{Margalef-Roig_Outerelo-Dominguez_differential_topology}.)
Let $X$ and $X'$ be $C^p$ Banach manifolds ($p\geq 1$) and $f:X\to X'$ be a $C^p$ map and $x\in X$ be a point. The map $f$ is called a \emph{submersion at $x$} if there are an open neighborhood $V_{f(x)}$ of $f(x)$ in $X'$ and a map $s:V_{f(x)} \to X$ of class $p$, such that $s(f(x))=x$ and
\[
  f\circ s = \id\quad\text{on } V_{f(x)}.
\]
The map $f$ is called a $C^p$ \emph{submersion on $X$} if $f$ is submersion at every point $x \in X$.
\end{defn}

\begin{rmk}[Equivalent forms of the definition of a submersion of Banach manifolds with boundary]
\label{rmk:Equivalent_definitions_submersion_Banach_manifolds_boundary}	
The forthcoming Proposition \ref{prop:Margalef-Roig_4-1-13} (Items \eqref{item:Margalef-Roig_proposition_4-1-13_a} and \eqref{item:Margalef-Roig_proposition_4-1-13_b}) assures us that the Definitions \ref{defn:Transversality_maps_Banach_manifolds_boundary} and \ref{defn:Margalef-Roig_4-1-1} of $f$ being a submersion at a point $x \in X$ are equivalent.
\end{rmk}

\begin{prop}[Openness of a submersion]
\label{prop:Margalef-Roig_4-1-2}
(See  Margalef Roig and Outerelo Dom{\'\i}nguez \cite[Proposition 4.1.2]{Margalef-Roig_Outerelo-Dominguez_differential_topology}.)
Every submersion of class $p$ is an open map.
\end{prop}

\begin{prop}[Regularity of composition of a submersion with another map]
\label{prop:Margalef-Roig_4-1-3}
(See  Margalef Roig and Outerelo Dom{\'\i}nguez \cite[Proposition 4.1.3]{Margalef-Roig_Outerelo-Dominguez_differential_topology}.)
Let $X$ and $X'$ be $C^p$ Banach manifolds ($p\geq 1$) and $f:X\to X'$ be a $C^p$ submersion with $f(X)=X'$ and $g:X'\to X''$ be a map into a $C^p$ Banach manifold $X''$. Then $g$ is a $C^p$ map if and only if $g\circ f$ is a $C^p$ map.
\end{prop}

Margalef Roig and Outerelo Dom{\'\i}nguez note \cite[p. 159]{Margalef-Roig_Outerelo-Dominguez_differential_topology} that if they had defined a submersion as a map that is, locally, a projection map (as does Lang \cite[p. 24]{Lang_introduction_differential_topology}) then Proposition \ref{prop:Margalef-Roig_4-1-3} would not be true in general.

\begin{prop}[Composition of submersions]
\label{prop:Margalef-Roig_4-1-5}
(See  Margalef Roig and Outerelo Dom{\'\i}nguez \cite[Proposition 4.1.5]{Margalef-Roig_Outerelo-Dominguez_differential_topology}.)
Let $X$, $X'$, and $X''$ be $C^p$ Banach manifolds ($p\geq 1$) and $f:X\to X'$ and $g:X'\to X''$ be $C^p$ maps. If $f$ is a submersion at $x\in X$ and $g$ is a submersion at $f(x)\in X'$, then $g\circ f$ is a submersion at $x$.
\end{prop}

\begin{prop}[Consequences of submersion property for the tangent map]
\label{prop:Margalef-Roig_4-1-10}
(See  Margalef Roig and Outerelo Dom{\'\i}nguez \cite[Proposition 4.1.10]{Margalef-Roig_Outerelo-Dominguez_differential_topology}.)
Let $f:X\to X'$ be a $C^p$ map ($p\geq 1$) and $x\in X$ be a point. If $f$ is a submersion at $x$, then $df(x):T_xX \to T_{f(x)} X'$ is a surjective, continuous linear operator and $\Ker df(x)$ admits a closed complement in $T_xX$.
\end{prop}

\begin{prop}[Image of a manifold interior and boundary under a submersion]
\label{prop:Margalef-Roig_4-1-11}
(See  Margalef Roig and Outerelo Dom{\'\i}nguez \cite[Proposition 4.1.11]{Margalef-Roig_Outerelo-Dominguez_differential_topology}.)
Let $X$ and $X'$ be $C^p$ Banach manifolds ($p\geq 1$) and $f:X\to X'$ be $C^p$ maps and $x\in X$ be a point. Assume that $df(x):T_xX\to T_{f(x)}X'$ is a surjective operator and that $x \in \Int(X)$. Then $f(x) \in \Int(X')$. Therefore, if $f$ is a $C^p$ submersion, then $f(\Int(X)) \subset \Int(X')$ and $f^{-1}(\partial X') \subset \partial X$. In particular, if $f$ is a surjective $C^p$ submersion, then $\partial X' \subset f(\partial X)$ and $\partial X=\emptyset$ implies $\partial X' = \emptyset$.
\end{prop}

\begin{prop}[Characterizations of submersions]  
\label{prop:Margalef-Roig_4-1-13}
(See  Margalef Roig and Outerelo Dom{\'\i}nguez \cite[Proposition 4.1.13]{Margalef-Roig_Outerelo-Dominguez_differential_topology}.)
Let $X$ and $X'$ be $C^p$ Banach manifolds ($p\geq 1$) and $f:X\to X'$ be a $C^p$ map and $x\in X$ be a point. If there is an open neighborhood $V_x$ of $x$ in $X$ such that\footnote{If $x \in \Int(X)$, then this condition is always fulfilled.} $f (V_x\cap \partial X) \subset \partial X'$, then the following statements are equivalent:
\begin{enumerate}
\item\label{item:Margalef-Roig_proposition_4-1-13_a} $f$ is a submersion at $x$.

\item\label{item:Margalef-Roig_proposition_4-1-13_b}  $df(x):T_xX\to T_{f(x)}X'$ is a continuous, linear surjective operator and its kernel admits a closed complement in $T_xX$.

\item\label{item:Margalef-Roig_proposition_4-1-13_c} There are a chart $(U,\varphi,(E,\lambda))$ for $X$ with $x \in U$ and $\varphi(x)=0$, a chart $(U',\varphi',(E',\lambda'))$ for $X'$ with $f(x) \in U'$ and $\varphi'(f(x))=0$ and $f(U)\subset U'$, and a continuous, linear surjective operator $q:E\to E'$ such that $\Ker q$ admits a closed complement in $E$ and the following diagram commutes:
\begin{equation}
\label{eq:Submersion_commutative_diagram}    
\begin{CD}
X \supset U @> f \restriction U >> U' \subset X'
\\
@V \varphi VV @VV \varphi' V
\\
E_\lambda^+ \supset \varphi(U) @> q\restriction \varphi(U) >> \varphi'(U') \subset E_{\lambda'}'^+
\end{CD}
\end{equation}

\end{enumerate}
\end{prop}

We can now give the

\begin{proof}[Proof of Theorem \ref{mainthm:Preimage_point_under_submersion_and_implied_embedding}]
By hypothesis, $f$ is a submersion at $x_0$ and so, by the equivalence of Items \eqref{item:Margalef-Roig_proposition_4-1-13_a} and \eqref{item:Margalef-Roig_proposition_4-1-13_c} in Proposition \ref{prop:Margalef-Roig_4-1-13}, there are coordinate charts $c = (U,\varphi,(E,\lambda))$ for $X$ with $\varphi(x_0)=0$ and $c' = (U',\varphi',(E',\lambda'))$ for $X'$ with $\varphi(x_0')=0$ and $f(U)\subset U'$, and a continuous, linear surjective operator $q : E \to E'$ such that $K :=q^{-1}(0)$ admits a closed complement in $E$ and the diagram \eqref{eq:Submersion_commutative_diagram} commutes. Consequently, we have
\begin{align*}
  q^{-1}(0)\cap\varphi(U) &= (\varphi'\circ f\restriction U\circ\varphi^{-1})^{-1}(0) \quad\text{(by \eqref{eq:Submersion_commutative_diagram})}
  \\
                          &= \varphi\left((f\restriction U)^{-1}\left((\varphi')^{-1}(0)\right)\right)
  \\
                          &= \varphi\left((f\restriction U)^{-1}(x_0')\right) \quad\text{(since $\varphi(x_0')=0$)}
  \\
                          &= \varphi\left(U\cap f^{-1}(x_0')\right).
\end{align*}
By using $K=q^{-1}(0)$ and applying the map $\varphi^{-1}$, the preceding identity yields
\[
  \varphi^{-1}(\varphi(U)\cap K) = U\cap f^{-1}(x_0'),
\]
which verifies Item \eqref{item:Maintheorem_Margalef-Roig_4-1-13_submanifold_set}.

By definition of a coordinate chart (see Section \ref{subsubsec:Margalef-Roig_1-2}), the map $\varphi^{-1}:E_\lambda^+\supset\varphi(U) \to U \subset X$ is a $C^p$ embedding (in the sense of Definition \ref{defn:Margalef-Roig_3-3-1}) of the open subset $\varphi(U) \subset E_\lambda^+$ onto the open subset $U\subset X$. Hence, the composition $g = \varphi^{-1}\circ\iota_K \restriction \varphi(U)\cap K$ in \eqref{eq:Maintheorem_preimage_point_under_submersion_and_implied_embedding} is also a $C^p$ embedding from the relatively open subset $\varphi(U)\cap K \subset E_\lambda^+$ onto the relatively open subset $\varphi^{-1}(\varphi(U)\cap K) \subset X$. According to Proposition \ref{cor:Margalef-Roig_3-3-3}, the subset $\varphi^{-1}(\varphi(U)\cap K)$ is therefore a $C^p$ Banach submanifold of $X$. 

We now verify Item \eqref{item:Maintheorem_Margalef-Roig_4-1-13_submanifold_tangent spaces}. For any $x\in U$, the commutative diagram \eqref{eq:Submersion_commutative_diagram} of smooth maps yields a commutative diagram of continuous linear operators,
\begin{equation}
  \label{eq:Derivatives_submersion_commutative_diagram}
  \begin{CD}
T_xX @> df(x) >> T_{f(x)}X'
\\
@V d\varphi(x) VV @VV d\varphi'(f(x)) V
\\
E @> q >> E'
  \end{CD}
\end{equation}
In particular,
\[
  q = d\varphi'(f(x))\circ df(x)\circ (d\varphi(x_0))^{-1},
\]
and so
\begin{align*}
  K &= q^{-1}(0)
  \\
  &= \left(d\varphi'(f(x))\circ df(x)\circ (d\varphi(x))^{-1}\right)^{-1}(0)
  \\
  &= d\varphi(x)\left((df(x))^{-1}\left((d\varphi'(f(x))^{-1}(0)\right)\right).
\end{align*}
Because $d\varphi'(f(x)):T_{f(x)}X'\to E'$ is an isomorphism of Banach spaces and consequently $(d\varphi'(f(x))^{-1}(0)=0$, we obtain
\[
  (d\varphi(x))^{-1}K = (df(x))^{-1}(0), \quad\forall\, x \in U.
\]
and hence
\begin{equation}
  \label{eq:Definition_K_x}
  K = d\varphi(x)\left((df(x))^{-1}(0)\right), \quad\forall\, x \in U.
\end{equation}
By taking $x=x_0$, this gives the alternative characterization of $K$ provided in \eqref{eq:Maintheorem_definition_K}. Given that
\[
  g=\varphi^{-1}\circ\iota_K \restriction \varphi(U)\cap K: \varphi(U)\cap K \to U\cap f^{-1}(x_0')
\]
in \eqref{eq:Maintheorem_preimage_point_under_submersion_and_implied_embedding} is a $C^p$ embedding, then for each point $x \in U\cap f^{-1}(x_0')$ the submanifold $U\cap f^{-1}(x_0') \subset X$ has tangent space
\begin{align*}
  T_x(f^{-1}(x_0')) &= (dg(\varphi(x))K \quad\text{(by \eqref{eq:Maintheorem_preimage_point_under_submersion_and_implied_embedding})}
  \\
                    &= (d\varphi^{-1}(\varphi(x))K
  \\
                    &= (d\varphi(x))^{-1}K
  \\
                    &= (df(x))^{-1}(0) \quad\text{(by \eqref{eq:Definition_K_x})}.                     
\end{align*}  
This completes the verification of Item \eqref{item:Maintheorem_Margalef-Roig_4-1-13_submanifold_tangent spaces}. Lastly, if $f$ is a submersion at every point $x \in f^{-1}(x_0')$, then $f^{-1}(x_0')$ is a $C^p$ Banach submanifold of $X$ since $U\cap f^{-1}(x_0')$ is a $C^p$ Banach submanifold by the preceding calculations when $x=x_0$ and applying that conclusion to each $x \in f^{-1}(x_0')$. This completes the proof of Theorem \ref{mainthm:Preimage_point_under_submersion_and_implied_embedding}.
\end{proof}


\subsubsection{Preimage of a submanifold with boundary under a submersion}
The following result is a simplified version of \cite[Proposition 4.2.1]{Margalef-Roig_Outerelo-Dominguez_differential_topology}, where we restrict our attention to the case of manifolds with boundary rather than manifolds with corners.

\begin{thm}[Preimage theorem for manifolds with boundary]
\label{thm:Margalef-Roig_proposition_4-2-1}
(See  Margalef Roig and Outerelo Dom{\'\i}nguez \cite[Proposition 4.2.1]{Margalef-Roig_Outerelo-Dominguez_differential_topology}.)
Let $X$ and $X'$ be $C^p$ Banach manifolds ($p\geq 1$), and $f:X\to X'$ be a $C^p$ map, and $Y'$ be a $C^p$ Banach submanifold $X'$. If for every $x \in f^{-1}(Y')$, the map $f$ is a submersion at $x$ and there is an open neighborhood $V_x\subset X$ of $x$ such that
\begin{equation}
\label{eq:Relatively_open_neighborhood_boundary_point_maps_into_boundary_codomain}
  f(V_x \cap \partial X) \subset \partial X',
\end{equation}
then the following hold:
\begin{enumerate}
\item\label{item:Margalef-Roig_proposition_4-2-1_a}
$f^{-1}(Y')$ is a $C^p$ Banach submanifold of $X$.

\item\label{item:Margalef-Roig_proposition_4-2-1_b}
$\partial(f^{-1}(Y')) = f^{-1}(\partial Y')$.

\item\label{item:Margalef-Roig_proposition_4-2-1_c}
$T_x(f^{-1}(Y')) = (df(x))^{-1}(T_{f(x)}Y')$ for every $x \in f^{-1}(Y')$.

\item\label{item:Margalef-Roig_proposition_4-2-1_d}
$\codim_x f^{-1}(Y) = \codim_{f(x)} Y'$ for every $x \in f^{-1}(Y')$.

\item\label{item:Margalef-Roig_proposition_4-2-1_e}
$f\restriction f^{-1}(Y') :f^{-1}(Y') \to Y'$ is a $C^p$ submersion.
\end{enumerate}
\end{thm}

\begin{exmp}[Counterexample to conclusion that preimage of boundary is boundary of preimage when codomain and target manifolds have empty boundary, so hypothesis \eqref{eq:Relatively_open_neighborhood_boundary_point_maps_into_boundary_codomain} does not hold]
\label{exmp:Counterexample_conclusion_preimage_boundary_is_boundary_preimage}  
  Suppose $f:\HH^2 \to \RR$ is the map $(x,y) \mapsto x$, where $\HH^2 = \{(x,y)\in \RR^2: y \geq 0\}$. Take $X = \HH^2$ and $X' = \RR$ and $Y' = \{0\}$ (the origin in $\RR$). Since $\partial X = \RR$ (the $x$-axis) and $\partial X' = \emptyset$, the hypothesis $f(V_x\cap \partial X) \subset \partial X'$ in \eqref{eq:Relatively_open_neighborhood_boundary_point_maps_into_boundary_codomain} does \emph{not} hold. Clearly $f$ (and $\partial f$) is a submersion. Note that $\partial Y' = \emptyset$. However, $f^{-1}(Y') = \{(x,y)\in\RR^2: x = 0, y \geq 0\}$ and thus $\partial (f^{-1}(Y')) = \{(0,0)\}$, the origin in $\HH^2$. In particular, the conclusion in Item \eqref{item:Margalef-Roig_proposition_4-2-1_b} in Theorem \ref{thm:Margalef-Roig_proposition_4-2-1}, which asserts that $\partial (f^{-1}(Y')) = f^{-1}(\partial Y')$, does \emph{not} hold.

Suppose now that we \emph{augment} our definition of $f$ and define $\tilde{f}:\HH^2\to\HH^2$ by $(x,y)\mapsto (x,y)$ and $\widetilde{X}' = \RR\times[0,\infty) = \HH^2$, so $\partial \widetilde{X}' = \RR$, and $\widetilde{Y}' = \{(x,y)\in\RR^2: x = 0, y \geq 0\}$. Observe that $\tilde{f}(\partial\widetilde{X}) = \partial\widetilde{X}'$, so the hypothesis \eqref{eq:Relatively_open_neighborhood_boundary_point_maps_into_boundary_codomain} trivially holds. Moreover, $\tilde{f}^{-1}(\widetilde{Y}') = \{(x,y)\in\RR^2: x = 0, y \geq 0\}$, while $\tilde{f}^{-1}(\partial \widetilde{Y}') = \tilde{f}^{-1}(0,0) = \{(x,y)\in\RR^2: x = 0, y \geq 0\}$. Therefore, $\partial (\tilde{f}^{-1}(\widetilde{Y}')) = \{(0,0)\} = \tilde{f}^{-1}(\partial\widetilde{Y}')$ and, as expected, the conclusion in Item \eqref{item:Margalef-Roig_proposition_4-2-1_b} holds in this case.
\end{exmp}  

\subsection{Transversality for maps of Banach manifolds with boundary}
\label{subsec:Margalef-Roig_7-1}
Margalef Roig and Outerelo Dom{\'\i}nguez provide the following definition of transversality for maps of Banach manifolds with boundary, which they show in Lemma \ref{lem:Margalef-Roig_7-1-5} is equivalent to Definition \ref{defn:Transversality_all_Banach_manifolds_without_boundary_AMR} in the special case of Banach manifolds without boundary.


\begin{defn}[Transversality for maps of Banach manifolds with boundary]
\label{defn:Margalef-Roig_7-1-1}  
(See Margalef Roig and Outerelo Dom{\'\i}nguez \cite[Definition 7.1.1]{Margalef-Roig_Outerelo-Dominguez_differential_topology}.)
Let $f: X \to X'$ be a $C^p$ map ($p\geq 1$) of $C^p$ Banach manifolds with boundary and let $X'' \subset X'$ be a $C^p$ Banach submanifold with boundary. Then $f$ is \emph{transverse to $X''$ at $x\in X$}, denoted $f\transv_x X''$, if either $f(x) \notin X''$ or if $f(x) \in X''$, then there are a chart $(U',\varphi', E')$ for $X'$ adapted to $X''$ at $f(x)$ by means of $E''$, a closed complement $F'$ of $E''$ in $E'$ and an open subset $U \subset X$ such that $x \in U$, $f(U) \subset U'$, and the map
\[
  U \xrightarrow{f\restriction U} U' \xrightarrow{\varphi'} \varphi'(U') \xrightarrow{(\sO')^{-1}} E'' \times F' \xrightarrow{\pi_2} F'
\]
is a submersion at $x \in U$, where $\sO':E'' \times F' \to E'$ is the isomorphism of Banach spaces defined by $\sO'(a,b)=a+b$ and $\pi_2(a,b)=b$. If $f\transv_x X''$ for all $x \in X$, then $f$ is \emph{transverse to $X''$}, denoted $f\transv X''$.
\end{defn}

\begin{rmk}[Equivalent forms of the definition of transversality for maps of Banach manifolds with boundary]
\label{rmk:Equivalent_definitions_transversality_maps_Banach_manifolds_boundary}	
The forthcoming Lemma \ref{lem:Margalef-Roig_7-1-5} and Proposition \ref{prop:Margalef-Roig_7-1-7} assure us that the Definitions \ref{defn:Transversality_maps_Banach_manifolds_boundary} and \ref{defn:Margalef-Roig_7-1-1} of $f\transv_x X''$ at $x \in X$ are equivalent.
\end{rmk}

\begin{prop}[Consequence of transversality for a map of Banach manifolds with boundary]
\label{prop:Margalef-Roig_7-1-3}
(See Margalef Roig and Outerelo Dom{\'\i}nguez \cite[Proposition 7.1.3]{Margalef-Roig_Outerelo-Dominguez_differential_topology}.)  
Let $f:X\to X'$ be a $C^p$ map ($p\geq 1$) of $C^p$ Banach manifolds with boundary, $X'' \subset X'$ be a $C^p$ Banach submanifold with boundary, and $x \in f^{-1}(X'')$. If $f \transv_x X''$ in the sense of Definition \ref{defn:Margalef-Roig_7-1-1}, then\footnote{We suppress explicit notation for inclusion maps and their differentials.} the following hold:
\begin{enumerate}
\item\label{item:Proposition_Margalef-Roig_7-1-3_span}
$T_{f(x)}X' = \Ran df(x) + T_{f(x)}X''$,
  
\item\label{item:Proposition_Margalef-Roig_7-1-3_complement}
$(df(x))^{-1}(T_{f(x)}X'')$ admits a closed complement in $T_xX$.
\end{enumerate}
\end{prop}

We have the following partial converse to Proposition \ref{prop:Margalef-Roig_7-1-3}.

\begin{lem}[Characterization of transversality at interior points for a map of Banach manifolds]
\label{lem:Margalef-Roig_7-1-5}
(See Margalef Roig and Outerelo Dom{\'\i}nguez \cite[Lemma 7.1.5]{Margalef-Roig_Outerelo-Dominguez_differential_topology}.)
Let $f:X\to X'$ be a $C^p$ map ($p\geq 1$) of $C^p$ Banach manifolds with boundary, $X'' \subset X'$ be a $C^p$ Banach submanifold with boundary, and $x \in f^{-1}(X'')\cap\Int(X)$. If Properties \eqref{item:Proposition_Margalef-Roig_7-1-3_span} and \eqref{item:Proposition_Margalef-Roig_7-1-3_complement} in Proposition \ref{prop:Margalef-Roig_7-1-3} hold, then $f \transv_x X''$ in the sense of Definition \ref{defn:Margalef-Roig_7-1-1}.
\end{lem}

Margalef Roig and Outerelo Dom{\'\i}nguez provide a generalization of Lemma \ref{lem:Margalef-Roig_7-1-5} that allows for arbitrary points $x \in f^{-1}(X'')$, without the restriction that $x \in \Int(X')$. Given Lemma \ref{lem:Margalef-Roig_7-1-5}, we shall only need the following special case of their result.

\begin{prop}[Characterization of transversality at boundary points for a map of Banach manifolds]
\label{prop:Margalef-Roig_7-1-7}
(See Margalef Roig and Outerelo Dom{\'\i}nguez \cite[Definition 1.2.16 and Proposition 7.1.7]{Margalef-Roig_Outerelo-Dominguez_differential_topology}.)
Let $f:X\to X'$ be a $C^p$ map ($p\geq 1$) of $C^p$ Banach manifolds with boundary, $X'' \subset X'$ be a $C^p$ Banach submanifold with boundary, and $x \in f^{-1}(X'')\cap\partial X$, and $\partial f \equiv f \restriction \partial X : \partial X \to X'$. Then the following are equivalent:
\begin{enumerate}
\item $T_{f(x)}X' = \Ran d(\partial f)(x) + T_{f(x)}X''$ and $(d(\partial f)(x))^{-1}(T_xX'')$ admits a closed complement in $T_xX'$.
\item $f\transv_x X''$.
\item $\partial f\transv_x X''$.  
\end{enumerate}  
\end{prop}

\begin{lem}[Properties of neat submanifolds]
\label{lem:Margalef-Roig_7-1-13}
(See Margalef Roig and Outerelo Dom{\'\i}nguez \cite[Lemma 7.1.13]{Margalef-Roig_Outerelo-Dominguez_differential_topology}.)
Let $X'$ be a $C^p$ Banach manifold with boundary ($p\geq 1$) and $X'' \subset X'$ be a neat $C^p$ Banach submanifold. Then the following hold:
\begin{enumerate}
\item $\Int(X'') \cap \partial X' = \emptyset$.
\item If $(U',\varphi', (E',\lambda'))$ is a chart for $X'$ adapted to $X''$ at $x'' \in X''$ by means of $(E'',\lambda'')$, then the following hold:
\begin{inparaenum}[(\itshape a\upshape)]
\item $E_{\lambda''}''^+ \subset E_{\lambda'}'^+$,
\item $\partial E_{\lambda''}''^+ \subset \partial E_{\lambda'}'^+$,
\item $\Int(E_{\lambda''}''^+) \subset \Int(E_{\lambda'}'^+)$,
\item $E'' \less E_{\lambda''}''^+ \subset E' \less E_{\lambda'}'^+$,
\item $(\varphi')^{-1}(E'') = U'\cap X''$, and
\item $E_{\lambda''}''^+ = E'' \cap E_{\lambda'}'^+$.
\end{inparaenum}
\end{enumerate}
\end{lem}

\begin{prop}[Openness of the property of transversality for a map of Banach manifolds with boundary]
\label{prop:Margalef-Roig_7-1-18}
(See Margalef Roig and Outerelo Dom{\'\i}nguez \cite[Proposition 7.1.18]{Margalef-Roig_Outerelo-Dominguez_differential_topology}.)  
Let $f:X\to X'$ be a $C^p$ map ($p\geq 1$) of $C^p$ Banach manifolds with boundary and $X'' \subset X'$ be a (topologically) closed $C^p$ Banach submanifold. Then the following hold:
\begin{enumerate}
\item If $f \transv_x X''$, then there is an open neighborhood $V_x$ of $x$ in $X$ such that $f \transv_y X''$ for all $y\in V_x$.
\item The subset $G := \{x \in X: f \transv_x X''\}$ is open in $X$.
\end{enumerate}
\end{prop}

\begin{cor}[Preimage of a neat submanifold under a transverse map]
\label{cor:Margalef-Roig_7-1-20}
(See Margalef Roig and Outerelo Dom{\'\i}nguez \cite[Corollary 7.1.20]{Margalef-Roig_Outerelo-Dominguez_differential_topology}.)
Let $f:X\to X'$ be a $C^p$ map ($p\geq 1$) of $C^p$ Banach manifolds with boundary and $X'' \subset X'$ be a neat\footnote{In the more general setting of manifolds with corners, as in \cite[Corollary 7.1.20]{Margalef-Roig_Outerelo-Dominguez_differential_topology}, one would need to strengthen the hypothesis that $X''$ be a neat submanifold to $X''$ being a totally neat submanifold.} $C^p$ Banach submanifold. If $f\transv X''$, then
the following hold:
\begin{enumerate}
\item\label{item:corollary_Margalef-Roig_7-1-20_a}
  $f^{-1}(X'')$ is a neat $C^p$ Banach submanifold of $X$.
\item\label{item:corollary_Margalef-Roig_7-1-20_b}
  For every $x\in f^{-1}(X'')$,
  \[
    T_x(f^{-1}(X'')) = (df(x))^{-1}(T_{f(x)}X'').
  \]
\item\label{item:corollary_Margalef-Roig_7-1-20_c}
  $\codim f^{-1}(X'') = \codim X''$.
\end{enumerate}
\end{cor}

In more sophisticated applications than those considered in this article, we shall require an extension of Theorem \ref{mainthm:Preimage_point_under_submersion_and_implied_embedding}, which assumes that the submanifold $Y' \subset X'$ is a point and that the map $f:X\to X'$ is a submersion, to the general context of Theorem \ref{mainthm:Preimage_submanifold_under_transverse_map_and_implied_embedding}, where $Y'\subset X'$ is a neat $C^p$ Banach submanifold and $f:X\to X'$ is a map such that $f\transv_{x_0} Y'$. In our forthcoming proof of Theorem \ref{mainthm:Preimage_submanifold_under_transverse_map_and_implied_embedding}, we shall adapt a trick described by Guillemin and Pollack \cite[pp. 27--28]{Guillemin_Pollack} that allows them (in their setting of manifolds without boundary) to reduce this general case to the special case where $Y'$ is a point and $f:X\to X'$ is a submersion. Moreover, this trick also leads to an alternative proof of the main conclusion of Corollary \ref{cor:Margalef-Roig_7-1-20}, as we note in the forthcoming Remark \ref{rmk:Alternative_proof_corollary_Margalef-Roig_7-1-20}.

\begin{proof}[Proof of Theorem \ref{mainthm:Preimage_submanifold_under_transverse_map_and_implied_embedding}]
Since $Y' \subset X'$ is a submanifold and $x_0'=f(x_0)\in Y'$, there is a chart $c'=(U',\varphi',(E',\lambda'))$ for $X'$ with $\varphi'(x_0')=0$ that is adapted to $Y'$ in the sense of Definition \ref{defn:Margalef-Roig_3-1-1_and_3-1-2}. Hence, there is a closed linear subspace $F'\subset E'$ and $\mu \in F'^*$ such that $\varphi'(U'\cap Y') = \varphi'(U')\cap F_\mu'^+$ and $\varphi'(U')\cap F_\mu'^+$ is an open subset of $F_\mu'^+$. Because $Y' \subset X'$ is a neat submanifold, Remark \ref{rmk:Interpetation_definition_neat_submanifold_coordinate_charts} implies that there is a positive constant $c$ such that
\[
  \mu = c\lambda \circ \iota_{F'},
\]
where $\iota_{F'}:F'\to E'$ denotes the continuous linear inclusion operator; we may assume without loss of generality that $c=1$.

Because $F'\subset E'$ admits a closed complement by Definition \ref{defn:Margalef-Roig_3-1-1_and_3-1-2}, there is a closed linear subspace $G'\subset E'$ such that $E'=F'\oplus G'$. Let $\pi_{G'}: E' \to G'$ denote the continuous linear projection operator corresponding to the splitting $E'=F'\oplus G'$, so that $F' = \pi_{G'}^{-1}(0)$ and
\[
\varphi'(U'\cap Y') = \varphi'(U')\cap F_\mu'^+ = \varphi'(U')\cap \pi_{G'}^{-1}(0) \cap E_\lambda'^+ = \varphi'(U'\cap (\pi_{G'}\circ\varphi')^{-1}(0)) \cap E_\lambda'^+.
\]
Hence, by applying the map $(\varphi')^{-1}$ to the preceding identity we obtain
\[
  U'\cap Y' = U'\cap (\pi_{G'}\circ\varphi')^{-1}(0).
\]
If we define $\nu = \lambda \circ \iota_{G'}$, then the map $\pi_{G'}: E_\lambda'^+ \to G_\nu'^+$ is well-defined. We introduce the key

\begin{lem}[Equivalence of $f$ transverse to $Y'$ at $x_0$ and $h$ a submersion at $x_0$]
\label{lem:f_transverse_Y_prime_at_x0_iff_h_is_submersion_at_x0}  
Assume the notation of the preceding paragraphs. If
\begin{equation}
  \label{eq:Definition_h}
  h \equiv \pi_{G'} \circ \varphi' \circ f \restriction U: X \supset U \to G_\nu'^+,
\end{equation}
then the following are equivalent:
\begin{enumerate}
\item $f\transv_{x_0} Y'$.  
\item $h$ is a submersion at $x_0$.
\end{enumerate}  
\end{lem}

By hypothesis of Theorem \ref{mainthm:Preimage_submanifold_under_transverse_map_and_implied_embedding}, we have $f \transv_{x_0} Y'$. We shall prove Lemma \ref{lem:f_transverse_Y_prime_at_x0_iff_h_is_submersion_at_x0} with the aid of the following two claims.

\begin{claim}
\label{claim:dh_surjective_iff_range_df_and_tangent_space_submanifold_span_tangent_space_codomain}  
If $x\in U\cap f^{-1}(Y')$, then the continuous linear operator $dh(x): T_xX \to G'$ is surjective if and only if
\begin{equation}
\label{eq:Range_df_and_tangent_space_submanifold_span_tangent_space_codomain}  
  \Ran df(x) + T_{f(x)}Y' = T_{f(x)}X'.
\end{equation}
\end{claim}  

\begin{proof}[Proof of Claim \ref{claim:dh_surjective_iff_range_df_and_tangent_space_submanifold_span_tangent_space_codomain}]
Following the strategy of \cite[p. 28]{Guillemin_Pollack}, we first observe that
\[
  dh(x) = \pi_{G'} \circ d\varphi'(f(x)) \circ df(x): T_xX \to G'
\]
is surjective if and only if
\begin{equation}
\label{eq:pi_G_prime_circ_dvarphi_prime_on_range_df}
  \pi_{G'} \circ d\varphi'(f(x)): \Ran df(x) \to G'
\end{equation}
is surjective. Because $d\varphi'(f(x)): T_{f(x)}X' \to E'$ is an isomorphism of vector spaces, $\pi_{G'}: E' \to G'$ has kernel $F' \subset E'$, and $d\varphi'(f(x)): T_{f(x)}Y' \to F'$ is an isomorphism of vector spaces, then
\begin{equation}
\label{eq:Kernel_pi_G_prime_circ_dvarphi_prime}
  \Ker\left( \pi_{G'} \circ d\varphi'(f(x)): T_{f(x)}X' \to G' \right) = T_{f(x)}Y'.
\end{equation}
One trivially has
\begin{equation}
\label{eq:Preimage_G_prime}
  \left( \pi_{G'} \circ d\varphi'(f(x)) \right)^{-1}(G') = T_{f(x)}X'.
\end{equation}  
If the composition of operators \eqref{eq:pi_G_prime_circ_dvarphi_prime_on_range_df} is surjective, then
\begin{equation}
\label{eq:Range_df_and_tangent_space_submanifold_span_preimage_G_prime} 
  \Ran df(x) + T_{f(x)}Y' = \left( \pi_{G'} \circ d\varphi'(f(x)) \right)^{-1}(G'),
\end{equation}
and this identity together with \eqref{eq:Preimage_G_prime} imply that
\eqref{eq:Range_df_and_tangent_space_submanifold_span_tangent_space_codomain} holds. Conversely, if \eqref{eq:Range_df_and_tangent_space_submanifold_span_tangent_space_codomain} holds, then that identity together with \eqref{eq:Preimage_G_prime} imply that \eqref{eq:Range_df_and_tangent_space_submanifold_span_preimage_G_prime} holds. But then the identification of the kernel \eqref{eq:Kernel_pi_G_prime_circ_dvarphi_prime} and the identity \eqref{eq:Range_df_and_tangent_space_submanifold_span_preimage_G_prime} imply that the composition of operators  \eqref{eq:pi_G_prime_circ_dvarphi_prime_on_range_df} is surjective. This completes the proof of Claim \ref{claim:dh_surjective_iff_range_df_and_tangent_space_submanifold_span_tangent_space_codomain}.  
\end{proof}

\begin{claim}
\label{claim:Preimage_dfx0_Tfx0_Y_prime_equals_kernel_dhx0}
Assume the notation of the preceding paragraphs. Then 
\begin{equation}
\label{eq:Preimage_dfx0_Tfx0_Y_prime_equals_kernel_dhx0}
  (df(x_0))^{-1}(T_{f(x_0)}Y') = \Ker dh(x_0).
\end{equation}
\end{claim}

\begin{proof}[Proof of Claim \ref{claim:Preimage_dfx0_Tfx0_Y_prime_equals_kernel_dhx0}]
We observe that
\begin{align*}
  \Ker\left(dh(x_0):T_{x_0}X \to G'\right)
  &=
  (dh(x_0))^{-1}(0)
  \\
  &= \left(\pi_{G'} \circ d\varphi'(f(x_0)) \circ df(x_0)\right)^{-1}(0)
  \\
  &= \left(d\varphi'(f(x_0)) \circ df(x_0)\right)^{-1}(F')
  \\
  &= (df(x_0))^{-1}\left(\left(d\varphi'(f(x_0))\right)^{-1}(F')\right)
  \\
  &= (df(x_0))^{-1}\left(T_{f(x_0)}Y'\right), 
\end{align*}
using the facts that $\pi_{G'}^{-1}(0)=F'$ and $d\varphi'(f(x_0)):T_{f(x_0)}Y'\to F'$ is an isomorphism of vector spaces. This completes the proof of Claim \ref{claim:Preimage_dfx0_Tfx0_Y_prime_equals_kernel_dhx0}.
\end{proof}

We can now conclude the 

\begin{proof}[Proof of Lemma \ref{lem:f_transverse_Y_prime_at_x0_iff_h_is_submersion_at_x0}]
We seek to apply the equivalence of Items \eqref{item:Margalef-Roig_proposition_4-1-13_a} and \eqref{item:Margalef-Roig_proposition_4-1-13_b} in the equivalent characterizations of submersions provided by Proposition \ref{prop:Margalef-Roig_4-1-13}.

If $f\transv_{x_0} Y'$, then Item \eqref{item:Proposition_Margalef-Roig_7-1-3_span} in Proposition \ref{prop:Margalef-Roig_7-1-3} implies that the identity 
\eqref{eq:Range_df_and_tangent_space_submanifold_span_tangent_space_codomain} holds at $x=x_0$, that is,
\begin{equation}
\label{eq:Range_df_and_tangent_space_submanifold_span_tangent_space_codomain_x0}
  \Ran df(x_0) + T_{f(x_0)}Y' = T_{f(x_0)}X'.
\end{equation}
By Claim \ref{claim:dh_surjective_iff_range_df_and_tangent_space_submanifold_span_tangent_space_codomain} and the preceding identity, we see that the operator $dh(x_0): T_{x_0}X \to G'$ is surjective. Because Item \eqref{item:Proposition_Margalef-Roig_7-1-3_complement} in Proposition \ref{prop:Margalef-Roig_7-1-3} implies that $(df(x_0))^{-1}(T_{f(x_0)}Y')$ admits a closed complement in $T_{x_0}X$, then $\Ker dh(x_0)$ admits a closed complement in $T_{x_0}X$ by Claim \ref{claim:Preimage_dfx0_Tfx0_Y_prime_equals_kernel_dhx0}. Therefore, Proposition \ref{prop:Margalef-Roig_4-1-13} implies that $h$ is a submersion at $x_0$.

If $h$ is a submersion at $x_0$, we may apply Proposition \ref{prop:Margalef-Roig_4-1-13} and reverse the preceding argument to conclude that \eqref{eq:Range_df_and_tangent_space_submanifold_span_tangent_space_codomain_x0} holds and that $(df(x_0))^{-1}(T_{f(x_0)}Y')$ admits a closed complement in $T_{x_0}X$. If $x_0 \in f^{-1}(Y')\cap\Int(X)$, then Lemma \ref{lem:Margalef-Roig_7-1-5} implies that $f\transv_{x_0} Y'$, while if $x_0 \in f^{-1}(Y')\cap\partial X$, then Proposition \ref{prop:Margalef-Roig_7-1-7} implies that $f\transv_{x_0} Y'$. This completes the proof of Lemma \ref{lem:f_transverse_Y_prime_at_x0_iff_h_is_submersion_at_x0}.
\end{proof}  

Given our hypothesis in Theorem \ref{mainthm:Preimage_submanifold_under_transverse_map_and_implied_embedding} that $f\transv_{x_0} Y'$, we can apply Lemma \ref{lem:f_transverse_Y_prime_at_x0_iff_h_is_submersion_at_x0} to conclude that $h$ is a submersion at $x_0$. Consequently, we can apply Proposition \ref{prop:Margalef-Roig_4-1-13} and Theorem \ref{mainthm:Preimage_point_under_submersion_and_implied_embedding}, but with the map $f$ replaced by $h$ and the codomain $X'$ replaced by $G_\nu'^+$. By the equivalence of Items \eqref{item:Margalef-Roig_proposition_4-1-13_a} and \eqref{item:Margalef-Roig_proposition_4-1-13_c} in Proposition \ref{prop:Margalef-Roig_4-1-13}, there are
\begin{itemize}
\item a chart $(V,\psi,(E,\alpha))$ for $X$ with $x_0 \in V$ and $\psi(x_0)=0$ and $V\subset U$,
\item a chart $(V',\psi',(E',\alpha'))$ for $G_\nu'^+$ with $h(x_0) = 0 \in V'$ and $\psi'(0)=0$ and $h(V)\subset V'$, and
\item a continuous, linear surjective operator $q:E\to E'$ such that $\Ker q$ admits a closed complement in $E$,
\end{itemize}
and such that the following diagram commutes:
\begin{equation}
\label{eq:Guillemin_Pollack_trick_submersion_commutative_diagram}    
\begin{CD}
X \supset V @> h \restriction V >> V' \subset G_\nu'^+
\\
@V \psi VV @VV \psi' V
\\
E_\alpha^+ \supset \psi(V) @> q\restriction \psi(V) >> \psi'(V') \subset G_{\alpha'}'^+
\end{CD}
\end{equation}
We observe that
\begin{align*}
  q^{-1}(0) &= (dq(0))^{-1}(0) = \left(d\left(\psi'\circ h\restriction V\circ\psi^{-1}\right)(0)\right)^{-1}(0)
  \\
  &= \left(d\psi'(0)\circ dh(x_0)\circ d(\psi^{-1})(0)\right)^{-1}(0)
  \\
            &= \left(d\psi'(0)\circ \pi_{G'}\circ d\varphi'(f(x_0))\circ df(x_0)\circ (d\psi(x_0))^{-1}\right)^{-1}(0)
  \\
            &= \left(\pi_{G'}\circ d\varphi'(f(x_0))\circ df(x_0) \circ (d\psi(x_0))^{-1}\right)^{-1}(0)
  \\
            &=  \left(d\varphi'(f(x_0))\circ df(x_0) \circ (d\psi(x_0))^{-1}\right)^{-1}(F')
  \\
            &=  \left(df(x_0) \circ (d\psi(x_0))^{-1}\right)^{-1}(T_{f(x_0)}Y')
 \\
            &=  d\psi(x_0)\left((df(x_0)^{-1}(T_{f(x_0)}Y')\right).                 
\end{align*}
Hence, $q^{-1}(0)=L$ in \eqref{eq:Maintheorem_definition_L}. Theorem \ref{mainthm:Preimage_point_under_submersion_and_implied_embedding} implies that the composition $g$ in \eqref{eq:Maintheorem_preimage_submanifold_under_transverse_map_and_implied_embedding} gives a $C^p$ embedding from the relatively open subset $\psi(V)\cap L \subset E_\alpha^+$ onto the $C^p$ Banach submanifold $\psi^{-1}(\psi(V)\cap L) \subset X$. Moreover, noting that $x_0' := h(x_0) = 0 \in G_\nu'^+$, we see that Theorem \ref{mainthm:Preimage_point_under_submersion_and_implied_embedding} yields
\begin{enumerate}
\item $\psi^{-1}(\psi(V)\cap L) = V\cap h^{-1}(0)$, and
\item $T_x(h^{-1}(0)) = (dh(x))^{-1}(0) = (d\psi(x))^{-1}L$, for all $x \in V\cap h^{-1}(0)$.  
\end{enumerate}
But
\[
  V\cap h^{-1}(0) = \left(\pi_{G'} \circ \psi' \circ f \restriction V\right)^{-1}(0) = (f \restriction V)^{-1}(V'\cap Y') = V\cap f^{-1}(Y')
\]
and thus  $\psi^{-1}(\psi(V)\cap L) = V\cap f^{-1}(Y')$, which verifies Item \eqref{item:Maintheorem_Margalef-Roig_4-2-1_submanifold_set}, and consequently
\[
  T_x(f^{-1}(Y')) = T_x(h^{-1}(0)) = (d\psi(x))^{-1}L,
\]
which verifies Item \eqref{item:Maintheorem_Margalef-Roig_4-2-1_submanifold_tangent spaces}.

We have proved that $V\cap f^{-1}(Y')$ is a $C^p$ Banach submanifold of the open neighborhood $V\subset X$ (see Item \eqref{item:Maintheorem_Margalef-Roig_4-2-1_submanifold_set} in the conclusions of Theorem \ref{mainthm:Preimage_submanifold_under_transverse_map_and_implied_embedding}). The fact that $V\cap f^{-1}(Y')$ is also neat is given by Item \eqref{item:corollary_Margalef-Roig_7-1-20_a} in Corollary \ref{cor:Margalef-Roig_7-1-20}. This completes the proof of Theorem \ref{mainthm:Preimage_submanifold_under_transverse_map_and_implied_embedding}.
\end{proof}

\begin{rmk}[Alternative proof of main conclusion of Corollary \ref{cor:Margalef-Roig_7-1-20}]
\label{rmk:Alternative_proof_corollary_Margalef-Roig_7-1-20}
In $f\transv Y'$ as in the hypothesis of Corollary \ref{cor:Margalef-Roig_7-1-20}, then $f\transv_{x_0} Y'$ for all $x_0\in f^{-1}(Y')$ and thus Theorem \ref{mainthm:Preimage_submanifold_under_transverse_map_and_implied_embedding} implies that $f^{-1}(Y')$ is a $C^p$ Banach submanifold of $X$, since $V_{x_0}\cap f^{-1}(Y')$ is a $C^p$ Banach submanifold of the open neighborhood $V_{x_0}\subset X$ and $x_0\in f^{-1}(Y')$ is arbitrary. The fact that $f^{-1}(Y')$ is also neat is given by Item \eqref{item:corollary_Margalef-Roig_7-1-20_a} in Corollary \ref{cor:Margalef-Roig_7-1-20}. 
\end{rmk}  


\section{Splicing map for connections}
\label{sec:Splicing_map_connections}

\subsection{Mass center and scale maps on the affine space of connections over the four-dimensional sphere}
\label{subsec:Mass_center_and_scale_maps_on_space_connections_over_sphere}
In this subsection, we discuss the concept of a \emph{centered connection} over $S^4$ and the action of the subgroup of translations and dilations of the group of conformal transformations of $S^4$, namely $\RR^4\times\RR_+ \subset \Conf(S^4)$, on the affine space of connections over $S^4$.

A choice of frame $v$ in the principal $\SO(4)$-frame bundle, $\Fr(TS^4)$, for $TS^4$, over the north pole $n\in S^4 \cong \RR^4\cup\{\infty\}$ (identified with the origin in $\RR^4$), defines a conformal diffeomorphism,
\begin{equation}
\label{eq:ConformalDiffeo}
\varphi_n:\RR^4 \to S^4\less\{s\},
\end{equation}
that is inverse to a stereographic projection from the south pole $s\in S^4\subset \RR^5$ (identified with the point at infinity in $\RR^4\cup\{\infty\}$). We let $y(\,\cdot\,): S^4\less\{s\}\to\RR^4$ be the corresponding coordinate chart.

\begin{defn}[Center and scale of a connection on a principal $G$-bundle over $S^4$]
\label{defn:Mass_center_scale_connection}
(Compare \cite[Equation (3.10)]{FeehanGeometry} and Taubes \cite[Equation (4.15)]{TauPath}, \cite[Equation (3.10)]{TauFrame}.)
Let $G$ be a compact Lie group and $P$ be a principal $G$-bundle over the four-dimensional sphere $S^4$ with its standard round Riemannian metric $g_\round$ of radius one and $p\in[2,\infty)$. The \emph{center} $z=z[A]\in\RR^4$ and the \emph{scale} $\lambda=\lambda[A] \in \RR_+ = (0,\infty)$ of a non-flat $W^{1,p}$ connection $A$ on $P$ are defined by
\begin{subequations}
\label{eq:Mass_center_and_scale_connection}
\begin{align}
\label{eq:Mass_center_connection}
\Center[A] &:= \left(\int_{\RR^4}|\varphi_n^*F_A(y)|_\delta^2\,d^4y\right)^{-1}
\int_{\RR^4}y|\varphi_n^*F_A(y)|_\delta^2\,d^4y,
\\
\Scale[A]^2 &:= \left(\int_{\RR^4}|\varphi_n^*F_A(y)|_\delta^2\,d^4y\right)^{-1}
\int_{\RR^4} |y - z[A]|^2|\varphi_n^*F_A(y)|_\delta^2\,d^4y,
\label{eq:Scale_connection}
\end{align}
\end{subequations}
where $\delta$ denotes the standard Euclidean metric on $\RR^4$. The connection $A$ is \emph{centered} if $\Center[A] = 0$ and $\Scale[A] = 1$. If $A$ is flat, one defines $\Center[A] := 0$ and $\Scale[A] := 0$.
\end{defn}

\begin{rmk}[Normalization constants in the definition of mass center and scale]
\label{rmk:Normalization_constants}
The choice of normalization constant in Definition \ref{defn:Mass_center_scale_connection} is consistent with \cite[Equations (29.44) and (29.45)]{Feehan_yang_mills_gradient_flow} and Taubes \cite[Equation (4.15)]{TauPath}, but differs in general from those of \cite[Equation (3.10)]{FeehanGeometry} or Taubes \cite[pp. 343--344]{TauFrame}.
\end{rmk}

\begin{rmk}[Round versus Euclidean metrics in the definition of mass center and scale]
\label{rmk:Round_vs_Euclidean_metrics_in_definition_mass_center_and_scale}
It is possible, as in Taubes \cite[Equation (4.15)]{TauPath}, to use the pullback to $\RR^4$ of the standard round metric of radius one on $S^4$ when defining the integrals in \eqref{eq:Mass_center_and_scale_connection}. However, is then more difficult to show that it is possible to center a non-centered connection $A$ on $P$ and the relationship between the connection $A$ and the required conformal diffeomorphism of $\RR^4$ is not explicit as it is in Lemma \ref{lem:Centering_connection_over_4sphere}, which may be compared with Taubes \cite[Lemma 4.11]{TauPath}.
\end{rmk}

For any $(z,\lambda)\in\RR^4\times\RR_+$, we define a conformal diffeomorphism of $\RR^4$ by
\begin{equation}
\label{eq:CenterScaleDiffeo}
h_{z,\lambda}:\RR^4\to\RR^4,\qquad y\mapsto (y-z)/\lambda.
\end{equation}
It is convenient to view this as a composition of translation, $\tau_z(y) = y - z$, and dilation, $\delta_\lambda(y) := y/\lambda$, that is, $h_{z,\lambda} = \delta_\lambda(\tau_z(y))$. The group $\RR^4\times\RR_+$ acts on $\sA(\varphi_n^*P)$ and $\sA(P)$ by pullback and composition with $\varphi_n:\RR^4\to S^4\less\{s\}$, that is,
\begin{align*}
\sA(\varphi_n^*P) \times \RR^4\times\RR_+ \ni (\varphi_n^*A, z,\lambda)
&\mapsto
h_{z,\lambda}^*\varphi_n^*A \in \sA(\varphi_n^*P),
\\
\sA(P) \times \RR^4\times\RR_+ \ni (A,z,\lambda)
&\mapsto
\tilde h_{z,\lambda}^*A \in \sA(P),
\end{align*}
where
\[
\tilde h_{z,\lambda}^*A := \varphi_n^{-1,*} h_{z,\lambda}^* \varphi_n^*A.
\]
The group $\RR^4\times\RR_+$ also acts on $\Aut(\varphi_n^*P)$ and $\Aut(P)$ by pullback and descends to an action on the quotient spaces $\sB(\varphi_n^*P)$ and $\sB(P)$. We have the following simpler analogue of Taubes \cite[Lemma 4.11]{TauPath}.

\begin{lem}[Centering a connection over $S^4$]
\label{lem:Centering_connection_over_4sphere}
(See Feehan \cite[Lemma 4.4]{Feehan_yangmillsenergy_lojasiewicz4d_v1}.)
Let $G$ be a compact Lie group, $P$ be a smooth principal $G$-bundle over $(S^4,g_\round)$, and $A$ be a non-flat $W^{1,p}$ connection on $P$ with $p \in [2,\infty)$. If $(z,\lambda) \in \RR^4\times\RR_+$, then
\[
\Center[\tilde h_{z,\lambda}^*A] = \lambda\Center[A] + z
\quad\text{and}\quad
\Scale[\tilde h_{z,\lambda}^*A] = \lambda\Scale[A].
\]
In particular, if $z = z[A]$ and $\lambda = \lambda[A]$, then $\tilde h_{z,\lambda}^{-1,*}A = (\tilde h_{z,\lambda}^{-1})^*A$ is a centered connection on $P$.
\end{lem}

\begin{lem}[Smoothness of the center and scale maps]
\label{lem:Smoothness_center_scale_maps}
(See Feehan \cite[Propositions 4.5 and 4.9]{Feehan_yangmillsenergy_lojasiewicz4d_v1} for statements and proofs of continuity and differentiability for the center and scale maps.)
Let $G$ be a compact Lie group and $P$ be a non-product smooth principal $G$-bundle over $(S^4,g_\round)$. If $p \in [2,\infty)$, then the following map is smooth: 
\begin{equation}
\label{eq:Center_scale_maps}  
\sA(P) \ni A \mapsto (\Center[A],\Scale[A]) \in \RR^4\times\RR_+.
\end{equation}
\end{lem}

\begin{lem}[Codimension-five submanifold of centered connections]
\label{lem:Submanifold_centered_connections}
(See Feehan \cite[Section 4]{Feehan_yangmillsenergy_lojasiewicz4d_v1} for related calculations.)
Let $G$ be a compact Lie group and $P_1$ be a non-product smooth principal $G$-bundle over $(S^4,g_\round)$. If $p \in [2,\infty)$, then the map \eqref{eq:Center_scale_maps} is a smooth submersion and the subset of \emph{centered connections} on $P$,
\begin{equation}
\label{eq:Centered_connections}  
\sA^\diamond(P) := \left\{A \in \sA(P): (\Center[A],\Scale[A]) = (\bzero,1)\right\},
\end{equation}
is a smooth submanifold of $\sA(P)$ of codimension five. 
\end{lem}

\subsection{Anti-self-dual connections over the four-dimensional sphere}
\label{subsec:Anti-self-dual_connections_over_4-sphere}
While we shall need to draw on facts concerning the moduli space of anti-self-dual connections over $S^4$ until we complete our proof of Theorem \ref{mainthm:Gluing} in Section \ref{sec:Completion_proof_main_gluing_theorem}, it is convenient review them here.
For any compact Lie group $G$ and principal $G$-bundle $P$ over $S^4$ with \emph{instanton number} $\kappa(P)$ is greater than or equal to $k_G$, a positive integer determined by $G$ (see, for example, \cite[Theorem 8.4]{AHS} for the computation of $k_G$ for all compact simple Lie groups), the moduli space $M(P,g_\round)$ of anti-self-dual connections with minimal stabilizer (equal to the center of $G$) is a non-empty, smooth manifold by virtue of the construction due to Atiyah, Hitchin, Drinfel$'$d, and Manin \cite{AtiyahGeomYM, ADHM, Atiyah_Hitchin_Singer_1977, AHS}; see also Bernard, Christ, Guth, and Weinberg \cite{Bernard_Christ_Guth_Weinberg_1977}. When $G=\SU(2)$, then $k_G=1$ and $M(P,g_\round)$ is diffeomorphic to the five-dimensional, open unit ball in $\RR^5$ when $\kappa(P)=1$. More generally, $M(P,g_\round)$ has dimension $8k-3$ when $G=\SU(2)$ and $\kappa(P)=k\geq 1$.

When $G$ is one of the classical Lie groups (namely, $\SU(n)$, $\SO(n)$, or $\Sp(n)$), we recall from Donaldson \cite{DonInstGeomInvar} or Donaldson and Kronheimer \cite[Section 3.3 and Theorem 3.3.8]{DK}, that the moduli space of instantons $M(P,g_\round)$ over $S^4\cong \RR^4\cup\{\infty\}$ is most naturally viewed, due to their identification with moduli spaces of stable holomorphic vector bundles, as a \emph{framed} (or \emph{based}) moduli space $M_0(P,g_\round)$ comprising pairs of anti-self-dual connections $A$ on $P$ and fiber points $p_s \in P|_s = P|_\infty$, modulo the action of automorphisms of $P$. The moduli space $M_0(P,g_\round)$ has dimension $8k$ when $G=\SU(2)$ and $\kappa(P)=k\geq 1$.

\subsection{Definition of the splicing map}
We continue the notation of Sections \ref{subsec:Local_Kuranishi_parameterization_neighborhood_interior_point_Taubes_approach} and \ref{subsec:Local_Kuranishi_parameterization_neighborhood_boundary_point}. In order to define the splicing map $\cS$, we first choose a point $x\in X$, a positive constant $\delta$ such that $8\sqrt{\delta} < \Inj(X,g)$, that is, $\delta < \Inj(X,g)^2/64$. Here, $\Inj_x(X,g)$ denotes the injectivity radius of $(X,g)$ at a point $x\in X$ and $\Inj(X,g) = \inf_{x\in X}\Inj_x(X,g)$ denotes the injectivity radius of $(X,g)$. The injectivity radius of $(S^4,g_\round)$ is equal to $\pi$ and so, if necessary, we also shrink $\delta$ so that $\delta<\pi^2/64$. Choose points $p_0(x) \in P_0|_x$ and $p_1 \in P_1|s$. We choose smooth reference connections $A_{0\flat}$ on $P_0$ and $A_{1\flat}$ on $P_1$ and construct smooth local sections $\varsigma_0$ of $P_0 \restriction B_\varrho(x)$ and $\varsigma_1$ of $P_1\restriction S^4\less\{n\}$ by parallel translation of $p_0$ and $p_1$ along radial geodesics emanating from $x$ and $s$, respectively, where we denote $\varrho=\Inj(X,g)$ for convenience. We choose an oriented, orthonormal frame $v(x)$ for $TX|_x$ and hence define an inverse local coordinate chart
\[
  \varphi_x = \exp_{v(x)}:TX|_x \supset B_\varrho(x) \to X.
\]
We fix, once and for all, an oriented, orthonormal frame $v(s)$ for $TS^4|_s=\RR^4$, and hence define an inverse local coordinate chart
\[
  \varphi_s = \exp_{v(s)}:TS^4|_s \supset B_\pi(s) \cong S^4\less\{n\} \subset S^4.
\]
Given constants $0<r_0<r_1<\varrho$, we let
\[
  \Omega(x;r_0,r_1) = \{y\in X: r_0 < \dist_g(y,x) < r_1\} = B_{r_1}(x) \less \bar B_{r_0}(x) \subset X
\]
denote the open annulus with radii $r_0<r_1$ and center $x$ and similarly define $\Omega(s;r_0,r_1) \subset S^4$. When $X=\RR^4$ and $x_0$ is the origin, we simply write $B_r = B_r(0)$ and
\[
  \Omega(r_0,r_1) = \{y\in \RR^4: r_0 < |y| < r_1\} = B_{r_1} \less \bar B_{r_0} \subset \RR^4.
\]
This annulus is simply connected (in fact, strongly simply connected in the sense of \cite[p. 161]{DK}), so there exists a smooth trivialization
\[
  P_0 \restriction \Omega(x;r_0,r_1) \cong \Omega(x;r_0,r_1) \times G.
\]
Recall that $p \in (2,\infty)$ is a constant. It is a consequence of the proof of Uhlenbeck's local Coulomb gauge estimate \cite[Theorem 1.3 or Theorem 2.1 and Corollary 2.2]{UhlLp} for the unit ball that if $A_0$ is a $W^{1,p}$ connection on $P_0$ that obeys
\[
  \|F_{A_0}\|_{L^2(\Omega(x;r_0,r_1))} < \eps,
\]
for small enough $\eps = \eps(g,G,r_0,r_1) \in (0,1]$, then there exist a constant $C=C(g,G,p,r_0,r_1)\in[1,\infty)$ and a $W^{2,p}$ local gauge transformation $u_0:\Omega(x;r_0,r_1)\to P_0$ such that
\[
  d^*u_0(\varsigma_0^*A_0) = 0 \quad\text{and}\quad \|u_0(\varsigma_0^*A_0)\|_{W^{1,p}(\Omega(x;r_0,r_1))} \leq C\|F_{A_0}\|_{L^p(\Omega(x;r_0,r_1))}.
\]
See Marini \cite{Marini_1992, Marini_1999, Marini_2000} and Wehrheim \cite{Wehrheim_2004} for statements of this kind.
If we assume $r_1=64r_0$ then, by conformal invariance in $\RR^4$ (with its standard Euclidean metric) and estimates for the Riemannian metric $g$ in local coordinates defined by its exponential map \cite{Aubin_1998},
the constants $\eps$ and $C$ are independent of $r_0,r_1$. For convenience, we write $\sigma_0 = \varsigma_0\cdot u_0$, so that $u_0(\varsigma_0^*A_0) = \sigma_0^*A_0$. The proof of existence of $u_0$ relies on the Implicit Mapping Theorem, so $u_0 \in W^{2,p}(\Omega(x;r_0,r_1);G)$
varies smoothly with the $W^{1,p}$ connection $A_0$
The analogous remarks apply to a $W^{1,p}$ connection $A_1$ on $P_1$.

\begin{defn}[Splicing map]
\label{defn:Splicing_map_connections} 
\begin{equation}
\label{eq:Splicing_map_connections}
\cS: \sA(P_0) \times \sA(P_1) \times \Gl_{x_{0\flat}} \times B_\delta(x_{0\flat}) \times (0,\lambda_0)
\ni (A_0,A_1,\rho,x_0,\lambda) \mapsto A \in \sA(P).
\end{equation}
\end{defn}

Our definition of the splicing map \eqref{eq:Splicing_map_connections} is a combination of those described by Donaldson and Kronheimer \cite{DK} and Taubes \cite{TauSelfDual, TauIndef, TauFrame}.

Note that to define $\cS$ and compute and estimate its derivatives with respect to all parameters, we shall need to restrict our attention to open subsets of connections $A_0 \in \sA(P_0)$ and $A_1 \in \sA(P_1)$ with
\[
  \|F_{A_0}\|_{L^2(B_\delta(x_{0\flat}))} < \eps \quad\text{and}\quad \|F_{A_1}\|_{L^2(B_\delta(s))} < \eps.
\]
This involves no loss in generality for our application and avoids our having to consider annuli in our definition of the domain space, though we have to consider them in the codomain space when checking surjectivity onto an open subset.

For the purpose of computing and estimating derivatives of $\cS$, we can adopt simpler local approaches that avoid having sections $\sigma_i$ and trivializations $\tau_i$ of $P_i$ depend on the variable connections $A_i$, for $i=0,1$:
\begin{itemize}
\item Fix local sections $\sigma_i$ and restrict attention to connections $A_i$ with $\|\sigma_i^*A_i\|_{W^{1,2}(B_i)} < \zeta$, where $\zeta\in(0,1]$ is small;
\item Fix local sections $\sigma_i$ and smooth reference connections $A_i^o$ with $\|\sigma_i^*A_i^o\|_{W^{1,p}(B_i)} < \zeta$ and then only consider open balls of $W^{1,p}$ connections $A_i$ with $\|A_i-A_i^o\|_{W^{1,p}(X)} < \eta$, where $\eta\in(0,1]$ is small;
\item While we do want to ultimately consider $x\in X$ and $\lambda\in(0,\delta]$ as parameters, we are not forced to do so and could instead consider the complete family of $W^{1,p}$ connections $A_1$ on $P_1$, rather than restrict to \emph{centered} connections.   
\end{itemize}
We do not need to make these simplifications for the purpose of defining $\cS$ or proving continuity of $\cS$, only for the purpose of computing and estimating derivatives.

Note that the space $\Isom(P|_{x_0},P_1|_s)$ of maps that are equivariant with respect to the right action of $G$ on the fibers of $P_0$ and $P_1$ may be (non-canonically) identified with a copy of $G$ (via a choice of point $p_0\in P_0|_{x_0}$).

In Feehan \cite{FeehanGeometry} and Peng \cite{Peng_1996} (though apparently not in Peng \cite{Peng_1995}), we also pulled back the connections $A_1$ on $S^4\less \varphi_s(B_{2\sqrt{\lambda}})$ via conformal maps to small balls $B_{\sqrt{\lambda}/2}(x_0) \subset X$. This step was done because the goal was to compute and estimate components of the $L^2$ metric on the moduli spaces of anti-self-dual connections on $P$, which is not conformally invariant. However, since that is not our goal here, we do not perform these pullbacks.


\subsection{Connected-sum manifold}
Let $\HH$ be the four-dimensional division algebra of quaternions and $\HH\PP^1 = S^4$ be right quaternionic projective space, with coordinate patches $U_n = \{[p,q]\in\HH^2:q\neq 0\} = S^4\less\{s\}$ around the north pole $n=[0,1]$ and $U_s = \{[p,q]\in\HH^2:p\neq 0\} = S^4\less\{n\}$ around the south pole $s=[1,0]$. We identify $\HH=\RR^4$ as inner product spaces and let $\varphi_n^{-1} : S^4\less\{s\} \ni [p,q] \mapsto pq^{-1} \in \RR^4$ and $\varphi_s^{-1} : S^4\less\{n\} \ni [p,q] \mapsto qp^{-1} \in \RR^4$ denote the standard local coordinate charts. Note that $\varphi_s^{-1} \circ \varphi_n(x) = \iota(x)$ for all $x\less\HH\less\{0\}$, where the inversion map $\iota: \HH\less\{0\} \to \HH\less\{0\}$ is given by $x \mapsto x^{-1} = \bar x/|x|^2$ and $\bar x := (x_1,-x_2,-x_3,-x_4)$. 

We first assume that $g$ is flat near $x_0$ and identify the geodesic ball $B_r(x_0) \subset X$ with $B_r(0) \subset (TX)_{x_0}$ and $(TX)_{x_0}$ with $\RR^4$ via a choice of oriented, orthonormal frame $v_0$ for $(TX)_{x_0}$. Following the recipe in \cite[Section 7.2.1]{DK}, we construct a connected sum $X\# S^4$ as $X_0'\cup_{f_{x_0,\lambda}} X_1'$, where $X_0=X$, and $X_1=S^4$, and $X_0' = X \less B_{\sqrt{\lambda}/2}(x_0)$, and $X_1' = S^4 \less \varphi_s(B_{\sqrt{\lambda}/2})$. The orientation-preserving diffeomorphism,
\[
  f_{x_0,\lambda}: \Omega(x_0;\sqrt{\lambda}/2, 2\sqrt{\lambda}) \cong \varphi_s(\Omega(\sqrt{\lambda}/2, 2\sqrt{\lambda})),
\]
is defined by composition of the map $x \mapsto c_\lambda^{-1}\circ\iota(x) = \lambda/x$ with the local coordinate charts, where $c_\lambda(x) := x/\lambda$ for all $\lambda\in(0,\infty)$ and $x\in\RR^4$. The inverse local coordinate charts are $\varphi_0 = \exp_{v_0}: (TX)_{x_0} \supset B_\varrho(0) \to B_\varrho(x_0) \subset X$ and $\varphi_s: \RR^4 \to S^4 \less\{n\}$. Therefore
\[
  f_{x_0,\lambda} = \varphi_s \circ c_\lambda^{-1} \circ \iota \circ \varphi_0^{-1}: B_r(x_0) \cong S^4\less\varphi_s(\bar B_{\lambda/r}),
\]
for any $r\in (0,\Inj(X,g))$. Using $\varphi_s^{-1} \circ \varphi_n(x) = x^{-1} = \iota(x)$ and $(c_\lambda^{-1} \circ \iota)(x) = \lambda/x = (\iota \circ c_\lambda)(x)$, for all $x\in\HH\less\{0\}$, we can also write
\[
f_{x_0,\lambda} = \varphi_n \circ c_\lambda \circ \varphi_0^{-1}: B_r(x_0) \cong \varphi_n(B_{r/\lambda}).
\]
In particular, this gives an orientation-preserving diffeomorphism
\[
f_{x_0,\lambda} = \varphi_n \circ c_\lambda \circ \varphi_0^{-1}: B_{2\sqrt{\lambda}}(x_0) \cong \varphi_n(B_{2/\sqrt{\lambda}})
\]
or, equivalently,
\[
  f_{x_0,\lambda} = \varphi_s \circ c_\lambda^{-1} \circ \iota \circ \varphi_0^{-1}: B_{2\sqrt{\lambda}}(x_0)
  \cong S^4\less\varphi_s(\bar B_{\sqrt{\lambda}/2}),
\]
The preceding orientation-preserving diffeomorphism serves to define the oriented, smooth connected sum $X\# S^4$.

\subsection{Cutoff functions}
Given a point $x_0 \in X$ and a constant $r \in (0,\Inj(X,g))$, we define a smooth cutoff function $\chi_{x_0,r}:X\to [0,1]$ by setting
\begin{equation}
\label{eq:Cutoff_function_zero_on_half-ball_one_on_manifold_minus_ball}
\chi_{x_0,r}(x)
:=
\kappa(\dist_g(x,x_0)/r), \quad\forall\, x\in X,
\end{equation}
where $\kappa:\RR\to [0,1]$ is a smooth function such that $\kappa(t)=1$ for $t\geq 2$ and $\kappa(t)=0$ for $t\leq 1/2$.  Thus, we have
\[
\chi_{x_0,r}(x)
=
\begin{cases}
1 &\text{for } x\in X \less B_{2r}(x_0),
\\
0 &\text{for } x\in B_{r/2}(x_0).
\end{cases}
\]
An elementary calculation yields (see \cite[Lemma 5.8]{FLKM1} for example),
\begin{equation}
\label{eq:L4_bound_dchi}
\|d\chi_{x_0,r}\|_{L^4(X)} \leq C,
\end{equation}
where the constant $C=C(g) \in [1, \infty)$ is independent of $x_0 \in X$ and $r \in (0,\varrho)$, but rather depends only on the fixed universal choice of $\kappa$ via $|d\kappa| \leq 1$, and the injectivity radius $\varrho = \Inj(X,g)$.

\subsection{Riemannian metric on the connected sum of a four-manifold and four-sphere}
\label{subsec:Riemannian_metric_connected_sum_X_4-sphere}
One can define a metric $g_\lambda$ on $X\# S^4$ that is conformally equivalent to $g$ using Peng \cite[Equation (2.6)]{Peng_1995} when $g$ is flat near $x_0$. In Feehan \cite[Definition 3.11]{FeehanGeometry}, irrespective of whether $g$ is flat near $x_0$, we define a conformal structure on $X\# S^4$, though not an actual metric, since we do not interpolate between the almost round metric $\tilde g_\round$ on $S^4$ near $s$ and the possibly non-flat metric $g$ on $X$ near $x_0$. Recall that (see Jost \cite[Equation (1.4.39)]{Jost_riemannian_geometry_geometric_analysis_e7})
\begin{equation}
  \label{eq:Feehan_1995_3-1}
  \varphi_n^*g_\round(x) = \frac{4\delta_{\mu\nu}}{1+|x|^2}dx^\mu dx^\nu, \quad\forall\, x \in \RR^4.
\end{equation}
Following Peng \cite[Equation (2.6)]{Peng_1995}, we set
\begin{equation}
\label{eq:Peng_1995_2-6_locally_flat_metric}
g_\lambda
:=
\begin{cases}
g,  &\text{on } X \less B_{2\sqrt{\lambda}}(x_0),
\\
\chi_{x_0,\sqrt{\lambda}}\,g + (1-\chi_{x_0,\sqrt{\lambda}})\,f_{x_0,\lambda}^*g_\round,  &\text{on } \Omega(x_0; \sqrt{\lambda}/2, 2\sqrt{\lambda}),
\\
g_\round,  &\text{on } S^4 \less \varphi_s(B_{2\sqrt{\lambda}}),
\end{cases}
\end{equation}
When $\lambda\downarrow 0$, the connected sum $X\# S^4$ becomes a copy of $X$ and $S^4$ with $x_0\in X$ identified with $s \in S^4$ and $g_\lambda$ becomes a copy of $g$ on $X$ and $g_\round$ on $S^4$. When $g$ is non-flat, we set 
\begin{equation}
\label{eq:Peng_1995_2-6_non-flat_metric}
g_\lambda
:=
\begin{cases}
g,  &\text{on } X \less B_{2\sqrt{\lambda}}(x_0),
\\
\chi_{x_0,\sqrt{\lambda}}\,g + (1-\chi_{x_0,\sqrt{\lambda}})\,f_{x_0,\lambda}^*\tilde g_\round,  &\text{on } \Omega(x_0; \sqrt{\lambda}/2, 2\sqrt{\lambda}),
\\
\tilde g_\round,  &\text{on } S^4 \less \varphi_s(B_{2\sqrt{\lambda}}),
\end{cases}
\end{equation}
Here, the almost round metric on $S^4$ depends on $x_0$, $g$, and $\lambda$ by writing $\varphi_0^*g(x) = g_{\mu\nu}(x)dx^\mu dx^\nu$ on $B_\varrho(x_0)$ and setting
\begin{equation}
  \label{eq:Feehan_1995_3-33}
  \varphi_n^*\tilde g_\round(x) := \frac{4g_{\mu\nu}(\lambda x)}{1+|x|^2}dx^\mu dx^\nu, \quad\forall\, x \in B_{\varrho/\lambda},
\end{equation}
as in Feehan \cite[Definition 3.11]{FeehanGeometry} or Peng \cite[pp. 153--154]{Peng_1995}. Note that
\[
  c_\lambda^*\varphi_n^*\tilde g_\round(x) = \frac{4\lambda^{-2}g_{\mu\nu}(x)}{1+|x/\lambda|^2}dx^\mu dx^\nu
  = \frac{4g_{\mu\nu}(x)}{\lambda^2+|x|^2}dx^\mu dx^\nu, \quad\forall\, x \in B_\varrho.
\]
Thus, $f_{x_0,\lambda}^*\tilde g_\round = (\varphi_n \circ c_\lambda \circ \varphi_0^{-1})^*\tilde g_\round = (\varphi_0^{-1})^*c_\lambda^*\varphi_n^*\tilde g_\round$ is conformally equivalent, though not equal, to $g$ on $B_\varrho(x_0)$. Recall from Aubin \cite[Corollary 1.32]{Aubin_1998} or Jost \cite[Theorem 1.4.4]{Jost_riemannian_geometry_geometric_analysis_e7} that $g_{\mu\nu}(0) = \delta_{\mu\nu}$ and $(\partial g_{\mu\nu}/\partial x^\alpha)(0) = 0$ for all indices $\mu,\nu,\alpha$, and so
\begin{subequations}
  \begin{align}
  |g_{\mu\nu}(x) - \delta_{\mu\nu}| &= O(|x|^2),
  \\
  \left|\frac{\partial g_{\mu\nu}}{\partial x^\alpha}(x)\right| &= O(|x|), \quad\forall\, x \in B_\varrho,
\end{align}
\end{subequations}
since $\varphi_0^{-1}$ is a geodesic normal coordinate chart on $B_\varrho(x_0)$ for the metric $g$.

\subsection{Spliced principal $G$-bundle}
We define a spliced principal $G$-bundle $P$ over $X$ by setting
\begin{equation}
\label{eq:Spliced_principal_G-bundle_P}
P
:= 
\begin{cases}
  P_0 &\text{over } X \less B_{\sqrt{\lambda}/2}(x_0),
  \\
  P_1 &\text{over } S^4 \less \varphi_s(B_{\sqrt{\lambda}/2}).
\end{cases}
\end{equation}
The bundles $P_0$ and $P_1$ are identified over the annulus $\Omega(x_0,\sqrt{\lambda}/2,2\sqrt{\lambda})$ in $X$ via the isomorphisms of principal $G$-bundles defined by the orientation-preserving diffeomorphism $f_{x_0,\lambda}$ that identifies the annulus $\Omega(x_0,\sqrt{\lambda}/2,2\sqrt{\lambda})$ with the annulus $\varphi_s(\Omega(\sqrt{\lambda}/2,2\sqrt{\lambda}))$ in $S^4$ and the $G$-bundle map defined by the trivializations $\tau_0$ and $\tau_1$ determined by the sections $\sigma_0$ and $\sigma_1$.

\subsection{Splicing map for connections}
\label{subsec:Splicing_map}
We define a cutoff connection on $P_0$ by
\begin{equation}
\label{eq:Cutoff_connection_A0}
A_0'
:=
\begin{cases}
A_0,  &\text{over } X \less B_{2\sqrt{\lambda}}(x_0),
\\
\Theta + \chi_{x_0,\sqrt{\lambda}}\,\sigma_0^*A_0,  &\text{over } \Omega(x_0; \sqrt{\lambda}/2, 2\sqrt{\lambda}),
\\
\Theta,  &\text{over } B_{\sqrt{\lambda}/2}(x_0),
\end{cases}
\end{equation}
where $\Theta$ is the product connection on $B_\varrho(x_0)\times G$. We define a cutoff connection on $P_1$ by
\begin{equation}
\label{eq:Cutoff_connection_A1}
A_1'
:=
\begin{cases}
A_1,  &\text{over } S^4 \less \varphi_s(B_{2\sqrt{\lambda}}(s)),
\\
\Theta + (1-\chi_{s,\sqrt{\lambda}})\sigma_1^*A_1,  &\text{over } \Omega(s; \sqrt{\lambda}/2, 2\sqrt{\lambda}),
\\
\Theta,  &\text{over } B_{\sqrt{\lambda}/2}(s),
\end{cases}
\end{equation}
where $\Theta$ is the product connection on $S^4\less\{n\}\times G$. Finally, we define a spliced principal $G$-bundle $P$ and a spliced connection $A$ on $P$ by
\begin{equation}
\label{eq:Spliced_connection_A}
A
:=
\begin{cases}
A_1,  &\text{over } S^4 \less \varphi_s(B_{2\sqrt{\lambda}}(s)),
\\
\Theta + \chi_{x_0,\sqrt{\lambda}}\,\sigma_0^*A_0 + (1-\chi_{x_0,\sqrt{\lambda}})f_{x_0,\lambda}^*\sigma_1^*A_1,  &\text{over } \Omega(x_0; \sqrt{\lambda}/2, 2\sqrt{\lambda}),
\\
A_0,  &\text{over } X \less B_{\sqrt{\lambda}/2}(x_0),
\end{cases}
\end{equation}
where $\Theta$ is the product connection on $B_\varrho(x_0)\less\{x_0\}\times G$.

\subsection{Unsplicing map and surjectivity of the splicing map for connections}
\label{subsec:Surjectivity_splicing_map_connections}
We shall prove that the splicing map $\cS$ in \eqref{eq:Splicing_map_connections} is surjective by exhibiting an explicit (smooth) \emph{right inverse} $\cU$, called the \emph{unsplicing map}. Since we shall need to refer to the parameters that we fix once and for all in our definitions of splicing and unsplicing maps, we collect these choices in the following data set:

\begin{data}[Fixed auxiliary parameters for the definition of splicing and unsplicing maps]
\label{data:Fixed_parameters_for_definition_splicing_and_unsplicing_maps}
Let $(X,g)$ be a closed, connected, four-dimensional, oriented, smooth Riemannian manifold, $G$ be a compact Lie group, $P_0$ be a smooth principal $G$-bundle over $X$, and $P_1$ be a smooth principal $G$-bundle over $S^4$, where $S^4 = \{x\in\RR^5:|x|=1\}$ is the four-dimensional sphere with its standard round metric $g_\round$ of radius one. Let $x_0 \in X$ be a point, $v$ be an oriented, orthonormal frame for $TX|_{x_0}$, and $p_0 \in P_0|_{x_0}$ and $p_1 \in P_1|_s$ be fiber points (where $s\in S^4$ denotes the south pole), and $A_{0\flat}$ on $P_0$ and $A_{1\flat}$ on $P_1$ be smooth connections, and $\varrho_0 \in (0,1]$ be a constant such that the Riemannian metric $g$ is flat on the geodesic ball $B_{\varrho_0}(x_{0\flat})$.
\end{data}

In applications to topology, as discussed for example by the authors in \cite{FL5}, one needs to allow some of the parameters recorded in Data \ref{data:Fixed_parameters_for_definition_splicing_and_unsplicing_maps} to vary. For now, we note that given $p_0 \in P_0|_{x_0}$, we obtain different points $\rho(p_0) \in P_1|_s$ by varying $\rho \in \Gl_{x_{0\flat}}$; conversely, given $p_1 \in P_1|_s$, we obtain different points $\rho^{-1}(p_1) \in P_0|_{x_0}$ varying $\rho \in \Gl_{x_{0\flat}}$.

In our definition of the splicing map $\cS$ in Section \ref{subsec:Splicing_map}, we used the facts that $P_0 \restriction B_\varrho(x_0) \cong B_\varrho(x_0) \times G$ and $P_1 \restriction S^4 \less\{n\} \cong S^4 \less\{n\}\times G$, where the local trivializations of the principal $G$-bundles are defined by the data $(p_0,p_1,A_{0\flat},A_{1\flat})$. The choice of oriented, orthonormal frame $v$ and resulting geodesic normal coordinate chart and the standard coordinate chart for $S^4$ yield the principal $G$-bundle $P$ over $X\# S^4$ by identifying $P_0$ and $P_1$ over annuli in $X$ and $S^4$, respectively, with $\Omega\times G$, where $\Omega\subset\RR^4$ is an open annulus.

Consider $A\in \sA(P)$. After suppressing notation for the local coordinate charts and local trivializations, we may write
\[
A = \chi_{x_0,\sqrt{\lambda}}\,A + (1-\chi_{x_0,\sqrt{\lambda}})A \quad\text{on } B_\varrho(x_0)\less\{x_0\} \times G,
\]
where $\chi_{x_0,\sqrt{\lambda}}\,A \in \sA(P_0)$ and $(1-\chi_{x_0,\sqrt{\lambda}})\,A \in \sA(P_1)$. Now $\chi_{x_0,\sqrt{\lambda}/4} = 1$ on $X\less B_{\sqrt{\lambda}/2}(x_0)$ (the support of $\chi_{x_0,\sqrt{\lambda}}$) and $\chi_{x_0,\sqrt{\lambda}/4} = 0$ on $B_{\sqrt{\lambda}/8}(x_0)$. On the other hand, $1-\chi_{x_0,4\sqrt{\lambda}} = 1$ on $B_{2\sqrt{\lambda}}(x_0)$ (the support of $1-\chi_{x_0,\sqrt{\lambda}}$) and $1-\chi_{x_0,4\sqrt{\lambda}} = 0$ on $X\less B_{8\sqrt{\lambda}}(x_0)$. Hence, we can write
\begin{align*}
  A &= \chi_{x_0,\sqrt{\lambda}}\,\chi_{x_0,\sqrt{\lambda}/4} A + (1-\chi_{x_0,\sqrt{\lambda}})(1-\chi_{x_0,4\sqrt{\lambda}})A
  \\
  &= \chi_{x_0,\sqrt{\lambda}}\,A_0 + (1-\chi_{x_0,\sqrt{\lambda}})A_1,
\end{align*}
where
\[
  A_0 := \chi_{x_0,\sqrt{\lambda}/4} A \quad\text{and}\quad A_1 := (1-\chi_{x_0,4\sqrt{\lambda}})A.
\]
In particular, $A = \cS(A_0,A_1,\rho,x_0,\lambda)$ and $\cS$ is surjective by construction.

Since these pairs of cutoff functions arise frequently in this article, we abbreviate them as
\begin{equation}
  \label{eq:Partition_unity_and_covering_cutoff_functions}
  \chi_0 := \chi_{x_0,\sqrt{\lambda}}, \quad \chi_1 := 1-\chi_{x_0,\sqrt{\lambda}}, \quad \psi_0 := \chi_{x_0,\sqrt{\lambda}/4}, \quad\text{and}\quad \psi_1 := \chi_{x_0,\sqrt{\lambda}/4}.
\end{equation}
Note that $\chi_0+\chi_1=1$ on $X$ while $\psi_0=1$ on $\supp\chi_0$ and $\psi_1=1$ on $\supp\chi_1$. The assertion that (as a map of sets) $\cS$ is surjective is equivalent to the assertion that it has a right inverse $\cU$:
\[
  \cS\circ\cU = \id \quad\text{on } \sA(P) \times \Gl_{x_{0\flat}}\times B_\delta(x_{0\flat}) \times (0,\lambda_0).
\]
We define a candidate for this right inverse by 
\begin{multline}
\label{eq:Unsplicing_map_connections}
\cU: \sA(P) \times \Gl_{x_{0\flat}}\times B_\delta(x_{0\flat}) \times (0,\lambda_0) \ni (A,\rho,x_0,\lambda)
\\
\mapsto (\psi_0A,\psi_1A,\rho,x_0,\lambda) \in \sA(P_0) \times \sA(P_1) \times \Gl_{x_{0\flat}}\times B_\delta(x_{0\flat}) \times (0,\lambda_0).
\end{multline}
Plainly, $\cU$ is smooth map. It is worth noting that the connection $A_1=\psi_1A\in\sA(P_1)$ need \emph{not} be centered with respect to the north pole in $S^4$. When we eventually apply the Inverse Function Theorem for $C^1$ maps of smooth Banach manifolds with boundary (in the form of Theorem \ref{mainthm:Preimage_submanifold_under_transverse_map_and_implied_embedding}), however, we shall obtain a gluing embedding map whose domain contains a factor corresponding to \emph{centered} anti-self-dual connections on $P_1$ over $S^4$.

\begin{lem}[Right inverse and surjectivity of the splicing map]
\label{lem:Surjectivity_splicing_map}
The map $\cU$ in \eqref{eq:Unsplicing_map_connections} is a right inverse for the splicing map $\cS$ in \eqref{eq:Splicing_map_connections} and thus $\cS$ is surjective.
\end{lem}

\begin{proof}
We observe that
\begin{align*}
  \cS\circ\cU(A) &= (\psi_0A,\psi_1A,\rho_0,x_0,\lambda_0)
  \\
                 &= \chi_0\psi_0A + \chi_1\psi_1A
  \\
  &= \chi_0A + \chi_1A = A,
\end{align*}
and so $\cU$ is a right inverse for $\cS$, as claimed.
\end{proof}

\subsection{Splicing map is a smooth submersion}
\label{subsect:Splicing_map_smooth_submersion}
We have the

\begin{prop}[$C^k$ submersion property of the splicing map]
\label{prop:Smooth_submersion_property_splicing_map}
For any integer $k\in\NN$, the splicing map $\cS$ in \eqref{eq:Splicing_map_connections} is a $C^k$ submersion.
\end{prop}

\begin{proof}  
It suffices to consider partial derivatives of $\cS$ with respect to the factors $A_0$ and $A_1$ for any fixed triple $(\rho,x_0,\lambda) \in \Gl_{x_{0\flat}}\times B_\delta(x_{0\flat})\times (0,\lambda_0)$. Consider directions $a_0 \in W^{1,2}(T^*X\otimes\ad P_0)$ and $a_1 \in W^{1,2}(T^*S^4\otimes\ad P_1)$ and compute the partial derivatives:
\[
  \frac{\partial\cS}{\partial A_0}(A_0,A_1,\rho,x_0,\lambda)a_0 = \chi_0a_0
  \quad\text{and}\quad
  \frac{\partial\cS}{\partial A_1}(A_0,A_1,\rho,x_0,\lambda)a_1 = \chi_1a_1.
\]
Therefore,
\[
  \frac{\partial\cS}{\partial A_0}(A_0,A_1,\rho,x_0,\lambda)a_0 + \frac{\partial\cS}{\partial A_1}(A_0,A_1,\rho,x_0,\lambda)a_1
  =
  \chi_0a_0 + \chi_1a_1.
\]
Observe that $T_A\sA(P) = W^{1,p}(T^*X\otimes\ad P)$, where $A := \cS(A_0,A_1,\rho,x_0,\lambda)$, while $T_{A_0}\sA(P_0) = W^{1,p}(T^*X\otimes\ad P_0)$ and $T_{A_1}\sA(P_1) = W^{1,p}(T^*S^4\otimes\ad P_1)$. If $a \in W^{1,p}(T^*X\otimes\ad P)$, then we may write
\[
  a = \chi_0a_0 + \chi_1a_1,
\]
for $a_0 := \psi_0a \in W^{1,p}(T^*X\otimes\ad P_0)$ and $a_1 := \psi_1a \in W^{1,p}(T^*S^4\otimes\ad P_1)$. Hence, the derivative
\begin{equation}
\label{eq:Derivative_splicing_map_connections}
d\cS:T_{(A_0,A_1,\rho,x_0,\lambda)}\left(\sA(P_0)\times \sA(P_1) \times \Gl_{x_{0\flat}}\times B_\delta(x_{0\flat})\times (0,\lambda_0)\right) \to T_A\sA(P)
\end{equation}
is surjective and $\cS$ is a $C^1$ submersion at $(A_0,A_1,\rho,x_0,\lambda)$. Since the point $(A_0,A_1,\rho,x_0,\lambda)$ is arbitrary and $\cS$ is clearly $C^k$ smooth for any $k\in\NN$, the conclusion follows.
\end{proof}

\subsection{Restriction of the splicing map to the codimension-five submanifold defined by centered connections}
We have seen thus far that the splicing map $\cS$ in \eqref{eq:Splicing_map_connections} is a surjective submersion (Lemma \ref{lem:Surjectivity_splicing_map} and Proposition \ref{prop:Smooth_submersion_property_splicing_map}). However, as we have chosen $(X_1,g_1) = (S^4,g_\round)$, readers familiar with the development by Donaldson and Kronheimer (see \cite[Section 8.2]{DK}) will recognize that there is a redundancy in the domain of $\cS$ provided in \eqref{eq:Splicing_map_connections}: An arbitrary connection $A_1$ on $P_1$ over $S^4$ has an intrinsic center of mass $z_1\in\RR^4$ and intrinsic scale $\nu\in(0,\infty)$ in addition to the extrinsic center $x_0 \in B_\delta(x_{0\flat})$ and intrinsic scale $\lambda\in(0,\lambda_0)$ prescribed in the domain for $\cS$ in \eqref{eq:Splicing_map_connections}. In this subsection, we observe that if we replace the Banach affine space $\sA(P_1)$ in \eqref{eq:Splicing_map_connections} by the codimension-five Banach submanifold $\sA^\diamond(P_1)$ of \emph{centered connections}, then the resulting restriction of $\cS$ remains a surjective submersion. We first prove surjectivity.

\begin{lem}[Surjectivity of the centered splicing map]
\label{lem:Surjectivity_splicing_map_centered_connections_4-sphere}
Continue the notation of this section. Then the \emph{centered splicing map}, given by composition of the splicing map $\cS$ in \eqref{eq:Splicing_map_connections}, restricted by replacing $\sA(P_1)$ with $\sA^\diamond(P_1)$, and the conformal diffeomorphism of $(S^4,g_\round)$ defined by translation and dilation of $\RR^4$ and stereographic projection from the south pole of $S^4$,
\begin{equation}
\label{eq:Splicing_map_connections_centered_connections_4-sphere}
\cS\circ(\cR\times\id): \sA(P_0) \times \sA^\diamond(P_1) \times \Gl_{x_{0\flat}} \times B_\delta(x_{0\flat}) \times (0,\lambda_0)
\ni (A_0,A_1^\diamond,\rho,x_0,\lambda) \mapsto A \in \sA(P),
\end{equation}
is a smooth surjective map, where
\[
  A := \cS(A_0,\cR(A_1^\diamond,x_0,\lambda),\rho,x_0,\lambda) = \chi_0A_0 + \chi_1A_1 \in \sA(P),
\]
and $\chi_0,\chi_1$ are determined by $(x_0,\lambda)$ as in \eqref{eq:Partition_unity_and_covering_cutoff_functions}, and
\[
  \cR: \sA^\diamond(P_1) \times \RR^4 \times (0,\infty) \ni (A_1^\diamond,x_0,\lambda)
  \mapsto A_1 = \tilde{h}_{x_0,\lambda}^*A_1^\diamond \in \sA(P_1).
\]
\end{lem}

\begin{proof}
Let $q = p^* = 4p/(4-p)\in [4,\infty)$ when $p\in[2,4)$ or $q \in [4,\infty)$ when $p\geq 4$ and $A \in \sA(P)$ be a connection that is $L^q$-close, in the sense of \cite[Section 7.3.1]{DK}, to a connection $A_0$ over an open subset $X''(\eta) := X\less\bar B_\eta(x_{0\flat}) \subset X$ for a small positive constant $\eta \in (0,\delta]$. Construct a cutoff function $\zeta_1 \in C^\infty(X;[0,1])$ such that $\zeta_1 = 1$ on $B_{\eta/2}(x_{0\flat})$ and $\zeta_1 = 0$ on $X\less B_{2\eta}(x_{0\flat})$. Recall that the smooth principal $G$-bundle $P$ is constructed (up to an isomorphism of smooth principal $G$-bundles) using the Data \ref{data:Fixed_parameters_for_definition_splicing_and_unsplicing_maps} and a triple $(\rho,x_{0\flat},\eta) \in \Gl_{x_{0\flat}}\times B_\delta(x_{0\flat})\times (0,1]$. We obtain a connection $A_1 := \zeta_1A$ on the principal $G$-bundle $P_1$ over $S^4$ with the aid of the trivialization $P \restriction \Omega(x_{0\flat};\eta/2, 2\eta)\times G$ and cutting off over $A$ over the annulus $\Omega(x_{0\flat};\eta/2, 2\eta) \subset X$. We next define
\[
  z := \Center[A_1] \in \RR^4 \quad\text{and}\quad \nu := \Scale[A_1] \in (0,\infty)
\]
and obtained a \emph{centered} connection $A_1^\diamond := (\tilde{h}_{z,\nu}^{-1})^*A_1$ on $P_1$ by pulling back $A_1$ via the inverse of the conformal diffeomorphism $\tilde{h}_{z,\nu}$ of $S^4$ defined by $h_{z,\nu}(y)=(y-z)/\nu$, for $y\in\RR^4$. We may assume without loss of generality that $z\in B_\delta$ and $\nu \in (0,1]$ and, in particular that $z$ is close to the origin in $T_{x_{0\flat}}X \cong \RR^4$ (isometric isomorphism), which is in turn identified with the point $x_{0\flat} \in X$, and that $\sqrt{\nu} \ll \eta$. We now write $x_0 := \exp_v(z) \in X$ and $\chi_0 := \chi_{x_0,\nu} \in C^\infty([0,1];X)$ (as in \eqref{eq:Partition_unity_and_covering_cutoff_functions} but with $\lambda$ replaced by $\nu$) and $\chi_1 := 1 - \chi_0 \in C^\infty([0,1];X)$ and unsplice $A$ by writing
\[
  A = \chi_0A + \chi_1A.
\]
But $\supp\chi_1 \subset \{\zeta_1=1\}$ and so we have $\chi_1A = \chi_1A_1$; similarly, we may choose $\zeta_0 \in C^\infty([0,1];X)$ such that $\supp\zeta_0 \Subset X\less\{x_{0\flat}\}$ but $\supp\chi_0 \subset \{\zeta_0=1\}$, define $A_0 := \zeta_0A$, and hence obtain
\[
  A = \chi_0A_0 + \chi_1A_1 = \chi_0A_0 + \chi_1(\tilde{h}_{z,\nu}^{-1})^*A_1^\diamond,
\]
as claimed.
\end{proof}

We next prove that the centered splicing map is a smooth submersion.

\begin{prop}[$C^k$ submersion property of the centered splicing map]
\label{prop:Smooth_submersion_property_splicing_map_centered_connections_4-sphere}
For any integer $k\in\NN$, the centered splicing map $\cS_\diamond$ in \eqref{eq:Splicing_map_connections_centered_connections_4-sphere} is a $C^k$ submersion.
\end{prop}

\begin{proof}
Let $(A_0,A_1^\diamond,\rho,y_0,\nu)$ be a point in the domain of $\cS_\diamond$ in \eqref{eq:Splicing_map_connections_centered_connections_4-sphere} and let $A_1 := \tilde{h}_{y_0,\nu}^*A_1^\diamond$. Recall from the proof of Proposition \ref{prop:Smooth_submersion_property_splicing_map} that in order to verify that the derivative map $d\cS(A_0,A_1,\rho,x_0,\lambda)$ in \eqref{eq:Derivative_splicing_map_connections} is surjective, it sufficed to consider partial derivatives of $\cS$ with respect to
\begin{enumerate}
\item $A_0$ in directions $a_0 \in T_{A_0}\sA(P_0) = W^{1,p}(T^*X\otimes\ad P_0)$, and
\item $A_1$ in directions $a_1 \in T_{A_1}\sA(P_1) = W^{1,p}(T^*S^4\otimes\ad P_1)$.  
\end{enumerate}
Because the map
\[
  \sA^\diamond(P_1) \times \RR^4 \times (0,\infty) \ni (A_1^\diamond,z,\nu)
  \mapsto \tilde{h}_{z,\nu}^*A_1^\diamond \in \sA(P_1)
\]
is a diffeomorphism, the conclusion follows from Proposition \ref{prop:Smooth_submersion_property_splicing_map}.
\end{proof}

\subsection{Injectivity of restrictions of splicing maps to finite-dimensional submanifolds}
\label{subsec:Injectivity_restrictions_splicing_maps_finite-dimensional_submanifolds}
It is important to note that the splicing map $\cS$ for connections in \eqref{eq:Splicing_map_connections} cannot be injective, even after restricting it to the smaller (but still infinite-dimensional) domain in \eqref{eq:Splicing_map_connections_centered_connections_4-sphere}. On the other hand, we shall see below, the splicing maps do become injective when restricted to finite-dimensional submanifolds their domains.

\begin{prop}[Injectivity of restrictions of splicing maps to finite-dimensional submanifolds]
\label{prop:Injectivity_restrictions_splicing_maps_finite-dimensional_submanifolds}  
Let $G$ be a compact Lie group, $P_i$ be smooth principal $G$-bundles over admissible Riemannian four-manifolds $(X_i,g_i)$, and $M_i \subset \sA(P_i)$ be finite-dimensional smooth submanifolds. Assume given the splicing data in Data \ref{data:Splicing_parameters_connected_sum_G-bundle_connected_sum_Riemannian_4-manifold} and that the metrics $g_i$ are conformally flat near the points $x_{i\flat} \in X_i$ for $i=0,1$. Then the following hold.
\begin{enumerate}
\item\label{item:Injectivity_Donaldson_splicing_map}
If $\lambda_0$ is small enough and $\lambda \in (0,\lambda_0)$ is a fixed scale parameter and $g_\lambda$ is the smooth Riemannian metric defined on the connected sum $X=X_0\# X_1$ by $\lambda$ and Data \ref{data:Splicing_parameters_connected_sum_G-bundle_connected_sum_Riemannian_4-manifold}, then the \emph{Donaldson splicing map}
\[
\cS: \sA(P_0)\times \sA(P_1) \times \Gl_{x_{0\flat},x_{1\flat}} \to \sA(P)
\]  
is not injective but its restriction to a finite-dimensional submanifold is injective:
\[
\cS: M_0 \times M_1 \times \Gl_{x_{0\flat},x_{1\flat}} \to \sA(P).
\]
\item\label{item:Injectivity_Taubes_splicing_map}
If $(X_0,g_0) = (X,g)$ and $(X_1,g_1) = (S^4,g_\round)$ and $x_{1\flat}$ is the south pole in $S^4$ and $M_1^\diamond \subset \sA^\diamond(P_1)$ is a finite-dimensional smooth submanifold, then the \emph{Taubes' splicing map}
\[
\cS: \sA(P_0)\times \sA^\diamond(P_1) \times \Gl_{x_{0\flat}} \times B_\delta(x_{0\flat}) \times (0,\lambda_0) \to \sA(P)
\]  
is not injective but its restriction to a finite-dimensional submanifold is injective:
\[
\cS: M_0 \times M_1^\diamond \times \Gl_{x_{0\flat}} \times B_\delta(x_{0\flat}) \times (0,\lambda_0) \to \sA(P).
\]
\end{enumerate}  
\end{prop}

\begin{proof}
When $\lambda>0$, it is an immediate consequence of their construction that neither the Donaldson nor Taubes splicing maps $\cS$ is injective. By the proof in Guillemin and Pollack \cite[pp. 51--54] {Guillemin_Pollack} of the easy version of the \emph{Whitney Embedding Theorem} (that a smooth submanifold in $\RR^N$ of dimension $m\geq 1$ embeds into $\RR^{2m+1} \subset \RR^N$), there are finite-dimensional vector subspaces
\[
  V_i \subset W^{1,p}(T^*X_i\otimes\ad P_i) \quad\text{with}\quad A_{i\flat} + V_i \subset \sA(P_i), \quad i=0,1,
\]
with inclusions $M_i\subset A_{i\flat} + V_i$ that are smooth embeddings for $i=0,1$.

Consider Item \eqref{item:Injectivity_Donaldson_splicing_map}. If the following restriction of the Donaldson splicing map $\cS$ is injective,
\[
  \cS:(A_{0\flat} + V_0) \times (A_{1\flat} + V_1) \times \Gl_{x_{0\flat},x_{1\flat}} \to \sA(P),
\]
then its further restriction $M_0 \times M_1 \times \Gl_{x_{0\flat},x_{1\flat}}$ will also be injective. When $\lambda_0$ is a small enough constant (depending on the subspaces $V_i$ and the injectivity radii of $(X_i,g_i)$ for $i=0,1$) and $p_i \in P_i|_{x_{i\flat}}$ are fixed fiber points, then the affine maps
\begin{equation}
\label{eq:Linear_cutting_off_map_Mi_affine_space_all_connections}  
 A_{i\flat} + V_i \ni A_i \mapsto \chi_iA_i \in \sA(P_i)
\end{equation}
are injective for $i=0,1$. To see this, one shows that the kernel of the linear map obtained from \eqref{eq:Linear_cutting_off_map_Mi_affine_space_all_connections} by subtracting the reference connection $A_{i\flat}$ is zero for small enough $\lambda_0 \in (0,1]$ and any $\lambda \in (0,\lambda_0]$. An argument of this kind is provided by the authors in \cite[Section 6.4]{FLKM1}. Briefly, if $\{a_{ik}\}_{k=1}^{v_i}$ with $\dim V_i = v_i$ is an $L^2$-orthonormal basis for $V_i$, then $\{\chi_i a_{ik}\}_{k=1}^{v_i}$ will be an approximately $L^2$-orthonormal (hence linearly independent) set of $v_i$ vectors in $\sA(P)$ and thus \eqref{eq:Linear_cutting_off_map_Mi_affine_space_all_connections} must be injective for $i=0,1$.

It follows from the proof of \cite[Proposition 7.2.9]{DK} that the splicing map
\[
\Gl_{x_{0\flat},x_{1\flat}} \ni \rho \mapsto \cS(A_0,A_1,\rho) = \chi_0A_0 + \chi_1A_1^\rho \in \sA(P)
\]
is injective for any fixed pair of connections $(A_0,A_1)$.
Consequently, it follows that the map
\[
\sA(P_0)\times \sA(P_1) \times \Gl_{x_{0\flat},x_{1\flat}} \ni (A_0,A_1,\rho) \mapsto \cS(A_0,A_1,\rho) = \chi_0A_0 + \chi_1A_1^\rho \in \sA(P)
\]
is injective.

Consider Item \eqref{item:Injectivity_Taubes_splicing_map}.  It is an immediate consequence of its construction that the map $\cS$ cannot be injective. However, when the Banach affine spaces $\sA(P_i)$ are replaced by finite-dimensional affine subspaces $M_i$, and $\lambda_0$ is a small enough constant, then $\cS$ is injective.
\end{proof}

\begin{rmk}[Injectivity for polyfold splicing maps]
\label{rmk:Injectivity_polyfold_splicing_maps}
While the splicing map $\cS$ is not injective, when we restrict it to a finite-dimensional subdomain defined by choosing finite-dimensional submanifolds $M_0 \subset \sA(P_0)$ and $M_1^\diamond \subset \sA^\natural(P_1$ (for example, anti-self-dual or extended  anti-self-dual connections in Coulomb gauge with respect to reference anti-self-dual connections), as in Proposition \ref{prop:Injectivity_restrictions_splicing_maps_finite-dimensional_submanifolds}, then $\cS$ becomes an embedding. In the theory of polyfolds \cite{Fabert_Fish_Golovko_Wehrheim_2016}, the definition of the corresponding splicing map is augmented so that it becomes injective.
\end{rmk}

\section{Smooth extension of splicing map for connections}
\label{sec:Smooth_extension_splicing_map_connections}
In order to analyze the continuity and smoothness properties of the splicing map $\cS$ in \eqref{eq:Splicing_map_connections} up to the boundary $\lambda=0$ of the factor $(0,\lambda_0)$ of its domain, we enlarge the codomain $\sA(P)$. This shift in perspective allows us to prove the

\begin{thm}[Smooth extension of the splicing map to the boundary of its domain]
\label{thm:Smoothness_splicing_map_manifold_with_boundary}
Let $(X,g)$ be a closed, connected, four-dimensional, oriented, smooth Riemannian manifold, $G$ be a compact Lie group, $P_0$ be a smooth principal $G$-bundle over $X$, and $P_1$ be a smooth principal $G$-bundle over $S^4$, and $p \in [1,\infty)$ be a constant. Let $P\cong P_0\#_{(\rho,x_0,\lambda)} P_1$ denote the smooth principal $G$-bundle over the connected sum $X\#_{(x_0,\lambda)} S^4 \cong X$ defined by the fixed parameters in Data \ref{data:Fixed_parameters_for_definition_splicing_and_unsplicing_maps} and the parameters $(\rho,x_0,\lambda) \in \Gl_{x_{0\flat}}\times B_{\delta/2}(x_{0\flat})\times (0,\lambda_0)$. The splicing map $\cS$ for connections in \eqref{eq:Splicing_map_connections} with codomain $\sA(P)$ can be equivalently viewed as a smooth section of a product affine bundle:
\begin{equation}
\label{eq:Splicing_map_connections_factored_codomain}  
\begin{tikzcd}
  &\sA(P_0)\times \sA(P_1) \times \Gl_{x_{0\flat}}\times B_\delta(x_{0\flat}) \times (0,\lambda_0) \times \sA(P)
  \arrow[d, "\pi"', shift right=1.0ex]
  \\
  & \arrow[u,"\cS"',  shift right=1.0ex]
  \sA(P_0)\times \sA(P_1) \times \Gl_{x_{0\flat}}\times B_\delta(x_{0\flat}) \times (0,\lambda_0)
\end{tikzcd}
\end{equation}
The preceding smooth section extends to a $C^1$ section of a $C^1$ affine bundle over a smooth Banach manifold with boundary,
\begin{equation}
\label{eq:Splicing_map_connections_factored_codomain_extended}
\begin{tikzcd}
  &\sA^{1,2}(P_0)\times \sA^{1,2}(P_1) \times \Gl_{x_{0\flat}}\times B_\delta(x_{0\flat}) \times \AAA
  \arrow[d, "\pi"', shift right=1.0ex]
  \\
  & \arrow[u,"\cS"',  shift right=1.0ex]
  \sA^{1,2}(P_0)\times \sA^{1,2}(P_1) \times \Gl_{x_{0\flat}}\times B_\delta(x_{0\flat}) \times [0,\lambda_0)
\end{tikzcd}
\end{equation}
where the fibers $\AAA|_{\{\lambda\}}$ are isomorphic to one another as Hilbert affine spaces for each $\lambda\in[0,\lambda_0)$, with
\begin{gather*}
  \AAA|_{\{0\}} := \{0\}\times \left(A_{0\flat}+W_{A_{0\flat}}^{1,2}(T^*X\otimes\ad P_0)\right) \times \left(A_{1\flat}+W_{A_{1\flat}}^{1,2}(T^*S^4\otimes\ad P_1)\right),
  \\
  \AAA|_{\{\lambda\}} := \{\lambda\}\times \left(A_\flat+W_{A_\flat}^{1,2}(T^*X\otimes\ad P)\right), \quad\forall\, \lambda \in (0,\lambda_0).
\end{gather*}
Moreover, we have
\[
  \cS(A_0,A_1,\rho,x_0,0) = \left((A_0,A_1,\rho,x_0,0), (0,A_0,A_1)\right)
\]
upon restriction to the boundary face
\[
  \sA^{1,2}(P_0)\times \sA^{1,2}(P_1) \times \Gl_{x_{0\flat}}\times B_\delta(x_{0\flat}) \times \{0\}.
\]
Finally, the map $\cS$ in \eqref{eq:Splicing_map_connections} extends to a continuous map,
\begin{multline}
\label{eq:Splicing_map_connections_extended}
\cS: \sA(P_0)\times \sA(P_1) \times \Gl_{x_{0\flat}}\times B_\delta(x_{0\flat}) \times [0,\lambda_0)
\\
\to \sA(P)\sqcup\left(\sA(P_0)\times \sA(P_1) \times \Gl_{x_{0\flat}}\times B_\delta(x_{0\flat})\right),
\end{multline}
when the codomain has the Uhlenbeck topology \cite[Section 4.4.1]{DK}.
\end{thm}

In preparation for our proof of Theorem \ref{thm:Smoothness_splicing_map_manifold_with_boundary}, we shall in the forthcoming Theorem \ref{thm:Comparison_Hilbert_spaces_bundle-valued_one-forms} establish a natural isomorphism between two Hilbert spaces.

Recall that the $L^2$ norm on $v \in \Omega^2(X;\ad P)$ and the $L^4$ norm on $a\in \Omega^1(X;\ad P)$ depend only on the \emph{conformal class} $[g]$ of the Riemannian metric $g$ on $X$ and not on the actual metric. With that in mind, we now consider the definition of Sobolev norms for sections of $T^*S^4\otimes\ad P_1$ that are equivalent to the standard norm, $\|\cdot\|_{W_{A_1,g_\round}^{1,2}(S^4)}$, but have more easily described conformal invariance properties.

Let $A_1$ be a smooth connection on a principal $G$-bundle $P_1$ over $S^4$ with its standard round metric $g_\round$ of radius one. Let $\delta = g_{\euclid}$ be the flat metric on $S^4\less\{s\} \cong \RR^4$ obtained by pullback of the standard Euclidean metric on $\RR^4$ via the conformal diffeomorphism $\varphi_n^{-1}: S^4\less\{s\}\to \RR^4$. Let $\nabla_{A_1}^{g_\round}$ denote the covariant derivative on $T^*S^4\otimes\ad P_1$ defined by the connection $A_1$ and metric $g_\round$, while $\nabla_{A_1}^\delta$ denotes the covariant derivative on $T^*S^4\otimes\ad P_1 \restriction S^4\less\{s\}$ defined by $A_1$ and $\delta$. We define the usual $W^{1,2}$ norm on $C^\infty$ sections $a_1$ of $T^*S^4\otimes\ad P_1$ by
\[
\|a_1\|_{W_{A_1,g_\round}^{1,2}(S^4)} := \|\nabla_{A_1}^{g_\round} a_1\|_{L^2(S^4,g_\round)} + \|a_1\|_{L^2(S^4,g_\round)}.
\]
Similarly, if $a_1$ is a $C^\infty$ section of $T^*S^4\otimes\ad P_1$ and has compact support in $S^4\less\{s\}$, define
\begin{align*}
|a_1|_{A_1} &:=  \|\nabla_{A_1}^\delta a_1\|_{L^2(S^4,\delta)},
\\
\|a_1\|_{W_{A_1,\delta}^{1,2}(S^4)}
&:=
\|\nabla_{A_1}^\delta a_1\|_{L^2(S^4,\delta)} + \|a_1\|_{L^2(S^4,\delta)}.
\end{align*}
For an integer $d\geq 2$, let $\Conf(S^d)$ denote the group of conformal transformations of $(S^d,g_\round)$. For any $\lambda\in(0,\infty)$, recall that $\delta_\lambda$ is the dilation of $\RR^d$ given by $x\mapsto x/\lambda$ and for any $z\in\RR^d$, recall that $\tau_z$ is the translation of $\RR^d$ defined by $x\mapsto x-z$. If $\delta_\lambda$ and $\tau_z$ again denote the conformal diffeomorphisms of $S^d$ induced by the chart $\varphi_n^{-1}:S^d\less\{s\}\cong\RR^d$, then the group $\SO(d)\times \RR_+ \times \RR^d$ of rotations, dilations, and translations of $\RR^d$ is identified with the subgroup $\Conf_s(S^d) \subset \Conf(S^d)$ of diffeomorphisms which fix the south pole $s\in S^d$ \cite[p. 346]{TauPath}. Indeed, the finite generators of $\Conf(S^d)$ are dilations, translations, rotations and special conformal transformations, where the latter can be understood as an inversion, followed by a translation, and followed again by an inversion \cite[Section 2.1 and Table 2.1]{Blumenhagen_Plauschinn_2009}, \cite[Theorem 1.9]{Schottenloher_2008}. The conformal and quasi-conformal invariance properties of $|\cdot|_{A_1}$ and $\|\cdot\|_{W_{A_1,\delta}^{1,2}(S^4)}$ are described by two lemmata of Taubes \cite{TauPath, TauFrame}.

\begin{lem}
\label{lem:Taubes_1984_path_proposition_2-4}
(See Taubes \cite[Proposition 2.4]{TauPath}.)
There is a universal constant $z \in [1,\infty)$ with the following significance. If $A$ is a smooth connection on a principal $G$-bundle $P_1$ over $S^4$ with its standard round metric of radius one, then the following hold:
\begin{enumerate}
\item $|\cdot|_{A_1}$ extends to a continuous norm on $\Omega^1(S^4; \ad P_1) = C^\infty(S^4; T^*S^4\otimes\ad P_1)$.

\item The norm $|\cdot|_{A_1}$ is $\Conf_s(S^4)$-invariant:
    \[
    |h^*a_1|_{h^*A_1}=|a|_{A_1}, \quad\forall\, h\in \Conf_s(S^4) \text{ and } a_1 \in \Omega^1(S^4; \ad P_1).
    \]

\item If $a_1 \in \Omega^1(S^4;\ad P_1)$, then
\begin{align*}
z^{-1}\|a_1\|_{W_{A_1,g_\round}^{1,2}(S^4)}
&\leq
|a_1|_{A_1}
\leq
z\|a_1\|_{W_{A_1,g_\round}^{1,2}(S^4)},
\\
z^{-1}\|a_1\|_{W_{A_1,g_\round}^{1,2}(S^4)}
&\leq
\|a_1\|_{W_{A_1,\delta}^{1,2}(S^4)}
\leq
z\|a_1\|_{W_{A_1,g_\round}^{1,2}(S^4)}.
\end{align*}
\end{enumerate}
\end{lem}

\begin{lem}
\label{lem:Taubes_1988_3-1}
(See Taubes \cite[Lemma 3.1]{TauFrame}.)
Assume the hypotheses and notation of Lemma \ref{lem:Taubes_1984_path_proposition_2-4}. If $h \in \Conf(S^4)$ and $a_1\in\Omega^1(S^4;\ad P_1)$, then
\[
z^{-1}\|a_1\|_{W_{A_1,g_\round}^{1,2}(S^4)}
\leq
\|h^*a_1\|_{W_{h^*A_1,g_\round}^{1,2}(S^4)}
\leq
z\|a_1\|_{W_{A_1,g_\round}^{1,2}(S^4)}.
\]
\end{lem}

\begin{rmk}[Variant of the $W^{1,2}$ norm on sections of $T^*S^4\otimes\ad P_1$]
\label{rmk:Taubes_1988_3-1}
A combination of the \emph{Kato Inequality} \cite[Equation (6.20)]{FU} and the Sobolev embedding $W^{1,2}(S^4;\RR) \subset L^4(S^4;\RR)$ given by \cite[Theorem 4.12]{AdamsFournier} yields a universal constant $z_0 \in [1,\infty)$ such that
\[
\|a_1\|_{L^4(S^4,g_\round)} \leq z_0\|a_1\|_{W_{A_1,g_\round}^{1,2}(S^4)},
\]
and thus, since $\|a_1\|_{L^2(S^4,g_\round)} \leq \vol(S^4)^{1/4}\|a\|_{L^4(S^4,g_\round)}$, the norm $\|a\|_{W_{A_1,g_\round}^{1,2}(S^4)}$ may be replaced by the equivalent norm,
\[
\|a_1\|_{\widehat{W}_{A_1,g_\round}^{1,2}(S^4)} := \|\nabla_{A_1}^{g_\round} a_1\|_{L^2(S^4,g_\round)} + \|a_1\|_{L^4(S^4,g_\round)},
\]
in the statements of Lemmata \ref{lem:Taubes_1984_path_proposition_2-4} and \ref{lem:Taubes_1988_3-1}.
\end{rmk}

We say that two Hilbert spaces $\sG$ and $\sH$ are \emph{isomorphic} if there exists a bounded linear map $T:\sG\to\sH$ with bounded linear inverse $T^{-1}:\sH\to\sG$ and in addition that $\sG$ and $\sH$ are \emph{isometrically isomorphic} if $T$ is also an isometry (and thus preserves inner products by the parallelogram identity). Recall that a sequence $\{a_k\}_{k\in\NN} \subset \sH$ is an \emph{orthonormal basis for $\sH$} \cite[Section 2.3]{Reed_Simon_v1} if the sequence is $\sH$-orthonormal and the subspace of finite linear combinations of basis elements is dense in $\sH$. 

\begin{thm}[Isomorphisms of Hilbert spaces of bundle-valued one-forms for small scale parameters]
\label{thm:Comparison_Hilbert_spaces_bundle-valued_one-forms}
Assume the hypotheses of Theorem \ref{thm:Smoothness_splicing_map_manifold_with_boundary} and, in particular, that $g$ is conformally flat on $B_{x_{0\flat}}(2\delta)$, so hypothesis \eqref{eq:Riemg_zero} from Theorem \ref{mainthm:Gluing} holds. Let $g_{x_0,\lambda}$ be the smooth Riemannian metric on the connected sum $X \cong X\#_{x_0,\lambda}S^4$ defined in Section \ref{subsec:Riemannian_metric_connected_sum_X_4-sphere} by the metric $g$ on $X$ and the metric $g_\round$ on $S^4$. Then the Hilbert spaces $W_{A_\flat}^{1,2}(T^*X\otimes\ad P)$ defined by $g_{x_0,\lambda}$ are isomorphic to one another, with uniformly equivalent norms for all $\lambda \in (0,\lambda_0]$, where $A_\flat := \cS(A_{0\flat},A_{1\flat},\rho_\flat,x_{0\flat},\lambda)$ on $P$.
Moreover, there is an isometric isomorphism of Hilbert spaces (depending on $\lambda \in (0,\lambda_0]$ and the preceding choices) 
\begin{equation}
\label{eq:Isomorphism_direct_sum_one-forms_P0_and_one-forms_P1_with_one-forms_P}
\Phi:  W_{A_{0\flat}}^{1,2}(T^*X\otimes\ad P_0)\oplus W_{A_{1\flat}}^{1,2}(T^*S^4\otimes\ad P_1) \cong  W_{A_\flat}^{1,2}(T^*X\otimes\ad P)
\end{equation}
defined by choices of $W^{1,2}$-orthonormal bases for the three $W^{1,2}$ spaces appearing in \eqref{eq:Isomorphism_direct_sum_one-forms_P0_and_one-forms_P1_with_one-forms_P}. Let 
\[
  \{a_{0j}\}_{j\in\NN} \subset W_{A_{0\flat}}^{1,2}(T^*X\otimes\ad P_0)
   \quad\text{and}\quad
  \{a_{1k}\}_{k\in\NN} \subset W_{A_{1\flat}}^{1,2}(T^*S^4\otimes\ad P_1)
\]
be $W^{1,2}$-orthonormal bases defined by sequences of $L^2$-orthonormal eigenvectors of the following second-order elliptic partial differential operators for $i=0,1$:
\[
  \Delta_{A_{i\flat}}^1 := d_{A_{i\flat}}^*d_{A_{i\flat}} + d_{A_{i\flat}}d_{A_{i\flat}}^* \quad\text{on }\Omega^1(\ad P_i).
\]  
There are a non-increasing \emph{sequential scale function}
\begin{equation}
\label{eq:Sequential_scale_function}
\bs:\NN\times\NN\times(0,\lambda_0] \ni (m,n,\lambda) \mapsto  \lambda_{m,n} := \bs(m,n,\lambda) \in (0,\lambda_0]  
\end{equation}
such that $\bs(1,1,\lambda) = \lambda$ and $\bs(m,n,\lambda) \to 0$ as $m$ or $n\to\infty$ and, for each $\lambda\in (0,\lambda_0]$, a sequence of approximate eigenvectors for $\Delta_{A_{i\flat}}^1$ on $\Omega^1(\ad P)$, 
\[
  \{a_{j,k}(\lambda_{jk})\}_{j,k\in\NN} \subset W_{A_\flat}^{1,2}(T^*X\otimes\ad P),
\]
given by the splicing map for bundle-valued one-forms in Section \ref{sec:Splicing_map_connections}:
\[
  a_{j,k}'(\lambda_{jk}) := \cS(a_{0j},a_{1k},\rho_\flat,x_{0\flat},\lambda_{jk}), \quad\forall\, (j,k) \in \NN\times\NN.
\]
Application of the \emph{Gram--Schmidt orthonormalization process} to $\{a_{j,k}'(\lambda_{jk})\}_{j,k\in\NN}$ yields a $W^{1,2}$-orthonormal basis $\{\bar a_{j,k}(\lambda_{jk})\}_{j,k\in\NN}$ for $W_{A_\flat}^{1,2}(T^*X\otimes\ad P)$ and if
\[
  \{a_{j,k}(\lambda_{jk})\}_{j,k\in\NN} \subset W_{A_\flat}^{1,2}(T^*X\otimes\ad P)
\]
is a suitably enumerated $W^{1,2}$-orthonormal basis for $W_{A_\flat}^{1,2}(T^*X\otimes\ad P)$ defined by a sequence of $L^2$-orthonormal eigenvectors of $\Delta_{A_\flat}^1$ on $\Omega^1(\ad P)$, then the assignment
\[
  \bar a_{j,k}(\lambda_{jk}) \mapsto a_{j,k}(\lambda_{jk}), \quad\forall\, (j,k) \in \NN\times\NN
\]
extends to an isometric isomorphism of the Hilbert space $W_{A_\flat}^{1,2}(T^*X\otimes\ad P)$ onto itself. Finally, the assignment
\[
  (a_{0j}, a_{1k}) \mapsto a_{j,k}(\lambda_{jk}), \quad\forall\, (j,k) \in \NN\times\NN
\]
extends to an isometric isomorphism $\Phi(\lambda)$ in \eqref{eq:Isomorphism_direct_sum_one-forms_P0_and_one-forms_P1_with_one-forms_P} of Hilbert spaces (depending on $\lambda \in (0,\lambda_0]$ and the preceding choices). If $\{\lambda_n\}_{n\in\NN} \subset (0,\lambda_0]$ is any sequence that converges to zero as $n\to\infty$, then the corresponding sequence $\{\Phi(\lambda_n)\}_{n\in\NN}$ of isomorphisms \eqref{eq:Isomorphism_direct_sum_one-forms_P0_and_one-forms_P1_with_one-forms_P} is Cauchy with respect to the operator norm.
\end{thm}

\begin{proof}
The construction of the preceding isomorphisms and verification of their properties are very similar to those in the proof of the forthcoming Theorem \ref{thm:Comparison_Hilbert_spaces_bundle-valued_self-dual_two-forms}. Note that if $\{a_{0j}\}_{j\in\NN}$ is an $L^2$-orthonormal basis for $W_{A_{0\flat}}^{1,2}(T^*X\otimes\ad P_0)$ defined by the eigenvectors of $\Delta_{A_{0\flat}}^1$ and we define the norm on $W_{A_{0\flat}}^{1,2}(T^*X\otimes\ad P_0)$ by
\[
  \|a\|_{W_{A_{0\flat}}^{1,2}(X)} := \left\|\left(\Delta_{A_{0\flat}}^1 + 1\right)^{1/2}a\right\|_{L^2(X)},
\]
then the basis $\{a_{0j}\}_{j\in\NN}$ is also orthonormal with respect to the corresponding $W^{1,2}$ inner product,
\[
  (a,b)_{W_{A_{0\flat}}^{1,2}(X)} = \left(a,\left(\Delta_{A_{0\flat}}^1 + 1\right)b\right)_{L^2(X)}
\]
for all $a,b \in W_{A_{0\flat}}^{1,2}(T^*X\otimes\ad P_0)$.  
\end{proof}  

In the inequalities that appear in the proofs of Theorem \ref{thm:Smoothness_splicing_map_manifold_with_boundary} and subsequent results, we write $a \lesssim b$ if $a\leq Cb$ for some constant $C\in[1,\infty)$ that is independent of the parameters. We write $a \sim b$ if both $a \lesssim b$ and $b\lesssim a$.

\begin{proof}[Proof of Theorem \ref{thm:Smoothness_splicing_map_manifold_with_boundary}]
We need to analyze the boundedness and uniform continuity properties of the partial derivatives of $\cS$ on its domain in \eqref{eq:Splicing_map_connections_factored_codomain} when $p=2$: 
\begin{equation}
\label{eq:Splicing_map_connections_partial_derivatives}
\frac{\partial\cS}{\partial A_0}, \quad \frac{\partial\cS}{\partial A_1}, \quad \frac{\partial\cS}{\partial\rho}, \quad \frac{\partial\cS}{\partial x_0}, \quad\text{and}\quad \frac{\partial\cS}{\partial\lambda}.
\end{equation}
Calculations of this kind were previously done by Feehan \cite[Theorem 3.33]{FeehanGeometry}, Groisser \cite{Groisser_1993, Groisser_1998}, Groisser and Parker \cite[Section 3]{GroisserParkerGeometryDefinite}, and Peng \cite{Peng_1995, Peng_1996}. The analysis by Feehan in \cite[Section 3]{FeehanGeometry} is closest to our needs here. Peng assumes initially that the Riemannian metric $g$ is flat near $x_{0\flat}$ (in other words, that hypothesis \eqref{eq:Riemg_zero} holds in Theorem \ref{mainthm:Gluing}) and then makes adjustments in \cite[Section 5]{Peng_1995} to allow for non-flat Riemannian metrics. Moreover, in \cite{Peng_1995, Peng_1996}, Peng only estimates the $L^2$ norms of the partial derivatives in \eqref{eq:Splicing_map_connections_partial_derivatives}, whereas in \cite[Section 3]{FeehanGeometry}, Feehan estimates their $L^p$ norms for any $p\in [1,\infty)$. The latter calculations extend to give estimates of their $L^{p^*}$ and $W^{1,p}$ norms when $p \in [1,4)$ and $p^*=4p/(4-p)\in[4/3,\infty)$. We only need the case $p=2$ and $p^*=4$.

\begin{step}[Derivatives with respect to the connections $A_0$ and $A_1$]
See the proof by Feehan of his \cite[Proposition 3.30]{FeehanGeometry} for details and see the proof by Peng of his \cite[Lemma 4.7]{Peng_1995} for related calculations. Suppose that $A_0(t_0)$ is a smooth curve in $\sA(P_0)$ parameterized by arclength $t_0\in(-\eps,\eps)$ for some $\eps\in(0,1]$. By \cite[Proposition 3.30]{FeehanGeometry}, we have
\[
\left\|\frac{\partial\cS}{\partial t_0}-\frac{\partial A_0}{\partial t_0}\right\|_{L^p(X_0'')} \lesssim \lambda^{2/p}
\quad\text{and}\quad
\left\|\frac{\partial\cS}{\partial t_0}\right\|_{L^p(X)} \lesssim 1.
\]
Similar calculations yield the estimates
\[
\left\|\frac{\partial\cS}{\partial t_0}-\frac{\partial A_0}{\partial t_0}\right\|_{W^{1,p}(X_0'')} \lesssim \lambda^{2/p^*}
\quad\text{and}\quad
\left\|\frac{\partial\cS}{\partial t_0}\right\|_{W^{1,p}(X)} \lesssim 1.
\]
Aside from obvious notational changes, the calculations and conclusions for $\partial\cS/\partial A_1$ are identical. This completes our estimates of the partial derivatives of $\cS(A_0,A_1,\rho,x_0,\lambda)$ with respect to the connections $A_0$ and $A_1$.
\end{step}

\begin{step}[Derivative with respect to the gluing parameter $\rho$]
See the proof by Feehan of his \cite[Proposition 3.28]{FeehanGeometry} for details and see the proof by Peng of his \cite[Lemma 4.2]{Peng_1995} for related calculations. Suppose that $\rho(t)$ is a smooth curve in $\Gl_{x_{0\flat}}\cong G$ parameterized by arclength $t\in(-\eps,\eps)$ for some $\eps\in(0,1]$. By \cite[Proposition 3.28]{FeehanGeometry}, we have
\[
\left\|\frac{\partial\cS}{\partial t}\right\|_{L^{p^*}(X)} \sim |t|\lambda^{2/p^* - 1/2}.
\]
Similar calculations yield the estimates
\[
\left\|\frac{\partial\cS}{\partial t}\right\|_{W^{1,p}(X)} \sim |t|\lambda^{2/p - 1}.
\]
This completes our estimate of the partial derivative of $\cS(A_0,A_1,\rho,x_0,\lambda)$ with respect to the gluing parameter $\rho$.
\end{step}

\begin{step}[Derivative with respect to the center point $x_0$]
See the proof by Feehan of his \cite[Proposition 3.25]{FeehanGeometry} for details and see the proof by Peng of his \cite[Lemma 4.10]{Peng_1995} for related calculations. Suppose that $x_0(t)$ is a smooth curve in $X$ parameterized by arclength $t\in(-\eps,\eps)$ for some $\eps\in(0,1]$. By \cite[Proposition 3.25 (b)]{FeehanGeometry}, we have
\[
\left\|\frac{\partial\cS}{\partial t}\right\|_{L^p)} \lesssim \lambda^{2/p}.
\]
Similar calculations yield the estimate
\[
\left\|\frac{\partial\cS}{\partial t}\right\|_{W^{1,p}(X)} \lesssim \lambda^{2/p^*}.
\]
This completes our estimate of the partial derivative of $\cS(A_0,A_1,\rho,x_0,\lambda)$ with respect to the center point $x_0$.  
\end{step}

\begin{step}[Derivative with respect to the scale $\lambda$]
See the proof by Feehan of his \cite[Proposition 3.25]{FeehanGeometry} for details and see the proof by Peng of his \cite[Lemma 4.5]{Peng_1995} for related calculations. By \cite[Proposition 3.25 (a)]{FeehanGeometry}, we have
\[
\left\|\frac{\partial\cS}{\partial\lambda}\right\|_{L^p)} \lesssim \lambda^{2/p-1/2}.
\]
Similar calculations yield the estimate
\[
\left\|\frac{\partial\cS}{\partial t}\right\|_{W^{1,p}(X)} \lesssim \lambda^{2/p^*-1/2}.
\]
This completes our estimate of the partial derivative of $\cS(A_0,A_1,\rho,x_0,\lambda)$ with respect to the scale parameter $\lambda$.    
\end{step}

By the preceding analysis, we conclude that the partial derivatives in \eqref{eq:Splicing_map_connections_partial_derivatives} are \emph{bounded} on the domain of $\cS$ in \eqref{eq:Splicing_map_connections_factored_codomain} when $p=2$. Recall now that if $S$ is a subset of metric space $M$ and $N$ is complete metric space and $f:S\to N$ is a uniformly continuous map, then $f$ has a unique extension to a uniformly continuous map $f:\bar S\to N$, where $\bar S$ is the closure of $S$ in $M$ \cite[Theorem 13.D]{Simmons_introduction_topology_modern_analysis}.

By employing further calculations provided by Feehan in \cite[Sections 3.6--3.9]{FeehanGeometry} (see also Peng \cite{Peng_1995, Peng_1996} for similar calculations), one finds moreover that the partial derivatives in \eqref{eq:Splicing_map_connections_partial_derivatives} are \emph{uniformly continuous} on the domain of $\cS$ in \eqref{eq:Splicing_map_connections_factored_codomain} when $p=2$ and, in particular, uniformly continuous with respect to $\lambda\in(0,\lambda_0)$. Hence, these partial derivatives have unique uniformly continuous extensions to the domain of $\cS$ in \eqref{eq:Splicing_map_connections_factored_codomain_extended}. Consequently, $\cS$ is a $C^1$ map of smooth Banach manifolds with boundary, as claimed.
\end{proof}

\section{Composition of self-dual curvature and splicing maps}
\label{sec:Composition_self-dual_curvature_splicing_maps}
We now consider the composition of the self-dual curvature map,
\begin{equation}
  \label{eq:Self-dual_curvature_map}
  F^+:\sA(P) \ni A \mapsto F_A^+ \in L^p(\wedge^+(T^*X)\otimes\ad P),
\end{equation}
and the splicing map $\cS$ in \eqref{eq:Splicing_map_connections}. Observe that we cannot extend $F^+\circ\,\cS$ to the boundary $\{\lambda=0\}$ of the domain of $\cS$ without also replacing the codomain $L^p(\wedge^+(T^*X)\otimes\ad P)$ by one that extends from $\lambda\in(0,\lambda_0)$ to $\lambda\in[0,\lambda_0)$. We now consider a suitable choice of codomain.

We begin by explaining how to define $\ad P$-valued self-dual two-forms as splicings of $\ad P_0$ and $\ad P_1$-valued self-dual two-forms. Given the auxiliary fixed splicing parameters in Data \ref{data:Fixed_parameters_for_definition_splicing_and_unsplicing_maps}, let
\begin{gather*}
  (\rho,x_0,\lambda) \in \Gl_{x_{0\flat}} \times B_\delta(x_{0\flat}) \times (0,\lambda_0),
  \\
  \omega_0 \in L^p(\wedge^+(T^*X)\otimes\ad P_0), \quad\text{and}\quad \omega_1 \in L^p(\wedge^+(T^*S^4)\otimes\ad P_1).
\end{gather*}
Using the splicing data, we can construct the smooth principal $G$-bundle $P = P_0\# P_1$ over the connected sum $X \cong X\# S^4$ (conformal diffeomorphism) and define 
\begin{equation}
\label{eq:Spliced_bundle-valued_self-dual_two-forms}
  \omega := \psi_0\omega_0 + \psi_1\omega_1 \in L^p(\wedge^+(T^*X)\otimes\ad P),
\end{equation}
where we recall from \eqref{eq:Partition_unity_and_covering_cutoff_functions} that $\psi_0, \psi_1 \in C^\infty(X;[0,1])$ with $\supp\psi_0 \subset X\less\{x_0\}$ and $\supp\psi_1 \subset B_{\varrho_0}(x_{0\flat})$ and
\[
  \psi_0 \equiv 1 \quad \text{on } \supp \chi_0 \quad\text{and}\quad \psi_1 \equiv 1 \quad \text{on } \supp\chi_1.
\]
The partition of unity $\chi_0,\chi_1\in C^\infty(X;[0,1])$ with $\chi_0+\chi_1=1$ on $X$ was also defined in \eqref{eq:Partition_unity_and_covering_cutoff_functions}. The sum \eqref{eq:Spliced_bundle-valued_self-dual_two-forms} is defined by using the trivializations of $P_0$ and $P_1$ implied by the fixed choices in Data \ref{data:Fixed_parameters_for_definition_splicing_and_unsplicing_maps}. Hence, we obtain a \emph{splicing map for bundle-valued self-dual two-forms} by analogy with our definition \eqref{eq:Splicing_map_connections} of the splicing map $\cS$ for connections:
\begin{multline}
\label{eq:Splicing_map_bundle-valued_self-dual_two-forms}
\sS: L^p(\wedge^+(T^*X)\otimes\ad P_0) \times L^p(\wedge^+(T^*S^4)\otimes\ad P_1) \times \Gl_{x_{0\flat}} \times B_\delta(x_{0\flat}) \times (0,\lambda_0)
\\
\ni (\omega_0,\omega_1,\rho,x,\lambda) \mapsto \omega \in L^p(\wedge^+(T^*X)\otimes\ad P).
\end{multline}
By analogy with our definition \eqref{eq:Unsplicing_map_connections} of the unsplicing map $\cU$ for connections, we may also define an \emph{unsplicing map for bundle-valued self-dual two-forms}:
\begin{multline}
\label{eq:Unsplicing_map_bundle-valued_self-dual_two-forms}
\sU: L^p(\wedge^+(T^*X)\otimes\ad P) \times \Gl_{x_{0\flat}} \times B_\delta(x_{0\flat}) \times (0,\lambda_0) \ni (\omega,\rho,x_0,\lambda) \mapsto (\omega_0,\omega_1,\rho,x_0,\lambda) 
\\
\in L^p(\wedge^+(T^*X)\otimes\ad P_0) \times L^p(\wedge^+(T^*S^4)\otimes\ad P_1) \times \Gl_{x_{0\flat}} \times B_\delta(x_{0\flat}) \times (0,\lambda_0),
\end{multline}
via the assignments
\[
  \omega_0 = \psi_0\omega \quad\text{and}\quad \omega_1 = \psi_1\omega,
\]
and choice of parameters $(\rho,x_0,\lambda)$, just as in the definition \eqref{eq:Unsplicing_map_connections} of $\cU$. (Recall that the cutoff functions $\chi_0,\chi_1,\psi_0,\psi_1$ depend on the choices of centers $x_0\in B_\delta(x_{0\flat})$ and scales $\lambda\in (0,\lambda_0)$.)

Given $A = \cS(A_0,A_1,\rho,x_0,\lambda)$, the composition
\[
  F^+(A) = (F^+\circ\,\cS)(A_0,A_1,\rho,x_0,\lambda) \in L^p(\wedge^+(T^*X)\otimes\ad P)
\]
naturally factors to give
\begin{equation}
\label{eq:Self-dual_curvature_map_factored_codomain}
\widehat{F}^+(\chi_0A_0,\chi_1A_1) := (\omega_0,\omega_1) \in L^p(\wedge^+(T^*X)\otimes\ad P_0) \times L^p(\wedge^+(T^*S^4)\otimes\ad P_1)  
\end{equation}
by noting that
\[
  A = \chi_0A_0 + \chi_1A_1 \quad\text{over } X.
\]
Indeed, we have $A = A_0$ over $\{\chi_0 = 1\}$ and $A = A_1$ over $\{\chi_1=1\}$, while over $\{0<\chi_0< 1\}=\{0<\chi_1<1\}$, where $P=\Omega_{x_{0\flat},\lambda}\times G$, we have
\begin{multline*}
  F^+(A) = (d\chi_0\wedge A_0)^+ + \chi_0 d^+A_0 + (d\chi_1\wedge A_1)^+ + \chi_1d^+A_1
  \\
  + \frac{1}{2}\chi_0^2[A_0,A_0] + \frac{1}{2}\chi_0\chi_1[A_0,A_1] + \frac{1}{2}\chi_1^2[A_1,A_1].
\end{multline*}
We can construct $\widehat{F}^+(\chi_0A_0,\chi_1A_1) = (\omega_0,\omega_1)$ in \eqref{eq:Self-dual_curvature_map_factored_codomain} by defining $\omega_0\in L^p(\wedge^+(T^*X)\otimes\ad P_0)$ as
\[
  \omega_0
  :=
  \begin{cases}
    F^+(A_0) &\text{over } \{\chi = 1\},
    \\
    (d\chi\wedge A_0)^+ + \chi d^+A_0 + \frac{1}{2}\chi^2[A_0,A_0] + \frac{1}{4}\chi(1-\chi)[A_0,A_1]
    &\text{over } \{0<\chi<1\},
    \\
    0 &\text{over } \{\chi=0\},
  \end{cases}  
\]
and by defining $\omega_1 \in L^p(\wedge^+(T^*S^4)\otimes\ad P_1)$ as
\[
  \omega_1
  :=
  \begin{cases}
    F^+(A_1) &\text{over } \{\chi = 0\},
    \\
    -(d\chi\wedge A_1)^+ + (1-\chi) d^+A_1 + \frac{1}{4}\chi(1-\chi)[A_0,A_1] + \frac{1}{2}(1-\chi)^2[A_1,A_1]
    &\text{over } \{0<\chi<1\},
    \\
    0 &\text{over } \{\chi=1\},
  \end{cases}  
\]
and observing that when $\lambda\in(0,\lambda_0)$ we have
\[
  F_A^+ = \omega_0 + \omega_1 \in L^p(\wedge^+(T^*X)\otimes\ad P).
\]
Lastly, we prove that the splicing map $\sS$ in \eqref{eq:Splicing_map_bundle-valued_self-dual_two-forms} is \emph{surjective} by observing that the unsplicing map for bundle-valued self-dual two-forms $\sU$ in \eqref{eq:Unsplicing_map_bundle-valued_self-dual_two-forms} is an explicit (smooth) \emph{right inverse}, just as we did in Section \ref{subsec:Surjectivity_splicing_map_connections} for the splicing and unsplicing maps for connections. Indeed, for any $\omega \in L^p(\wedge^+(T^*X)\otimes\ad P)$, we have
\begin{align*}
  \sS\circ\sU(\omega,\rho,x_0,\lambda) &= \sS(\psi_0\omega,\psi_1\omega,\rho,x_0,\lambda)
  \\
                 &= \chi_0\psi_0\omega + \chi_1\psi_1\omega
  \\
  &= \chi_0\omega + \chi_1\omega = \omega,
\end{align*}
and so $\sU$ is a right inverse for $\sS$,
\[
  \sS\circ\sU = \id \quad\text{on } L^p(\wedge^+(T^*X)\otimes\ad P),
\]
and $\sS$ is surjective, as we had claimed.

\section{Smooth extension of composition of self-dual curvature and splicing maps}
\label{sec:Smooth_extension_composition_self-dual_curvature_and_splicing_maps}
Our main goal in this section is to prove 

\begin{thm}[Smooth extension of the composition of the self-dual curvature and splicing maps to the boundary]
\label{thm:Smooth_extension_composition_self-dual_curvature_splicing_maps}
Let $(X,g)$ be a closed, connected, four-dimensional, oriented, smooth Riemannian manifold, $G$ be a compact Lie group, $P_0$ be a smooth principal $G$-bundle over $X$, and $P_1$ be a smooth principal $G$-bundle over $S^4$, and $p \in [2,\infty)$ be a constant. Let $P\cong P_0\#_{(\rho,x_0,\lambda)} P_1$ denote the smooth principal $G$-bundle over the connected sum $X\#_{(x_0,\lambda)} S^4 \cong X$ defined by the fixed parameters in Data \ref{data:Fixed_parameters_for_definition_splicing_and_unsplicing_maps} and the parameters $(\rho,x_0,\lambda) \in \Gl_{x_{0\flat}}\times B_\delta(x_{0\flat})\times (0,\lambda_0)$. The smooth composition $F^+\circ\,\cS$ of the splicing map for connections in \eqref{eq:Splicing_map_connections} and the self-dual curvature map in \eqref{eq:Self-dual_curvature_map},
\begin{multline}
\label{eq:Composition_splicing_map_connections_self-dual_curvature_map}
F^+\circ\,\cS:\sA(P_0)\times \sA(P_1) \times \Gl_{x_{0\flat}}\times B_\delta(x_{0\flat}) \times (0,\lambda_0)
\ni (A_0,A_1,\rho,x_0,\lambda)
\\
\mapsto F_A^+ \in L^p(\wedge^+(T^*X)\otimes\ad P),
\end{multline}
where $A := \cS(A_0,A_1,\rho,x_0,\lambda) \in \sA(P)$, can be equivalently viewed as a smooth section of a product vector bundle:
\begin{equation}
\label{eq:Composition_splicing_map_connections_self-dual_curvature_map_factored_codomain}  
\begin{tikzcd}
  &\sA(P_0)\times \sA(P_1) \times \Gl_{x_{0\flat}}\times B_\delta(x_{0\flat}) \times (0,\lambda_0) \times L^p(\wedge^+(T^*X)\otimes\ad P)
  \arrow[d, "\pi"', shift right=1.0ex]
  \\
  & \arrow[u,"F^+\circ\,\cS"',  shift right=1.0ex]
  \sA(P_0)\times \sA(P_1) \times \Gl_{x_{0\flat}}\times B_\delta(x_{0\flat}) \times (0,\lambda_0)
\end{tikzcd}
\end{equation}
The preceding smooth section extends to a $C^1$ section of a $C^1$ vector bundle over a smooth Banach manifold with boundary,
\begin{equation}
\label{eq:Composition_splicing_map_connections_self-dual_curvature_map_unsplicing_map_self-dual_two-forms}  
\begin{tikzcd}
  &\sA^{1,2}(P_0)\times \sA^{1,2}(P_1) \times \Gl_{x_{0\flat}}\times B_\delta(x_{0\flat}) \times \VV 
  \arrow[d, "\pi"', shift right=1.0ex]
  \\
  & \arrow[u,"F^+\circ\,\cS"',  shift right=1.0ex]
  \sA^{1,2}(P_0)\times \sA^{1,2}(P_1) \times \Gl_{x_{0\flat}}\times B_\delta(x_{0\flat}) \times [0,\lambda_0)
\end{tikzcd}
\end{equation}
where the fibers $\VV|_{\{\lambda\}}$ are isometrically isomorphic to one another as Hilbert spaces for each $\lambda\in[0,\lambda_0)$, with
\begin{gather*}
  \VV|_{\{0\}} := \{0\}\times L^2(\wedge^+(T^*X)\otimes\ad P_0)\times L^2(\wedge^+(T^*S^4)\otimes\ad P_1),
  \\
  \VV|_{\{\lambda\}} := \{\lambda\}\times L^2(\wedge^+(T^*X)\otimes\ad P), \quad\forall\, \lambda \in (0,\lambda_0).
\end{gather*}
Moreover, we have
\[
  F^+\circ\,\cS(A_0,A_1,\rho,x_0,0) = \left((A_0,A_1,\rho,x_0,0), (0,F_{A_0}^+,F_{A_1}^+)\right)
\]
upon restriction to the boundary face
\[
  \sA^{1,2}(P_0)\times \sA^{1,2}(P_1) \times \Gl_{x_{0\flat}}\times B_\delta(x_{0\flat}) \times \{0\}.
\]
\end{thm}

In preparation for our proof of Theorem \ref{thm:Smooth_extension_composition_self-dual_curvature_splicing_maps}, we shall establish the following natural isomorphism between two Hilbert spaces.

\begin{thm}[Isomorphisms of Hilbert spaces of bundle-valued self-dual two-forms for small scale parameters]
\label{thm:Comparison_Hilbert_spaces_bundle-valued_self-dual_two-forms}
Assume the hypotheses of Theorem \ref{thm:Smooth_extension_composition_self-dual_curvature_splicing_maps} and, in particular, that $g$ is conformally flat on $B_{x_{0\flat}}(2\delta)$, so hypothesis \eqref{eq:Riemg_zero} from Theorem \ref{mainthm:Gluing} holds. Let $g_{x_0,\lambda}$ be the smooth Riemannian metric on the connected sum $X \cong X\#_{x_0,\lambda}S^4$ defined in Section \ref{subsec:Riemannian_metric_connected_sum_X_4-sphere} by the metric $g$ on $X$ and the metric $g_\round$ on $S^4$. Then the Hilbert spaces $L^2(\wedge^+(T^*X)\otimes\ad P)$ defined by $g_{x_0,\lambda}$ are isometrically isomorphic to one another for all $\lambda \in (0,\lambda_0]$, that is, they are independent of $\lambda\in (0,\lambda_0]$.
Moreover, there is an isometric isomorphism of Hilbert spaces
\begin{equation}
\label{eq:Isomorphism_direct_sum_SD2forms_P0_and_SD2forms_P1_with_SD2forms_P}  
\Psi:  L^2(\wedge^+(T^*X)\otimes\ad P_0)\oplus L^2(\wedge^+(T^*S^4)\otimes\ad P_1) \cong  L^2(\wedge^+(T^*X)\otimes\ad P)
\end{equation}
defined by choices of orthonormal bases for the three $L^2$ spaces appearing in \eqref{eq:Isomorphism_direct_sum_SD2forms_P0_and_SD2forms_P1_with_SD2forms_P}. Let 
\[
  \{v_{0j}\}_{j\in\NN} \subset L^2(\wedge^+(T^*X)\otimes\ad P_0)
   \quad\text{and}\quad
  \{v_{1k}\}_{k\in\NN} \subset L^2(\wedge^+(T^*S^4)\otimes\ad P_1)
\]
be orthonormal bases defined by sequences of $L^2$-orthonormal eigenvectors of the second-order elliptic partial differential operators $d_{A_{i\flat}}^+d_{A_{i\flat}}^{+,*}$ on $\Omega^+(\ad P_i)$ for $i=0,1$. There are a non-increasing \emph{sequential scale function}
\begin{equation}
\label{eq:Sequential_scale_function}
\bs:\NN\times\NN\times(0,\lambda_0] \ni (m,n,\lambda) \mapsto  \lambda_{m,n} := \bs(m,n,\lambda) \in (0,\lambda_0]  
\end{equation}
such that $\bs(1,1,\lambda) = \lambda$ and $\bs(m,n,\lambda) \to 0$ as $m$ or $n\to\infty$ and, for each $\lambda\in (0,\lambda_0]$, a smooth connection $A_\flat := \cS(A_{0\flat},A_{1\flat},\rho_\flat,x_{0\flat},\lambda)$ on $P$ and a sequence of approximate eigenvectors for $d_{A_\flat}^+d_{A_\flat}^{+,*}$ on $\Omega^+(\ad P)$,
\[
  \{v_{j,k}'(\lambda_{jk})\}_{j,k\in\NN} \subset L^2(\wedge^+(T^*X)\otimes\ad P),
\]
given by the splicing map for bundle-valued, self-dual two-forms in Section \ref{sec:Composition_self-dual_curvature_splicing_maps}:
\[
  v_{j,k}'(\lambda_{jk}) := \sS(v_{0j},v_{1k},\rho_\flat,x_{0\flat},\lambda_{jk}), \quad\forall\, (j,k) \in \NN\times\NN.
\]
Application of the \emph{Gram--Schmidt orthonormalization process} to $\{v_{j,k}'(\lambda_{jk})\}_{j,k\in\NN}$ yields an orthonormal basis $\{\bar v_{j,k}(\lambda_{jk})\}_{j,k\in\NN}$ for $L^2(\wedge^+(T^*X)\otimes\ad P)$ and if
\[
  \{v_{j,k}(\lambda_{jk})\}_{j,k\in\NN} \subset L^2(\wedge^+(T^*X)\otimes\ad P)
\]
is a suitably enumerated orthonormal basis for $L^2(\wedge^+(T^*X)\otimes\ad P)$ defined by a sequence of $L^2$-orthonormal eigenvectors of $d_{A_\flat}^+d_{A_\flat}^{+,*}$ on $\Omega^+(\ad P)$, then the assignment
\[
  \bar v_{j,k}(\lambda_{jk}) \mapsto v_{j,k}(\lambda_{jk}), \quad\forall\, (j,k) \in \NN\times\NN
\]
extends to an isometric isomorphism of the Hilbert space $L^2(\wedge^+(T^*X)\otimes\ad P)$ onto itself. Finally, the assignment
\[
  (v_{0j}, v_{1k}) \mapsto v_{j,k}(\lambda_{jk}), \quad\forall\, (j,k) \in \NN\times\NN
\]
extends to an isometric isomorphism  \eqref{eq:Isomorphism_direct_sum_SD2forms_P0_and_SD2forms_P1_with_SD2forms_P} of Hilbert spaces for all $\lambda\in(0,\lambda_0]$.
\end{thm}

\begin{proof}
The construction of the preceding isomorphisms and verification of their properties are obtained by extending proofs of related results in \cite{Feehan_yangmillsenergygap, FLKM1}, inspired in turn by constructions due to Taubes \cite{TauIndef, TauFrame}. For comparison of \emph{finite}-dimensional vector subspaces defined by the linear spans of \emph{finitely} many eigenvectors of the operators $d_{A_{i\flat}}^+d_{A_{i\flat}}^{+,*}$ for $i=0,1$ and $d_{A_\flat}^+d_{A_\flat}^{+,*}$, then a choice of \emph{one} small-enough scale parameter $\lambda$ is sufficient. However, in order to construct isomorphisms of the full infinite-dimensional Hilbert spaces, we must consider sequences of scales as indicated in the statement of the theorem.
\end{proof}  




We now turn to the

\begin{proof}[Proof of Theorem \ref{thm:Smooth_extension_composition_self-dual_curvature_splicing_maps}]
Our proof is similar to that of Theorem \ref{thm:Smoothness_splicing_map_manifold_with_boundary}. We analyze the boundedness and uniform continuity properties of the partial derivatives of $F^+\circ\,\cS$ on its domain in \eqref{eq:Composition_splicing_map_connections_self-dual_curvature_map} when $p=2$: 
\begin{equation}
\label{eq:Composition_splicing_map_connections_self-dual_curvature_map_partial_derivatives}
\frac{\partial(F^+\circ\,\cS)}{\partial A_0}, \quad \frac{\partial(F^+\circ\,\cS)}{\partial A_1}, \quad \frac{\partial(F^+\circ\,\cS)}{\partial\rho}, \quad \frac{\partial(F^+\circ\,\cS)}{\partial x_0}, \quad\text{and}\quad \frac{\partial(F^+\circ\,\cS)}{\partial\lambda}.
\end{equation}
As before, calculations of this kind were previously done by Feehan \cite[Theorem 3.33]{FeehanGeometry}  and Peng \cite{Peng_1995, Peng_1996}, with the analysis by Feehan in \cite[Section 3]{FeehanGeometry} again being closest to our current needs. Peng assumes initially that the Riemannian metric $g$ is flat near $x_{0\flat}$ and makes adjustments in \cite[Section 5]{Peng_1995} to allow for non-flat Riemannian metrics and makes some corrections and adjustments in \cite[Section 2]{Peng_1996} to his earlier paper \cite{Peng_1995}. Moreover, in \cite{Peng_1995, Peng_1996}, Peng only estimates the $L^{4/3}$ norms of the partial derivatives in \eqref{eq:Composition_splicing_map_connections_self-dual_curvature_map_partial_derivatives}, whereas in \cite[Section 3]{FeehanGeometry}, Feehan estimates their $L^p$ norms for any $p\in [1,\infty)$; we only need the case $p=2$.

\setcounter{step}{0}
\begin{step}[Derivatives with respect to the connections $A_0$ and $A_1$]
See the proof by Feehan of his \cite[Proposition 3.31]{FeehanGeometry} for details and see the proof by Peng of his estimate in \cite[Equation (4.48)]{Peng_1995} for related calculations. Suppose that $A_0(t_0)$ is a smooth curve in $\sA(P_0)$ parameterized by arclength $t_0\in(-\eps,\eps)$ for some $\eps\in(0,1]$. By \cite[Proposition 3.30]{FeehanGeometry}, we have
\[
\left\|\frac{\partial(F^+\circ\,\cS)}{\partial t_0}\right\|_{L^p(X)} \lesssim \lambda^{2/p-1/2}
\]
Aside from obvious notational changes, the calculations and conclusions for $\partial\cS/\partial A_1$ are identical. This completes our estimates of the partial derivatives of $F^+\circ\,\cS(A_0,A_1,\rho,x_0,\lambda)$ with respect to the connections $A_0$ and $A_1$.
\end{step}

\begin{step}[Derivative with respect to the gluing parameter $\rho$]
  See the proof by Feehan of his \cite[Proposition 3.28]{FeehanGeometry} for the derivation of the related estimate for $\partial\cS/\partial\rho$ and see the proof by Peng of his \cite[Lemma 4.1]{Peng_1995}. Suppose that $\rho(t)$ is a smooth curve in $\Gl_{x_{0\flat}}\cong G$ parameterized by arclength $t\in(-\eps,\eps)$ for some $\eps\in(0,1]$. By adapting the calculation by Peng in \cite[Lemma 4.1]{Peng_1995} for $p=4/3$ to $p\in [2,\infty)$, we obtain
\[
\left\|\frac{\partial(F^+\circ\,\cS)}{\partial t}\right\|_{L^{4/3}(X)} \sim |t|\lambda^{2/p+1/2}.
\]
This completes our estimates of the partial derivative of $F^+\circ\,\cS(A_0,A_1,\rho,x_0,\lambda)$ with respect to the gluing parameter $\rho$.
\end{step}

\begin{step}[Derivative with respect to the center point $x_0$]
See the proof by Feehan of his \cite[Proposition 3.26 (b)]{FeehanGeometry} for details and see Peng \cite[Section 4.4]{Peng_1995} for related calculations. Suppose that $x_0(t)$ is a smooth curve in $X$ parameterized by arclength $t\in(-\eps,\eps)$ for some $\eps\in(0,1]$. By \cite[Proposition 3.26 (b)]{FeehanGeometry}, we have
\[
\left\|\frac{\partial(F^+\circ\,\cS)}{\partial t}\right\|_{L^p)} \lesssim \lambda^{2/p-1/2}.
\]
This completes our estimates of the partial derivative of $F^+\circ\,\cS(A_0,A_1,\rho,x_0,\lambda)$ with respect to the center point $x_0$.  
\end{step}

\begin{step}[Derivative with respect to the scale $\lambda$]
See the proof by Feehan of his \cite[Proposition 3.26 (a)]{FeehanGeometry} for details and see the proof by Peng of his \cite[Displayed equation prior to Equation (4.48)]{Peng_1995}. By \cite[Proposition 3.26 (a)]{FeehanGeometry}, we have
\[
\left\|\frac{\partial(F^+\circ\,\cS)}{\partial\lambda}\right\|_{L^p)} \lesssim \lambda^{2/p-1}.
\]
This completes our estimate of the partial derivative of $F^+\circ\,\cS(A_0,A_1,\rho,x_0,\lambda)$ with respect to the scale parameter $\lambda$.    
\end{step}

By the preceding analysis, we conclude that the partial derivatives in \eqref{eq:Composition_splicing_map_connections_self-dual_curvature_map_partial_derivatives} are \emph{bounded} on the domain of $F^+\circ\,\cS$ in \eqref{eq:Composition_splicing_map_connections_self-dual_curvature_map} when $p=2$. By employing further calculations provided by Feehan in \cite[Sections 3.6--3.9]{FeehanGeometry} (see also Peng \cite{Peng_1995, Peng_1996} for similar calculations), one finds moreover that the partial derivatives in \eqref{eq:Composition_splicing_map_connections_self-dual_curvature_map_partial_derivatives} are \emph{uniformly continuous} on the domain of $F^+\circ\,\cS$ in \eqref{eq:Composition_splicing_map_connections_self-dual_curvature_map} when $p=2$ and, in particular, uniformly continuous with respect to $\lambda\in(0,\lambda_0)$. Hence, these partial derivatives have unique uniformly continuous extensions to the domain of $F^+\circ\,\cS$ in \eqref{eq:Composition_splicing_map_connections_self-dual_curvature_map_unsplicing_map_self-dual_two-forms}. Consequently, $F^+\circ\,\cS$ extends to a $C^1$ section of a $C^1$ vector bundle over a smooth Banach manifold with boundary, as claimed.
\end{proof}

\section{Completion of proof of  main gluing theorem}
\label{sec:Completion_proof_main_gluing_theorem}
We can give the relatively short

\begin{proof}[Proof of Theorem \ref{mainthm:Gluing}]
We proceed by verifying that the hypotheses of Theorem \ref{mainthm:Preimage_submanifold_under_transverse_map_and_implied_embedding} hold and hence produce the desired gluing map, where we choose
\begin{align*}
  \sX &= \sA^{1,2}(P_0)\times \sA^{1,2}(P_1) \times \Gl_{x_{0\flat}}\times B_\delta(x_{0\flat}) \times [0,\lambda_0),
  \\
  x_0 &= (A_{0\flat}, A_{1\flat},\rho_0,x_{0\flat},0)
  \\
  &\in \partial\sX = \sA^{1,2}(P_0)\times \sA^{1,2}(P_1) \times \Gl_{x_{0\flat}}\times B_\delta(x_{0\flat}) \times \{0\},
  \\
  \sX' &= L^2(\wedge^+(T^*X)\otimes\ad P_0) \times L^2(\wedge^+(T^*S^4)\otimes\ad P_1),
  \\
  \sX'' &= \{0\},
\end{align*}          
where $\rho_0 \in \Gl_{x_{0\flat}}$ is any fixed bundle gluing parameter, and as a candidate for the map $f:\sX\to\sX'$ we consider
\begin{multline*}
  \widehat{F}^+\circ\widehat{\cS}: \sA^{1,2}(P_0)\times \sA^{1,2}(P_1) \times \Gl_{x_{0\flat}}\times B_\delta(x_{0\flat}) \times [0,\lambda_0)
  \\
  \to L^2(\wedge^+(T^*X)\otimes\ad P_0) \times L^2(\wedge^+(T^*S^4)\otimes\ad P_1).
\end{multline*}
When we restrict to $\lambda \in (0,\lambda_0)$, we obtain the smooth map
\[
  F^+\circ\cS: \sA^{1,2}(P_0)\times \sA^{1,2}(P_1) \times \Gl_{x_{0\flat}}\times B_\delta(x_{0\flat}) \times (0,\lambda_0)
  \to L^2(\wedge^+(T^*X)\otimes\ad P)
\]
as the composition of $\widehat{F}^+\circ\widehat{\cS}$ and addition of pairs of sections in $L^2(\wedge^+(T^*X)\otimes\ad P_0) \times L^2(\wedge^+(T^*S^4)\otimes\ad P_1)$ over the annuli $\Omega_{x_0,\lambda} \subset X$.

We have already seen in Theorem \ref{thm:Smooth_extension_composition_self-dual_curvature_splicing_maps} that $\widehat{F}^+\circ\widehat{\cS}$ is a $C^1$ map of smooth Banach manifolds with boundary. Since $\sX'$ is without boundary and $\sX''$ is a point, the hypothesis on $\sX''$ being a neat submanifold of $\sX'$ is trivially obeyed. 

We now check that the remaining hypothesis of Theorem \ref{mainthm:Preimage_submanifold_under_transverse_map_and_implied_embedding} that $\widehat{F}^+\circ\widehat{\cS}$ is transverse to zero at the boundary point $(A_{0\flat}, A_{1\flat},\rho_0,x_{0\flat},0)$, in other words, that $\widehat{F}^+\circ\widehat{\cS}$ is a submersion at $(A_{0\flat}, A_{1\flat},\rho_0,x_{0\flat},0)$. By hypothesis of Theorem \ref{mainthm:Gluing}, $F^+(A_{0\flat})=0$ and $F^+(A_{1\flat})=0$ and moreover the smooth maps
\begin{align*}
  F^+:\sA^{1,2}(P_0) &\to L^2(\wedge^+(T^*X)\otimes\ad P_0) \quad\text{and}
  \\
  F^+:\sA^{1,2}(P_1) &\to  L^2(\wedge^+(T^*S^4)\otimes\ad P_1)
\end{align*}
vanish transversely at $A_{0\flat}$ and $A_{1\flat}$, respectively. Therefore, the boundary map
\begin{multline*}
  \partial(\widehat{F}^+\circ\widehat{\cS}): \sA^{1,2}(P_0)\times \sA^{1,2}(P_1) \times \Gl_{x_{0\flat}}\times B_\delta(x_{0\flat}) \times \{0\}
  \\
  \to L^2(\wedge^+(T^*X)\otimes\ad P_0) \times L^2(\wedge^+(T^*S^4)\otimes\ad P_1)
\end{multline*}
vanishes transversely at $(A_{0\flat}, A_{1\flat},\rho,x_0,0)$, for all $(\rho,x_0) \in \Gl_{x_{0\flat}}\times B_\delta(x_{0\flat})$.

We can thus apply Theorem \ref{mainthm:Preimage_submanifold_under_transverse_map_and_implied_embedding}, noting that $T_{f(x_0)}\sX'' = \{0\}$ and so $(df(x_0))^{-1}(T_{f(x_0)}\sX'') = \Ker df(x_0)$.
Observe that
\begin{multline*}
  \Ker d(\widehat{F}^+\circ\widehat{\cS})(A_{0\flat}, A_{1\flat},\rho_0,x_{0\flat},0)
  \\
  =
  \Ker d_{A_0}^+\cap W^{1,2}(T^*X\otimes\ad P_0) \oplus \Ker d_{A_1}^+\cap W^{1,2}(T^*S^4\otimes\ad P_1) \oplus \fg_{x_{0\flat}} \oplus T_{x_{0\flat}}X\oplus\RR,
\end{multline*}
where $\fg_{x_{0\flat}} = T_{\rho_0}\Gl_{x_{0\flat}}$ is isomorphic to the Lie algebra $\fg$ of $G$.
\end{proof}

\section{Non-regular boundary points}
\label{sec:Non-regular_boundary_points}
We now allow for the possibility that $H_{A_{0\flat}}^2(X;\ad P_0) \neq 0$ and prove Corollary \ref{maincor:Gluing_HA2_non-zero}. Following the paradigm in Section \ref{subsec:Local_Kuranishi_parameterization_neighborhood_interior_point_Donaldson_approach}, we extend our previous definition of the splicing map in \eqref{eq:Splicing_map_connections} to read
\begin{multline}
\label{eq:Splicing_map_connections_stabilized}
\cS: \sA_\delta^{1,p}(A_{0\flat})\times \bH_{A_{0\flat}}^2(X;\ad P_0) \times \sA_\delta^{1,p}(A_{1\flat}) \times \Gl_{x_{0\flat}}\times B_\delta(x_{0\flat}) \times (0,\lambda_0)
\\
\to \sA(P)\times \bH_{A_{0\flat}}^2(X;\ad P_0),
\end{multline}
where, if we relabel our previous definition \eqref{eq:Splicing_map_connections} of the splicing map as $\cS_0$, then
\[
  \cS(A_0,v,A_1,\rho,x_0,\lambda) = (\cS_0(A_0,A_1,\rho,x_0,\lambda),v),
\]
so $\cS$ restricts to the identity on the factor $\bH_{A_{0\flat}}^2(X;\ad P_0)$ of the domain. We stabilize the self-dual curvature map \eqref{eq:Self-dual_curvature_map} so that the pullback $(F^++L_{A_{0\flat}})\circ\cS$ will be a submersion on an open neighborhood of boundary points $(A_{0\flat},v,A_1,\rho,x_0,0)$ in the domain (corresponding to $\lambda=0$) by defining
\begin{equation}
  \label{eq:Self-dual_curvature_map_stabilized}
  F^++L_{A_{0\flat}}:\sA(P)\times \bH_{A_{0\flat}}^2(X;\ad P_0) \ni (A,v) \mapsto F_A^++v \in L^p(\wedge^+(T^*X)\otimes\ad P)
\end{equation}
We can now give the

\begin{proof}
Corollary \ref{maincor:Gluing_HA2_non-zero} follows from the proof of Theorem \ref{mainthm:Gluing} by replacing the role of $F^+$ in \eqref{eq:Self-dual_curvature_map} by that of $F^++L_{A_{0\flat}}$ in \eqref{eq:Self-dual_curvature_map_stabilized}.
\end{proof}

\section{Splicing map for gauge transformations}
\label{sec:Splicing_gauge_transformations}

\subsection{Based gauge transformations, slice theorems, and connections with non-trivial stabilizer}
Given a point $x_0 \in X$, consider the subgroup of $\Aut(P)$ defined by the set of \emph{based gauge transformations},
\[
  \Aut_0(P) := \left\{u \in \Aut(P): u \restriction P_{x_0} \text{ is the identity on } P_{x_0}\right\}.
\]
See Atiyah and Bott \cite[p. 605]{Atiyah_Bott_1983}, Cohen and Milgram \cite[p. 18]{Cohen_Milgram_1994}, Marathe \cite[p. 190]{Marathe_topics_physical_mathematics} for a discussion of topological and group-theoretic issues concerning $\Aut_0(P)$ and Groisser and Parker \cite[Section 1]{GroisserParkerGeometryDefinite} for a discussion of analytical aspects of $\Aut_0(P)$, including the proof of a slice theorem for the action of $\Aut_0(P)$ on $\sA(P)$. Recall from Marathe \cite[Proposition 6.6]{Marathe_topics_physical_mathematics} that $\Aut_0(P)$ is a normal subgroup of $\Aut(P)$ and that $\Aut(P)/\Aut_0(P) \cong G$.

For brevity, denote $\Gamma_A = \Stab(A) \subset G$ and $H_A = \Hol(A) \subset G$, as in \cite[p. 132]{DK} for a connected base manifold $X$. Then $\Gamma_A \cong C_G(H_A)$ by \cite[Lemma 4.2.8]{DK}, where $C_G(H_A) = \{g\in G: gh=hg, \forall\, h\in H_A\}$ is the centralizer of $H_A$ of $G$.

We first recall a well-known result concerning the existence of slice neighborhoods for a smooth action of a compact Lie group on a smooth manifold. For background, we refer to Warner \cite{Warner}. Let $G$ be a  Lie group with a smooth left action on a smooth manifold $W$ (so $W$ is a smooth $G$-manifold) and, for each point $x\in W$, let $G\cdot x := \{y \in W: y = g\cdot x \text{ for some } g \in G\} \subset W$ denote the orbit of $x$ in $W$ under the action of $G$ and $G_x := \{g\in G: g\cdot x=x\}$ denote isotropy subgroup (stabilizer) of $G$ defined by the point $x$. If $V$ is another smooth $G$-manifold and $\varphi:V\to W$ is a smooth map, then $\varphi$ is $G$-equivariant if $\varphi(g\cdot x) = g\cdot\varphi(x)$ for all $x\in V$ and $g\in G$. For $x\in W$, let
\[
  N_x := T_xW/T_x(G\cdot x)
\]
denote the \emph{normal space at $x$} with respect to the orbit $G\cdot x$ and $N := \cup_{g \in G}N_{g\cdot x}$ is the \emph{normal bundle} for the orbit $G\cdot x \subset W$. When $G$ is compact, we recall from Meinrenken \cite[Proposition 1.20]{Meinrenken_group_actions_manifolds_lecture_notes} (in the case of Lie groups) or Montgomery and Zippin \cite[Theorem 2.13]{Montgomery_Zippin_topological_transformation_groups} (in the case of topological groups) that for each $x\in W$, the orbit $G\cdot x$ is an closed, embedded submanifold of $W$. When $W$ is a Riemannian manifold and $G$ acts isometrically, then there is a canonical isomorphism \cite[p. 306]{Bredon},
\[
  N_x \cong T_x(G\cdot x)^\perp,
\]
where $T_x(G\cdot x)^\perp$ denotes the orthogonal complement of $T_x(G\cdot x)$ in $T_xW$. When $G$ is compact, such an invariant metric always exists \cite[Theorem 1.2]{Koszul_lectures_groups_transformations}. For each $g\in G_x$, the derivative $dg(x):T_xW \to T_{g\cdot x}W = T_xW$ of the diffeomorphism $g:W\to W$ defines an automorphism of the vector space $T_xW$. Moreover, the smooth map $g:G\cdot x \to G\cdot x$ is a diffeomorphism and its derivative $dg(x):T_x(G\cdot x) \to T_x(G\cdot x)$ defines an automorphism of the vector space $T_x(G\cdot x)$ and thus
\[
  dg(x) \in \Aut(N_x).
\]
Hence, the stabilizer $G_x$ acts on $G \times N_x$, by left multiplication on $G$ and by the preceding linear representation on $N_x$.

\begin{thm}[Existence of smooth slices for the action of a compact Lie group on a smooth manifold]
\label{thm:Existence_smooth_slices}  
(See Audin \cite[Theorem I.2.1]{Audin_torus_actions_symplectic_manifolds}, Bredon \cite[Theorem 6.2.2]{Bredon}, tom Dieck \cite[Section 1.5]{tomDieck}, Groisser and Parker \cite[p. 514]{GroisserParkerGeometryDefinite}, Kankaanrinta \cite[Theorem 4.4]{Kankaanrinta_2007}, and \cite{Koszul_1953}, Koszul \cite[Lemma 2.4 and Theorem 2.1]{Koszul_lectures_groups_transformations}.)
There exists a $G$-equivariant diffeomorphism $\varphi$ from a $G$-invariant open neighborhood of the zero section $G/G_x$ in $G \times_{G_x} N_x$ onto an open neighborhood of $G\cdot x$ in $W$, which sends the zero section onto the orbit $G\cdot x$ by the natural map $f_x:G \ni g \mapsto g\cdot x \in W$.
\end{thm}

When $W$ is Riemannian and the $G$-equivariant diffeomorphism $\varphi$ is the exponential map, Theorem \ref{thm:Existence_smooth_slices} can be used to describe the geometry of the quotient $W/G$ near the point $[x]$ and this is how Groisser and Parker use this slice result in the context of the action of $G$ on the smooth Banach manifold $\sB_0(P) = \sA(P)/\Aut(P)$ \cite[Section 2]{GroisserParkerGeometryDefinite}, by virtue of their slice result \cite[Theorem 1.1]{GroisserParkerGeometryDefinite} for the action of the Banach Lie group $\Aut_0(P)$ on the Banach affine space $\sA(P)$. In particular, if $G_{(A)}$ is the stabilizer of $(A) := \Aut_0(P)\cdot A$ in $G$, with respect to the smooth action $G\times \sB_0(P) \to \sB_0(P)$, and $\Stab(A)$ is the stabilizer of $A$ in $\Aut(P)$, with respect to the smooth action $\Aut(P)\times \sA(P) \to \sA(P)$, then
\[
  \Stab(A) \cong G_{(A)}
\]
is a canonical isomorphism. 

To apply Theorem \ref{thm:Existence_smooth_slices}, we choose $W = \sB_0(P)$ and $x = (A) \in \sB_0(P)$, so $G_x = \Stab(A)$ and $N_x = T_{(A)}\sB_0(P)/T_{(A)}(G\cdot(A)) = N_{(A)}$. Observe that the quotient $G/\Stab(A)$ may be viewed as the zero section of the vector bundle $G\times_{\Stab(A)} N_{(A)}$. Theorem \ref{thm:Existence_smooth_slices} yields a map
\begin{equation}
  \label{eq:Equivariant_tubular_neighborhood_smooth_embedding}
  \varphi:G\times_{\Stab(A)} N_{(A)} \to \sB_0(P)
\end{equation}
that is a $G$-equivariant diffeomorphism onto an open neighborhood of $G\cdot(A) \subset \sB_0(P)$ such that
\[
  \varphi(G/\Stab(A)) = G\cdot(A).
\]
The induced map obtained by taking quotients by $G$ yields a homeomorphism onto an open neighborhood of $[A] := \Aut(P)\cdot A\in\sB(P)$,
\begin{equation}
  \label{eq:Equivariant_tubular_neighborhood_induced_homeomorphism}
  \bar\varphi: N_{(A)}/\Stab(A) \to \sB(P),
\end{equation}
where $\sB_0(P)/G = \sB(P)$ and $[A] = G\cdot(A)$. Compare \cite[Proposition 4.2.29]{DK}, which asserts (in our notation) that there is a homeomorphism
\[
  \phi:\sN_{(A)}/\Stab(A) \to \sB(P)
\]
onto an open neighborhood of $[A]$ in $\sB(P)$, where
\[
  \sN_A := T_A\sA(P)/T_A(\Aut(P)\cdot A) = T_A(\Aut(P)\cdot A)^\perp = \Ker d_A^* \cap E(P)
\]
and $\sA(P)=A+E(P)$ and we abbreviate $E(P) := W^{1,p}(X;T^*X\otimes\ad P)$.

\subsection{Slice theorem for the action of based gauge transformations on the affine space of Sobolev connections}
As Groisser and Parker note, slice theorems for the action of (Sobolev completions of) $\Aut(P)$ on $\sA(P)$ are standard and can be found in \cite[Proposition 4.2.9 and discussion, pp. 132--133]{DK}, \cite[Theorem 3.2 and Corollary, p. 50]{FU} for the case of $W^{2,k+1}$ gauge transformations acting on $W^{2,k}$ connections over a four-dimensional base manifold $X$ for integers $k\geq 2$. The first author and Maridakis prove sharper versions of these slice theorems in \cite[Section 1.5]{Feehan_Maridakis_Lojasiewicz-Simon_coupled_Yang-Mills} that allow for $W^{2,p}$ gauge transformations acting on $W^{1,p}$ connections over a base manifold $X$ of arbitrary dimension $d\geq 2$ and constant $p\in(d/2,\infty)$.

However, when analyzing neighborhoods of points $[A]\in\sB(P)$ defined by connections $A$ with non-trivial stabilizer in $\Aut(P)$, it is convenient to first consider the quotient of $\sA(P)$ by $\sB_0(P)$ and then consider $\sB(P)$ as the quotient of $\sB_0(P)$ by the finite-dimensional Lie group $G$, as in Austin and Braam \cite{Austin_Braam_1995, Austin_Braam_1996} or Groisser and Parker \cite{GroisserParkerGeometryDefinite}. For this purpose, however, one needs a slice theorem for the action of $\sB_0(P)$ on $\sA(P)$. Groisser and Parker prove the following result for $W^{s+1,2}$ gauge transformations acting on $W^{s,2}$  connections when $X$ has dimension four and fractional Sobolev exponent $s>1$. A combination of their methods and those of the first author and Maridakis \cite{Feehan_Maridakis_Lojasiewicz-Simon_coupled_Yang-Mills} yields the following $L^p$ analogues of \cite[Theorem 1.1]{GroisserParkerGeometryDefinite} due to Groisser and Parker and \cite[Theorem 14 and Corollary 18]{Feehan_Maridakis_Lojasiewicz-Simon_coupled_Yang-Mills} due to the first author and Maridakis. A result similar to a combination of Theorem \ref{thm:Feehan_Maridakis_14_based} and Corollary \ref{cor:Feehan_Maridakis_18_based} is proved by Wilkins \cite[Theorem 7.2]{Wilkins_1989}.

\begin{thm}[Existence of $W^{2,p}$ based Coulomb gauge transformations for $W^{1,p}$ connections]
\label{thm:Feehan_Maridakis_14_based}
Let $(X,g)$ be a closed, smooth Riemannian manifold of dimension $d \geq 2$, and $G$ be a compact Lie group, $P$ be a smooth principal $G$-bundle over $X$, and $x_0\in X$ be a point. If $A_\flat$ is a $C^\infty$ connection on $P$, and $A_0$ is a $W^{1,p}$ connection on $P$ with $p \in (d/2,\infty)$, then there exists a constant $\zeta = \zeta(A_0,A_\flat,g,G,p) \in (0,1]$ with the following significance. If $A$ is a $W^{1,p}$ connection on $P$ that obeys
\begin{equation}
\label{eq:Feehan_Maridakis_14_based_A_minus_Ab_W1p_close}
\|A - A_0\|_{W_{A_\flat}^{1,p}(X)} < \zeta,
\end{equation}
then there exists a $W^{2,p}$ based gauge transformation $u \in \Aut_0(P)$ such that
\[
u(A) - A_0 \perp \Ran\left(d_{A_0}:T_{\id}\Aut_0(P) \to W^{1,p}(T^*X\otimes\ad P)\right),
\]
where $\perp$ denotes $L^2$-orthogonal, $T_{\id}\Aut_0(P) = \{\xi\in W^{2,p}(\ad P): \xi(x_0)=0\}$, and
\[
\|u(A) - A_0\|_{W_{A_\flat}^{1,p}(X)} \leq C\|A - A_0\|_{W_{A_\flat}^{1,p}(X)},
\]
where $C = C(A_0,A_\flat,g,G,p) \in [1,\infty)$ is a constant.
\end{thm}

\begin{cor}[Real analytic Banach manifold structure on the based quotient space of $W^{1,p}$ connections]
\label{cor:Feehan_Maridakis_18_based}  
Let $(X,g)$ be a closed, smooth Riemannian manifold of dimension $d \geq 2$, and $G$ be a compact Lie group, $P$ be a smooth principal $G$-bundle over $X$, and $x_0\in X$ be a point, and $p\in(d/2,\infty)$ be a constant. If $A_\flat$ is a $C^\infty$ connection on $P$ and $[A] \in \sB_0(P)$, then there is a constant $\eps = \eps(A_\flat,[A],g,G,p) \in (0,1]$ with the following significance. If
\[
\bB_A(\eps) := \left\{a \in W_{A_\flat}^{1,p}(X;T^*X\otimes\ad P): d_A^*a = 0 \text{ and } \|a\|_{W_{A_\flat}^{1,p}(X)} < \eps\right\},
\]
then the map,
\[
\pi_A: \bB_A(\eps) \ni [a] \mapsto [A+a] \in \sB_0(P),
\]
is a homeomorphism onto an open neighborhood of $[A]\in \sB_0(P)$. In particular, the inverse coordinate charts $\pi_A$ determine real analytic transition functions for $\sB_0(P)$, giving it the structure of a real analytic Banach manifold.
\end{cor}

\subsection{Splicing based gauge transformations}
We construct a splicing map $\fS$ for based gauge transformations that is analogous to our previously constructed splicing maps $\cS$ for connections in \eqref{eq:Splicing_map_connections} and $\sS$ for bundle-valued self-dual two-forms in \eqref{eq:Splicing_map_bundle-valued_self-dual_two-forms}:
\begin{equation}
\label{eq:Splicing_map_based_gauge_transformations}
\fS: \Aut(P_0)\times \Aut_0(P_1) \ni (u_0,u_1) \mapsto u \in \Aut(P).
\end{equation}
Indeed, because $u_0(x)$ is close to $\id_x \in \Aut(P_0|_x)$ when $x$ is close to $x_0\in X$ (and equal to $\id_{x_0}$ when $x=x_0$), we may write
\[
  p_0\cdot u_0(x) = p_0\cdot\exp_G(\xi_0(x)), \quad\forall\, x \in B_{\delta(u_0)}(x_0) \text{ and } p_0 \in P_0|_x, 
\]
where $\exp_G:\fg \to G$ is the exponential map for the Lie group $G$ and $\fg$ is the Lie algebra of $G$ and $\xi_0:B_{\delta_0}(x_0)\to\fg$ is a smooth map and $\delta_0=\delta_0(u_0)\in(0,1]$. Similarly, because $u_1(x)$ is close to $\id_x \in \Aut(P_1|_x)$ when $x$ is close to $s\in S^4$ (and equal to $\id_s$ when $x=s$), we may also write
\[
  p_1\cdot u_1(x) = p_1\cdot\exp_G(\xi_1(x)), \quad\forall\, x \in B_{\delta_1}(s) \text{ and } p_1 \in P_1|_x, 
\]
and $\xi_1:B_{\delta_1}(s)\to\fg$ is a smooth map and $\delta_1=\delta_1(u_1)\in(0,1]$. We can now construct $u \in \Aut_0(P)$ by setting
\[
  u :=
  \begin{cases}
    u_0 &\text{on }X\less B_{\delta_0}(x_0),
    \\
    \exp_G(\chi_0\xi_0 + \chi_1\xi_1) &\text{on } \Omega(x_0;\delta_0, \delta_1),
    \\
    u_1 &\text{on }S^4\less B_{\delta_1}(s).
  \end{cases}
\]
Conversely, every $u\in \Aut_0(P)$ arises in this way, since such gauge transformations are close to the identity map on fibers $P_x$, for all $x$ in an small open annulus $\Omega \subset X$ that is a neighborhood of the copy of $S^3$ joining $X$ and $S^4$. Because
\[
  u(x) = \exp_G(\xi(x)), \quad\forall\, x \in \Omega,
\]
we may define $u_0 = \exp_G(\psi_0\xi)$ near $x_0\in X$ and $u_1 = \exp_G(\psi_1\xi)$ near $s \in S^4$, while $u_0 = u$ on $X \less B_{\delta_0}(x_0)$ and $u_1 = u$ on $S^4 \less B_{\delta_1}(s)$. Therefore, the splicing map $\fS$ in \eqref{eq:Splicing_map_based_gauge_transformations} has a smooth right inverse
\begin{equation}
\label{eq:Unsplicing_map_based_gauge_transformations}
\fU: \Aut_0(P) \ni u \mapsto (u_0,u_1) \in \Aut_0(P_0)\times \Aut_0(P_1)
\end{equation}
such that
\[
  \fS\circ\fU = \id \quad\text{on } \Aut_0(P).
\]

\subsection{Gauge equivariance}
We now observe that the splicing map $\cS$ for connections in \eqref{eq:Splicing_map_connections} is equivariant with respect to the map $\fS$ and the action of $\Aut_0(P_0)\times \Aut_0(P_1)$ on the domain of $\cS$ and action of $\Aut_0(P)$ on the codomain of $\cS$:
\[
  \cS(u_0(A_0),u_1(A_1),\rho,x_0,\lambda) = u(\cS(A_0,A_1,\rho,x_0,\lambda)),
\]
where $u = \fS(u_0,u_1)$. The map $\cS$ in \eqref{eq:Splicing_map_connections} thus descends to a $G$-equivariant submersion on the quotient spaces:
\begin{equation}
\label{eq:Splicing_map_connections_mod_based_gauge_transformations}
\cS: \sB_0(P_0)\times \sB_0(P_1) \times \Gl_{x_{0\flat}}\times B_\delta(x_{0\flat}) \times (0,\lambda_0) \to \sB_0(P).
\end{equation}
The map \eqref{eq:Splicing_map_connections_mod_based_gauge_transformations} can be composed with $G$-equivariant smooth embeddings of the form \eqref{eq:Equivariant_tubular_neighborhood_smooth_embedding} that define equivariant tubular neighborhoods of orbits $G\cdot(A_0) \subset \sB_0(P_0)$ and $G\cdot(A_1) \subset \sB_0(P_1)$:
\[
  \varphi_0:G\times_{\Stab(A_0)} N_{(A_0)} \to \sB_0(P_0)
  \quad\text{and}\quad
  \varphi_1:G\times_{\Stab(A_1)} N_{(A_1)} \to \sB_0(P_1),
\]
where
\[
  N_{(A_0)} := T_{(A_0)}\sB_0(P_0)/T_{(A_0)}(G\cdot(A_0))
  \quad\text{and}\quad
  N_{(A_1)} := T_{(A_1)}\sB_0(P_1)/T_{(A_1)}(G\cdot(A_1)).
\]
The resulting composition is a $G$-equivariant submersion onto an open subset of $\sB_0(P)$:
\[
  \cS\circ(\varphi_0\times\varphi_1): G\times_{\Stab(A_0)} N_{(A_0)} \times G\times_{\Stab(A_1)} N_{(A_1)} \times \Gl_{x_{0\flat}}\times B_\delta(x_{0\flat}) \times (0,\lambda_0) \to \sB_0(P).
\]
By analogy with the definition of the topological embedding \eqref{eq:Equivariant_tubular_neighborhood_induced_homeomorphism}, the preceding $G$-equivariant submersion descends to a topological submersion \cite[p. 27]{Lang_fundamentals_differential_geometry} onto an open subset of $\sB(P)$:
\[
  \cS\circ(\bar\varphi_0\times\bar\varphi_1):N_{(A_0)}/\Stab(A_0) \times N_{(A_1)}/\Stab(A_1) \times \Gl_{x_{0\flat}}\times B_\delta(x_{0\flat}) \times (0,\lambda_0) \to \sB(P).
\]

\section{Boundary points with non-trivial isotropy groups}
\label{sec:Boundary_points_with_non-trivial_isotropy_groups}
In this section, we complete the

\begin{proof}[Proof of Corollary \ref{maincor:Gluing_HA2_and_HA0_non-zero}]
For the sake of clarity, we first consider the simpler case where $H_{A_{0\flat}}^2(X;\ad P)$ is zero but allow the connections $A_{i\flat}$ to have non-trivial isotropy groups in $\Aut(P_i)$ for $i=0,1$. Recall that our splicing map $\cS$ and the composition $F^+\circ\cS$ of the splicing and curvature maps are equivariant with respect to the action of the Banach Lie groups
\[
  \Aut_0(P_0) \times \Aut_0(P_1) \lhd \Aut(P_0) \times \Aut(P_1)
\]
on the domain and the Banach Lie group $\Aut(P)$ on the codomain, with the action of $(u_0,u_1)$ on the domain implying the action of $u=\fS(u_0,u_1)$ on the codomain.
  
We must address the complication that when we consider the quotient of affine spaces of $W^{1,p}$ connections by the Banach Lie groups of $W^{2,p}$ gauge transformations, we must choose $p>2$, whereas the derivatives of splicing maps and derivatives of compositions of splicing and curvature maps only extend continuously from domains involving $\lambda\in(0,\lambda_0)$ to $\lambda=0$ when $p\leq 2$. For this purpose, we first restrict $\cS$ and $F^+\circ\cS$ to Coulomb-gauge slices in $\sA(P_0)$ and $\sA(P_1)$ provided by the Groisser--Parker Slice Theorem \ref{thm:Feehan_Maridakis_14_based} for based gauge transformations:
\begin{multline*}
  \cS: \left(A_{0\flat}+\Ker d_{A_{0\flat}}^*\cap W^{1,p}(T^*X\otimes\ad P_0) \right)
  \\
  \times \left(A_{1\flat}+\Ker d_{A_{1\flat}}^*\cap W^{1,p}(T^*S^4\otimes\ad P_1) \right)\cap\sA^\diamond(P_1)
  \\
  \times \Gl_{x_{0\flat}}\times B_\delta(x_{0\flat}) \times (0,\lambda_0) \to \sA(P)
\end{multline*}
where we choose $A_{1\flat}$ be a centered smooth anti-self-dual connection on $P_1$, and
\begin{multline*}
  F^+\circ\cS: \left(A_{0\flat}+\Ker d_{A_{0\flat}}^*\cap W^{1,p}(T^*X\otimes\ad P_0) \right)
  \\
  \times \left(A_{1\flat}+\Ker d_{A_{1\flat}}^*\cap W^{1,p}(T^*S^4\otimes\ad P_1) \right)\cap\sA^\diamond(P_1)
  \\
  \times \Gl_{x_{0\flat}}\times B_\delta(x_{0\flat}) \times (0,\lambda_0) \to L^p(\wedge^+(T^*X)\otimes\ad P).
\end{multline*}
These maps are $G$-equivariant. We can now apply our proof of Theorem \ref{mainthm:Gluing} \mutatis to yield Corollary \ref{maincor:Gluing_HA2_and_HA0_non-zero} in the special case where $H_{A_{0\flat}}^2(X;\ad P)$ is zero.

Lastly, we can extend our argument to allow for non-zero $H_{A_{0\flat}}^2(X;\ad P)$ in almost exactly the same way as we did in our proof of Corollary \ref{maincor:Gluing_HA2_non-zero}.
\end{proof}

\section{Riemannian metrics that are not locally conformally flat}
\label{sec:Riemannian_metric_not_locally_flat}
We have assumed that the Riemannian metric $g$ on the four-dimensional manifold $X$ is conformally flat near the point $x_{0\flat} \in X$, in other words, that assumption \eqref{eq:Riemg_zero} in Theorem \ref{mainthm:Gluing} holds. This assumption ensures that exponential map $\exp_v:T_{x_{0\flat}}X \supset B_\varrho(x_{0\flat}) \to X$ is a smooth inverse coordinate chart that is isometric, leading to simplifications implicit in our calculations thus far:
\begin{itemize}
\item The Riemannian metric $g_{v,\lambda}$ on the connected sum $X\# S^4$ defined by an oriented, orthonormal frame $v$ for $T_{x_0}X$ (obtained by parallel transport of an oriented, orthonormal frame $v_0$ for $T_{x_{0\flat}}X$ with respect to the Levi--Civita connection on $TX$ along the geodesic curve joining $x_{0\flat}$ to $x_0$) and scale $\lambda \in (0,\lambda_0]$ is conformally equivalent to $g$ on $X$ and (when near $x_0$) to the standard round metric $g_\round$ of radius one on $S^4$.

\item The calculation and estimation of derivatives of the map $F^+\circ\cS$ with respect to the center and scale parameters $(x_0,\lambda) \in B_\delta(x_{0\flat}) \times [0,\lambda_0)$ is as simple as possible. 
\end{itemize}

When $g$ is not flat near $x_{0\flat}$, all of the preceding calculations can still be carried out, but they become slightly more involved. Fortunately, calculations of this kind were performed by the first author in \cite[Section 3]{FeehanGeometry} and by Peng in \cite[Section 5]{Peng_1995}, with some additional corrections and adjustments to \cite{Peng_1995} in \cite[Section 2]{Peng_1996}. We summarize the main changes below:
\begin{itemize}
\item In order to ensure that the Riemannian metric $g_{v,\lambda}$ on the connected sum $X\# S^4$ remains conformally equivalent to the fixed Riemannian metric $g$ as the parameters $(x_0,\lambda)$ vary, we must allow for smooth Riemannian metrics on $S^4$ that are $C^1$ close to $g_\round$ but only become equal to $g_\round$ when $\lambda=0$. Connected sum metrics with this property were constructed by the first author in \cite[Section 3.5]{FeehanGeometry} and \cite[Section 6]{Feehan_yangmillsenergy_lojasiewicz4d}.

\item The calculation of and estimation of derivatives of the map $F^+\circ\cS$ with respect to the center and scale parameters $(x_0,\lambda) \in B_\delta(x_{0\flat}) \times [0,\lambda_0)$ must be modified following the calculations and estimates due to the first author in \cite[Section 3.5]{FeehanGeometry} and Peng in \cite[Section 5]{Peng_1995} and \cite[Section 2]{Peng_1996}. While Groisser and Parker exclusively use Donaldson's parameterization \cite{DonApplic} of the collar neighborhood in $M(P,g)$ (in the special case that $G=\SU(2)$, $c_2(P)=1$, $\pi_1(X)=\{1\}$, and $b^+(X)=0$) in their articles \cite{Groisser_1993, Groisser_1998, GroisserParkerGeometryDefinite}, their methods are also relevant here since their calculations rely heavily on estimates for derivatives of the $L^2$ metric on $M(P,g)$ with respect to the center and scale parameters.
\end{itemize}

Given the preceding comments, we can conclude the

\begin{proof}[Proof of Corollary \ref{maincor:Gluing_HA2_and_HA0_non-zero_and_non-flat_metric}]
The required adjustments to the proofs of Theorem \ref{mainthm:Gluing} and Corollaries \ref{maincor:Gluing_smooth}, \ref{maincor:Gluing_HA2_and_HA0_non-zero}, and \ref{maincor:Gluing_HA2_and_HA0_non-zero} follow from (and are considerably simpler than) the calculations in \cite[Section 3.5]{FeehanGeometry}, \cite[Section 5]{Peng_1995}, and \cite[Section 2]{Peng_1996}.
\end{proof}

%
%

\bibliography{/Users/pfeehan/Dropbox/LATEX/Bibinputs/master,/Users/pfeehan/Dropbox/LATEX/Bibinputs/mfpde}
\bibliographystyle{/Users/pfeehan/Dropbox/LATEX/Texinputs/amsplain-nodash}

\end{document}